\documentclass[a4paper,11pt,reqno]{amsart}
\usepackage{amssymb}
\usepackage{amsthm}
\usepackage{amsmath,latexsym}
\usepackage[T1]{fontenc}
\usepackage[cp1250]{inputenc}
\usepackage{enumerate}
\usepackage{graphicx}
\usepackage[all]{xy}

\usepackage{tikz}
\usetikzlibrary{shapes} 

\newtheorem{theorem}{Theorem}[section]
\newtheorem{prop}[theorem]{Proposition}
\newtheorem{lem}[theorem]{Lemma}
\newtheorem{coro}[theorem]{Corollary}
\theoremstyle{remark}
\newtheorem{rem}[theorem]{Remark} 
\theoremstyle{remark}
\newtheorem{df}[theorem]{\textsc{Definition}}
\theoremstyle{remark}

\theoremstyle{remark}
\newtheorem{exmp}[theorem]{Example}

\theoremstyle{plain}
\newtheorem*{theorem1}{Theorem 1} 
\newtheorem*{theorem2}{Theorem 2}

\renewcommand{\mod}{\operatorname{mod}}

\newcommand{\add}{\operatorname{add}}
\newcommand{\Hom}{\operatorname{Hom}}
\newcommand{\End}{\operatorname{End}}

\newcommand{\op}{\operatorname{op}}

\begin{document} 

\baselineskip=16pt
\parindent0pt

\title[Generalized WSA]{Generalized weighted surface algebras}

\author[A. Skowro\'nski]{Andrzej Skowro\'nski}
\address[Andrzej Skowro\'nski]{Faculty of Mathematics and Computer Science, 
Nicolaus Copernicus University, Chopina 12/18, 87-100 Torun, Poland}
\email{skowron@mat.umk.pl} 

\author[A. Skowyrski]{Adam Skowyrski}
\address[Adam Skowyrski]{Faculty of Mathematics and Computer Science, 
Nicolaus Copernicus University, Chopina 12/18, 87-100 Torun, Poland}
\email{skowyr@mat.umk.pl} 


\medskip

\begin{abstract} The weighted triangulation (surface) algebras associated to triangulation quivers (triangulated surfaces) 
and their socle deformations were recently introduced and studied in \cite{ES2}-\cite{ES9} and \cite{BEHSY}. These algebras, 
based on surface triangulations and originated from the theory of cluster algebras, were also proved to be finite-dimensional 
tame symmetric and periodic, of period $4$ (with some minor exceptions). In this paper, we introduce a new concept of a 
generalized triangulation quiver, extending existing notion of a triangulation quiver. In particular, it is shown that the 
generalized triangulation quivers can be reconstructed from triangulations of orientable surfaces with marked self-folded 
triangles. Moreover, motivated by the results of \cite{HSS}, we define and investigate so called weighted generalized 
triangulation algebras associated to generalized triangulation quivers, which naturally arise as mutations of weighted 
triangulation algebras. This gives new important class of tame symmetric periodic algebras of period $4$, essentially 
extending the class of weighted triangulation (surface) algebras, which justifies the name we chose. \end{abstract} 

\maketitle 

\section{Introduction and the main results}\label{sec:0} 

Throughout the article, we assume that $K$ is an algebraically closed field. An algebra is always finite-dimensional 
associative $K$-algebra with identity, which is assumed to be basic and connected. Given an algebra $A$, we denote by 
$\mod A$ the category of finite-dimensional (right) $A$-modules. Recall that an algebra $A$ is \emph{self-injective}, 
if all projective modules in $\mod A$ are injective. Important class of self-injective algebras is provided by the 
\emph{symmetric} algebras, for which we have a nondegenerate symmetric $K$-bilinear form $(-,-):A\times A\to K$. 
There exist many well-known classical examples of symmetric algebras, such as blocks of (finite-dimensional) group 
algebras \cite{E} or Hecke algebras associated to finite Coxeter groups \cite{Ar}. Furthermore, any algebra $A$ is 
a quotient of a symmetric algebra $T(A)=A\ltimes D(A)$ called a trivial extension of $A$, where $D$ denotes the 
standard duality $\Hom_K(-,K)$ on $\mod A$. 

For a module $M$ in $\mod A$, we denote by $\Omega_A(M)$ its \emph{syzygy}, that is the kernel of a minimal projective 
cover of $M$ in $\mod A$. A module $M$ in $\mod A$ is called \emph{periodic} if $M\simeq \Omega_A^n(M)$, for some 
$n\geqslant 1$, and the minimal such a number is called the \emph{period} of $M$. A prominent class of self-injective 
algebras consists of \emph{periodic} algebras $A$, for which $A$ is a periodic module in $\mod A^e$, where $A^e=A\otimes_K A$ 
is the enveloping algebra of $A$ (this is equivalent to say that $A$ is periodic as an $A$-$A$-bimodule). Every periodic 
algebra $A$ has periodic module category \cite[see Theorem IV.11.19]{SY}, that is all (nonprojective) modules in $\mod A$ 
are periodic (with period dividing the period of $A$), and moreover, its Hochschild cohomology is also periodic. Periodic 
algebras appear in many places, revealing its connections with group theory, topology, singularity theory, cluster algebras 
and algebraic combinatorics (see the survey \cite{ES1}). 

A general aim we are concerned with is to classify all tame symmetric periodic algebras. In \cite{Du1} Dugas proved that 
all representation-finite self-injective algebras without simple blocks are periodic algebras. The classification of all 
representation-infinite periodic algebras of polynomial growth was established in \cite{BES}; see also \cite{S1}. It is 
conjectured \cite[Problem]{ES2} that all tame symmetric periodic algebras of non-polynomial growth have period equal four. 

Recently, motivated by cluster theory, Erdmann and Skowro\'nski defined large class of symmetric periodic algebras of 
period $4$ associated to triangulations of real compact surfaces, called \emph{weighted surface algebras} (see 
\cite{ES2,ES5,ES9}). These are algebras of the form $\Lambda=\Lambda(S,\overrightarrow{\mathcal{T}},m_\bullet,c_\bullet,b_\bullet)$ 
associated to a (directed) triangulated surface $(S,\overrightarrow{\mathcal{T}})$ and some additional data encoded 
in functions $m_\bullet$, $c_\bullet$, and $b_\bullet$. More specifically, starting from $(S,\overrightarrow{\mathcal{T}})$ 
they first construct so-called triangulation quiver $(Q,f)$ uniuely determined by $(S,\overrightarrow{\mathcal{T}})$, 
where $Q$ is a $2$-regular quiver and $f$ is a permutation of arrows in $Q$ (of order $3$), and then define $\Lambda$ 
as a quotient $\Lambda=KQ/I$ of the path algebra $KQ$ of $Q$ by an ideal $I$ depending on $(Q,f,m_\bullet,c_\bullet,b_\bullet)$. 
The structure of $\Lambda$ relies only on $(Q,f,m_\bullet,c_\bullet,b_\bullet)$, so we mostly prefer to use the notation 
$\Lambda(Q,f,m_\bullet,c_\bullet,b_\bullet)$ for $\Lambda$, instead of 
$\Lambda(S,\overrightarrow{\mathcal{T}},m_\bullet,c_\bullet,b_\bullet)$, and refer to as a \emph{weighted triangulation 
algebra} -- this term was already used in the first paper \cite{ES2}. Precise definitions can be found in Section \ref{sec:4}. 
In this paper, we mainly focus on the combinatorics of underlying (generalized) triangulation quivers, and the context 
of surfaces appears only in the last short section. 

By the results of \cite{ES2,ES5}, the weighted triangulation algebras are tame symmetric periodic algebras of period $4$, 
with some minor exceptions. More than that, we mention results of Erdmann and Skowro\'nski from \cite{ES4}, where it is proved 
that an algebra $A$ with $2$-regular Gabriel quiver (having at least three vertices) is a tame symmetric periodic algebra of 
period $4$, or in general -- so called algebra of generalized quaternion type -- if and only if $A$ is a weighted surface 
algebra $\Lambda(S,\overrightarrow{\mathcal{T}},m_\bullet,c_\bullet,b_\bullet)$ different from a singular tetrahedral algebra, 
or it is the higher tetrahedral algebra $\Lambda(m,\lambda)$, $m\geqslant 2$, $\lambda\in K^*$. For details we refer to 
\cite{ES4}; see also \cite{ES3}. Note also that there are many relevant examples of wild symmetric periodic algebras of 
period $4$ realized as stable endomorphism rings of cluster-tilting Cohen-Macaulay modules over one-dimensional hypersurface 
singularities \cite{BIKR}; see also \cite[Corollary 2]{ES4}. 

Triangulations of surfaces provide a very useful tool to study algebraic structures. For instance, were applied to study 
cluster algebraic structures in Teichm\"uller theory \cite{FG,GSV}, cluster algebras of topological origin \cite{FST2}, 
as well as cluster algebras of finite mutation type with skew symmetric exchange matrices \cite{FST1}. Mutations introduced 
by Fomin and Zelevinsky (see \cite{DWZ1,DWZ2}), closely related to flips of triangulations, were originally applied to 
quivers without loops and $2$-cycles (or corresponding skew-symmetric integer matrices) and the triangulation quivers 
associated to triangulated surfaces were of slightly different nature from the ones we are dealing with. Here both loops 
and $2$-cycles are allowed, however, we may adapt the notion of mutation to our case. Namely, given a weighted triangulation 
algebra $\Lambda=\Lambda(Q,f,m_\bullet,c_\bullet,b_\bullet)$ we define its \emph{mutation}, which is an algebra $\Lambda '$ 
of the form $\Lambda ':=\End_{K^b(P_\Lambda)}(T)$, where $T$ is a tilting complex in the homotopy category $K^b(P_\Lambda)$ 
of bounded complexes of projective $\Lambda$-modules determined by a certain left approximation of a projective module 
(see Section \ref{sec:1} for details). We only mention that this construction is a special case of a general approach 
presented in \cite{Du2}. The authors together with Holm have recently studied a large group of examples of mutations 
$\Lambda '$, for which $T$ was coming from a left approximation of an indecomposable projective $\Lambda$-module at a 
vertex of $Q$. We observed in many cases that $\Lambda '$ is of the form $\Lambda '=KQ'/I'$, where $Q'$ is obtained from 
$Q$ by an operation which resembles mutation in the classical sense (it reverses arrows at a fixed vertex and adds arrows 
corresponding to paths passing through the vertex). Inspired by this observation we introduced and investigated \cite{HSS} 
a wide class of tame symmetric periodic algebras of period $4$, called the \emph{virtual mutations of weighted surface 
(or triangulation) algebras}, which are derived equivalent but not isomorphic to weighted triangulation algebras. Actually, 
these algebras are obtained from weighted triangulation algebras by a sequence of mutations at special vertices (see also 
\cite{ES6,Sk}). Moreover, we discovered that the mutations of weighted triangulation algebras (in all considered examples) 
are algebras given by a quiver which is glueing of a finite number of the following five types of \emph{blocks} 
$$
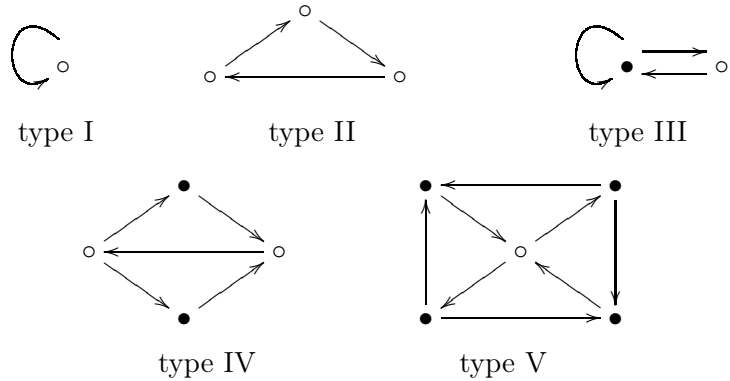
\xymatrix@R=0.01cm{&\\&\mbox{ } \\ &\ar@(lu,ld)[d]_{\mbox{ }}\\& \circ }\qquad\qquad
\xymatrix@R=0.5cm{&\circ\ar[rd] & \\ \circ\ar[ru] & &\circ\ar[ll]}\qquad\qquad
\xymatrix@R=0.01cm{&&\\&&\\&\ar@(lu,ld)[d]_{\mbox{ }}&\\& \bullet\ar@<0.2cm>[r]  &\circ \ar@<0.1cm>[l]}$$ 
$$\mbox{type I}\qquad\qquad\quad\quad\mbox{type II}\qquad\qquad\qquad\qquad\mbox{type III}$$
$$\xymatrix@R=0.5cm{&\bullet\ar[rd] & \\ \circ\ar[ru]\ar[rd] & &\circ\ar[ll] \\  & \bullet\ar[ru] & }\qquad\qquad 
\xymatrix@R=0.5cm{\bullet\ar[rd] & & \bullet\ar[ll]\ar[dd] \\ &\circ \ar[ld]\ar[ru] &  \\ 
\bullet\ar[rr]\ar[uu] & & \bullet\ar[lu] }$$ 
$$\mbox{type IV}\qquad\qquad\quad\qquad\mbox{type V}$$
where by gluing we mean that every white vertex ($\circ$) in any given block is glued with exactly one white vertex in a 
different block (precise definition is given in Section \ref{sec:2}). We will call such a quiver \emph{block decomposable}, 
although this notion was defined in a different way in \cite{FST2}, where block decomposable quivers were in one-to-one 
correspondence with the adjacency matrices of arcs of ideal (tagged) triangulations of bordered two-dimensional surfaces 
with marked points (see \cite[Section 13]{FST2}). We only note that block decomposable quivers in the sense of \cite{FST2} 
appeared also in \cite{FST1} and \cite{GLFS} in the context of cluster algebras and Jacobian algebras of quivers with potentials. 
 
If $Q$ is a block decomposable quiver and $*$ denotes a marking of one of (two) triangles in each block of type IV and V, 
then the pair $(Q,*)$ is called the \emph{generalized triangulation quiver}. For a generalized triangulation quiver $(Q,*)$ 
there is the associated quiver $Q^*$ with natural involution and a permutation of arrows $f$, which gives a triangulation-like 
structure on $Q^*$. As in case of the weighted triangulation algebras, we consider functions $m_\bullet$, $c_\bullet$, and 
$b_\bullet$ (satisfying some technical conditions), which are defined in terms of $(Q^*,f)$. With this setting we define 
the \emph{weighted generalized triangulation algebra} $\Lambda(Q,*,m_\bullet,c_\bullet,b_\bullet)=KQ/I$, where 
$I=I(Q,*,m_\bullet,c_\bullet,b_\bullet)$ is an ideal of the path algebra $KQ$ of $Q$. We note that, if $Q$ is a glueing 
of blocks of types I-III (empty marking), then $Q$ is a triangulation quiver with $Q=Q^*$ and the weighted generalized 
triangulation algebra $\Lambda(Q,*,m_\bullet,c_\bullet,b_\bullet)$ is isomorphic to the weighted triangulation algebra 
$\Lambda(Q,f,m_\bullet,c_\bullet,b_\bullet)$. Moreover, it was shown in \cite{HSS} that the virtual mutations of weighted 
triangulation algebras are the weighted generalized triangulation algebras of the form $\Lambda(Q,*,m_\bullet,c_\bullet,b_\bullet)$, 
where $Q$ is a gluing of blocks of types I-IV. In other words, class of weighted generalized triangulation algebras serves 
as a natural extension of both mentioned classes, and it is definitely the class of algebras worth studying -- not only 
because of wide variety of examples it provides. It was conjectured by Skowro\'nski (oral communication) that the class 
of weighted generalized triangulation algebras is closed under mutations and it exhausts all algebras derived equivalent 
to weighted triangulation algebras, and perhaps, even all tame symmetric periodic algebras of period four. 

We assume that all weighted generalized triangulation algebras considered in the paper are different from a singular 
disc, triangle, tetrahedral or spherical algebras (these exceptional algebras are not symmetric or not periodic, see 
\cite{ES2, ES5}). The following theorem is the main result of this article. 

\begin{theorem1}\label{thm1} Let $\Lambda=\Lambda(Q,*,m_\bullet,c_\bullet,b_\bullet)$ be a weighted generalized 
triangulation algebra. 
Then the following statements hold. 
\begin{enumerate}[(i)]
\item $\Lambda$ is a finite-dimensional symmetric algebra. 
\item $\Lambda$ is a tame algebra of infinite representation type.
\item $\Lambda$ is a periodic algebra of period $4$.
\end{enumerate} 
\end{theorem1} 

To any weighted generalized triangulation algebra $\Lambda=\Lambda(Q,*,m_\bullet,c_\bullet,b_\bullet)$ we will associate 
a weighted triangulation algebra $\Lambda^\Delta=\Lambda(Q^\Delta,f^\Delta,m_\bullet^\Delta,c_\bullet^\Delta,b_\bullet^\Delta)$ 
in a canonical way, so that the following theorem holds. 

\begin{theorem2}\label{thm2} Let $\Lambda$ be a weighted generalized triangulation algebra and $\Lambda^\Delta$ the associated 
weighted triangulation algebra. Then the algebras $\Lambda$ and $\Lambda^\Delta$ are derived equivalent. \end{theorem2} 

The paper is organized in the following way. In Section \ref{sec:1} we give a quick overview of basic facts on derived 
equivalences. Section \ref{sec:2} is devoted to introduce the notion of a generalized triangulation quiver. Further, 
in Section \ref{sec:3} we discuss the weighted gereralized triangulation algebras and their properties (including precise 
formula for dimension). Section \ref{sec:4} is dedicated to recall some background on weighted triangulation algebras. 
Here we also explain the construction of the weighted triangulation algebra $\Lambda^\Delta$ associated to a weighted 
generalized triangulation algebra $\Lambda$ (in fact, we will see that $\Lambda$ can be obtained from $\Lambda^\Delta$ 
by a sequence of mutations). Section \ref{sec:proof} delivers the proofs of Theorems 1 and 2. In the final Section \ref{sec:5} 
we describe a construction (intrinsically motivated by \cite{FST2}) which allows to associate to any (oriented) triangulated 
surface with marked self-folded triangles a generalized triangulation quiver, so its algebra as well. 

For necessary background on the representation theory we refer the reader to \cite{ASS,SY}.  

\section{Derived equivalences of algebras}\label{sec:1}
In this section, we recall a few basic facts on derived equivalences of algebras needed in our paper. 
Moreover, we discuss a relevant construction of tilting complexes which is very useful in realizing 
derived equivalences between symmetric and periodic algebras. 

Given an algebra $A$, we denote by $K^b(\mod A)$ the homotopy category of bounded complexes of modules 
in $\mod A$ and by $K^b(P_A)$ its subcategory formed by bounded complexes of projective modules. The 
derived category $D^b(\mod A)$ of $A$ is the localization of $K^b(\mod A)$ with respect to quasi-isomorphisms, 
and admits structure of a triangulated category, where the suspension functor is given by left shift 
$(-)[1]$ (see \cite{Ha1}). Algebras $A$ and $B$ are called \emph{derived equivalent} provided 
their derived categories $D^b(\mod A)$ and $D^b(\mod B)$ are equivalent as triangulated categories. Derived equivalences 
are commonly realized through tilting complexes, being natural extension of tilting modules. A complex $T$ in $K^b(P_A)$ 
is called a \emph{tilting complex} \cite{Ric1}, if the following conditions are satisfied: 
\begin{enumerate}
\item[(T1)] $\Hom_{K^b(P_A)}(T,T[i])=0$, for all integers $i\neq 0$, 
\item[(T2)] $\add(T)$ generates $K^b(P_A)$ as triangulated category. 
\end{enumerate} 

We have the following theorem due to Rickard \cite[Theorem 6.4]{Ric1}.

\begin{theorem}\label{thm:1.1} Algebras $A$ and $B$ are derived equivalent if and only if there exists 
a tilting complex $T$ in $K^b(P_A)$ such that $\End_{K^b(P_A)}\cong B$. 
\end{theorem}

We recall also the following theorems (see \cite[Corollary 5.3]{Ric3} and \cite[Theorem 2.9]{ES1}) showing some invariants of derived equivalence. 

\begin{theorem}\label{thm:1.2} Let $A$ and $B$ be derived equivalent algebras. Then $A$ is symmetric if 
and only if $B$ is symmetric.
\end{theorem}

\begin{theorem}\label{thm:1.3} Let $A$ and $B$ be derived equivalent algebras. Then $A$ is periodic if 
and only if $B$ is periodic. Moreover, if this is the case, then $A$ and $B$ have the same period. 
\end{theorem}

In the class of self-injective algebras derived equivalence implies stable equivalence (see \cite[Corollary 2.2]{Ric2}), so we 
conclude from \cite[Theorems 4.4 and 5.6]{CB} and \cite[Corollary 2]{KZ} that the following theorem holds.

\begin{theorem}\label{thm:1.4} Let $A$ and $B$ be derived equivalent self-injective algebras. Then the 
following equivalences are valid.
\begin{enumerate}[$(1)$]
\item $A$ is tame if and only if $B$ is tame.
\item $A$ is of polynomial growth if and only if $B$ is of polynomial growth. 
\end{enumerate}
\end{theorem} 

Finally, we present a simple construction of tilting complexes of length 2 over symmetric algebras, observed 
first by Okuyama \cite{O} and Rickard \cite{Ric2}. These tilting complexes have been used extensively to 
verify various cases of Brou\'e's abelian defect group conjecture \cite{CR}, as well as in realizing 
derived equivalences between symmetric algebras (see \cite{BiHS}, \cite{BoHS}, \cite{Ho}, \cite{Ka}, 
\cite{MS}, \cite{Ric2}). 

Let $A$ be a basic, connected, symmetric algebra with Grothendieck group of rank at least 2 and $A=A_A=P\oplus Q$ 
a proper decomposition in $\mod A$. Let $f:P\to Q'$ be a left $\add(Q)$-approximation of $P$, that is $Q'$ is a 
module in $\add(Q)$ and $f$ induces a surjective map 
$$\Hom_A(f,Q''):\Hom_A(Q',Q'')\to\Hom_A(P,Q''),$$ 
for any module $Q''$ in $\add(Q)$. Moreover, consider the following complexes in $K^b(P_A)$ 
$$T_1:\qquad \xymatrix{0\ar[r]& Q \ar[r] & 0}$$
concentrated in degree $0$, and 
$$T_2:\xymatrix{0\ar[r]&P\ar[r]^{f}&Q'\ar[r]&0}$$ 
concentrated in degrees $1$ and $0$. 

\begin{prop}\label{pro:1.5}
$T:=T_1\oplus T_2$ is a tilting complex in $K^b(P_A)$. 
\end{prop} 

\begin{proof} For a proof we refer to \cite[Proposition 2.1]{Du2}. \end{proof} 

\section{Generalized triangulation quivers}\label{sec:2} 
A quiver is a quadruple $Q=(Q_0,Q_1,s,t)$ consisting of a finite set $Q_0$ of vertices, a finite set $Q_1$ of arrows, and two maps 
$s,t:Q_1\to Q_0$ which associate to each arrow $\alpha\in Q_1$ its source $s(\alpha)\in Q_0$ and its target $t(\alpha)\in Q_0$. We 
assume throughout that any quiver is connected. 

A quiver $Q$ is called $2$-regular if for each $i\in Q_0$ there are precisely two arrows with source $i$ and precisely two arrows 
with target $i$. We recall that a \emph{triangulation quiver} is a pair $(Q,f)$ consisting of a $2$-regular quiver $Q$ and a 
permutation $f:Q_1\to Q_1$ of the arrows such that $t(\alpha)=s(f(\alpha))$ for each arrow $\alpha\in Q_1$, and $f^3$ is the identity 
on $Q_1$ (see \cite{ES2}, \cite{ES4}). Hence, all cycles of $f$ in $Q_1$ have length $3$ (triangles) or $1$ (loops). We only 
mention that every triangulation quiver with at least three vertices is a triangulation quiver associated to a directed triangulated 
surface (see \cite[Theorem 4.11]{ES2}). 

In this article, by a \emph{block} we mean a quiver of one of the forms 
$$
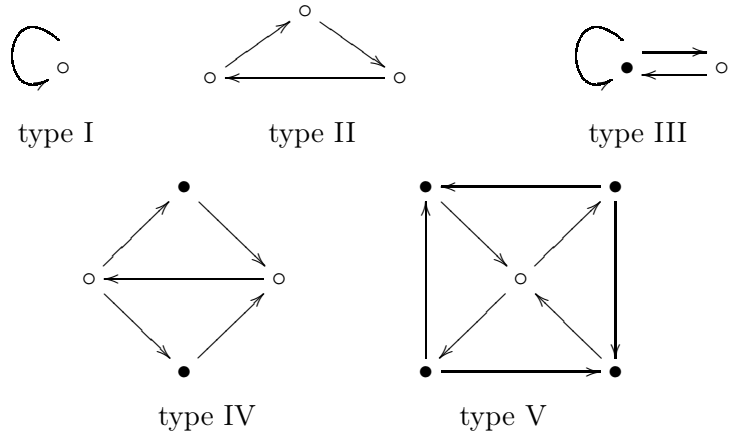
\xymatrix@R=0.01cm{&\\&\mbox{ } \\ &\ar@(lu,ld)[d]_{\mbox{ }}\\& \circ }\qquad\qquad
\xymatrix@R=0.5cm{&\circ\ar[rd] & \\ \circ\ar[ru] & &\circ\ar[ll]}\qquad\qquad
\xymatrix@R=0.01cm{&&\\&&\\&\ar@(lu,ld)[d]_{\mbox{ }}&\\& \bullet\ar@<0.2cm>[r]  &\circ \ar@<0.1cm>[l]}$$ 
$$\mbox{type I}\qquad\qquad\quad\quad\mbox{type II}\qquad\qquad\qquad\qquad\mbox{type III}$$
$$\xymatrix{&\bullet\ar[rd] & \\ \circ\ar[ru]\ar[rd] & &\circ\ar[ll] \\  & \bullet\ar[ru] & }\qquad\qquad 
\xymatrix{\bullet\ar[rd] & & \bullet\ar[ll]\ar[dd] \\ &\circ \ar[ld]\ar[ru] &  \\ \bullet\ar[rr]\ar[uu] & & \bullet\ar[lu] }$$ 
$$\mbox{type IV}\qquad\qquad\quad\qquad\mbox{type V}$$
The vertices marked by the white circles are called \emph{outlets}. 

Let $B_1,\dots,B_r$, $r\geqslant 2$, be a family of blocks, and $W_1,\dots,W_r$ denote their outlets with $W$ their disjoint union. A \emph{glueing map} (for $B_1,\dots,B_r$) is an involution $\Theta:W\to W$ such that $\Theta(W_k)\cap W_k=\emptyset$, for any 
$k\in\{1,\dots,r\}$. Then we define the quiver $glue(B_1,\dots,B_r;\Theta)$ obtained from the disjoint union of the quivers 
$B_1,\dots,B_r$ by identifying vertices $x$ and $\Theta(x)$, where $x$ runs through $W$. We say that a finite connected quiver $Q$ 
is \emph{block decomposable} if $Q$ is of the form $Q=glue(B_1,\dots,B_r;\Theta)$ for a family of blocks $B_1,\dots,B_r$, 
$r\geqslant 2$, and a glueing map $\Theta$.    

\begin{df}\label{def:2.1} A \emph{generalized triangulation quiver} is a pair $(Q,*)$ consisting of a block decomposable 
quiver $Q$ together with a marking $*$ of one triangle in each block of type IV and V. \end{df} 

For a generalized triangulation quiver $(Q,*)$, $Q=glue(B_1,\dots,B_r;\Theta)$, let $Q^*$ be the quiver obtained from $Q$ by removing 
all the black vertices in all marked triangles and the arrows attached to them. We observe that $Q^*$ is a connected quiver and there 
is a permutation $f:Q^*_1\to Q^*_1$ of arrows of $Q^*$ such that $t(\eta)=s(f(\eta))$ for any arrow $\eta\in Q^*_1$, and $f^3$ is the 
identity. Namely, $f$ fixes any loop of a block of type I and rotates the three arrows of any triangle comming from blocks of types 
II-V. Note also that after removing marked triangles from $Q$, blocks of type IV and V become usual triangles in $Q^*$ (with one 
or two $1$-regular vertices; see below), whereas blocks of types I-III remain unchanged in $Q^*$. 

Moreover, any white vertex (outlet) $x\in Q_0^*$ is $2$-regular, that is $x$ is the source of exactly two arrows and the 
target of exactly two arrows (in $Q^*$). On the other hand, every black vertex $x\in Q^*_0$ not lying in a block of type 
III is $1$-regular, which means that $x$ is the source of exactly one arrow and the target of exactly one arrow (in $Q^*$). 
In particular, $Q^*$ is 2-regular (and so a triangulation quiver) if and only if $Q$ does not contain blocks of types IV and V. 
However, we can define an involution $\bar{\mbox{ }}:Q^*_1\to Q^*_1$ which assigns to each arrow $\alpha \in Q^*_1$ with 
$s(\alpha)$ being white, the second arrow $\bar{\alpha}$ with $s(\bar{\alpha})=s(\alpha)$, and to each arrow $\beta\in Q^*_1$ with $s(\beta)$ being black, the same arrow $\bar{\beta}=\beta$. Then we obtain the second permutation $g:Q^*_1\to Q^*_1$ defined 
as $g(\alpha)=\overline{f(\alpha)}$ for any arrow $\alpha\in Q^*_1$. For each arrow $\alpha\in Q^*_1$ we denote by 
$\mathcal{O}(\alpha)$ the $g$-orbit of $\alpha$ in $Q^*_1$ and set $n_\alpha= |\mathcal{O}(\alpha)|$. Hence, $\mathcal{O}(\alpha)$ 
is a $g$-cycle of the form $(\alpha\mbox{ } g(\alpha)\mbox{ }\dots\mbox{ } g^{n_\alpha -1}(\alpha))$. We write $\mathcal{O}(g)$ 
for the set of all $g$-orbits in $Q^*_1$. 

One should think of $Q^*$ as of a quasi-triangulated quiver, by which we mean a triangulation quiver, besides that not all 
vertices in $Q^*$ are required $2$-regular. But still, it admits a permutation $f$ and an involution, which produce permutation $g$, 
and hence, allow to consider associated algebras as presented in the next section. We note the following obvious fact. 

\begin{prop}\label{prop:2.2} Let $(Q,*)$ be a generalized triangulation quiver. The following statements are equivalent. 
\begin{enumerate}[(i)]
\item $Q=Q^*$.
\item There is a triangulation quiver $(Q,f)$.
\item $Q=glue(B_1,\dots,B_r;\Theta)$ with all blocks $B_1,\dots,B_r$ of types I-III. 
\end{enumerate} 
\end{prop} 

Let us present one example of a generalized triangulation quiver which is a glueing of blocks of all five possible types.  

\begin{exmp}\label{exm:2.3} Let $(Q,*)$ be the following generalized triangulation quiver 
$$\xymatrix@R=0.01cm@C=0.3cm{
&&&&&&&&&\\ 
&&&&\bullet_{y_2}\ar[rr]^{\epsilon}\ar[rrdddd]_(0.2){\rho} & & \bullet_{y_1}\ar@/^30pt/[ldddddd]^(0.25){\psi} &&&\\
&&&&&&&&&\\
&&&&&&&&&\\
&&&&&&&&&\\
&&&&\bullet_{x_2}\ar[rr]^{\sigma}_{*}\ar[rruuuu]^(0.2){\eta} && \bullet_{x_1}\ar[ldd]^{\omega} &&&\\
&&&&&&&&&\\
&&&&&\circ_{z}\ar[luu]^{\gamma}\ar@/^30pt/[luuuuuu]^(0.75){\phi}\ar[rdddd]^{\pi} &&&&\\
&&&&&&&&&\\
&&&&&&&&&\\
&&&&&&&&&\\
&&\circ_{c_4}\ar[rr]^{\beta_4}\ar@<-0.1cm>[rdddd]_{\xi_4} &&\circ_{b_4}\ar[ldddd]^{\nu_4}\ar[ruuuu]^{\lambda} & &\circ_{a_5}\ar[ll]^{\chi}\ar[rr]^{\alpha_5} 
 & & \circ_{c_5}\ar[rdddd]^{\beta_5}\ar@<-0.1cm>[ldddd]_{\xi_5} &\\
&&&&&&&&&\\
&&&&&&&&&\\
&&&&&&&&&\\
&\circ_{a_4}\ar[ldddd]_{w_1}\ar[ruuuu]^{\alpha_4} &&\circ_{d_4}\ar@<-0.1cm>[luuuu]_{\mu_4}\ar[ll]^{\delta_4} &&&& 
 \circ_{d_5}\ar[luuuu]^{\delta_5}\ar@<-0.1cm>[ruuuu]_{\mu_5} && \circ_{b_5}\ar[ll]^{\nu_5}\ar@<0.2cm>[lddddddddddd]^{v_2} \\
&&&&&&&&&\\
&&&&&&&&&\\ 
\ar@(lu,ld)[d]_{\zeta}&&&&&& \ar@(lu,ld)[d]_{\theta}& & &\\ 
\circ^{p_1}\ar[rrddddddd]_{u_1}& &&&&& \bullet^{p_3}\ar@<0.2cm>[rr]^{\kappa}  & &
\circ^{p_2} \ar@<0.1cm>[ll]^{\iota}\ar[ruuuu]_{u_2} &\\
&&&&&&&&&\\
&&&&&&&&&\\
&&&&&&&&&\\
&&&\bullet_{c_1}\ar[lddd]_{\alpha_1} &&\bullet_{d_2}\ar[rddd]^{\delta_2} &&\bullet_{d_3}\ar[rddd]^(0.4){\delta_3} & &\\
&&&&&&&&&\\
&&&&&&&&&\\
&&\circ_{a_1}\ar[luuuuuuuuuuu]^{v_1}\ar[rr]_{\tau_1}^{*} & & \circ_{a_2}\ar[ruuu]^{\nu_2}\ar[rddd]_{\beta_2}\ar[luuu]_{\beta_1}\ar[lddd]^{\nu_1} & & \circ_{a_3}\ar[ll]_{\tau_2}^{*}\ar[ruuu]^{\nu_3}\ar[rddd]_{\beta_3} & & 
\circ_{b_3}\ar[ll]_{\tau_3}^{*}\ar[uuuuuuu]^(0.7){w_2} &\\ 
&&&&&&&&&\\
&&&&&&&&&\\
&&&\bullet_{d_1}\ar[luuu]^{\delta_1} &&\bullet_{c_2}\ar[ruuu]_{\alpha_2} & &\bullet_{c_3}\ar[ruuu]_{\alpha_3} & & }$$ 
Then $Q$ is a gluing of: one block of type I, seven blocks of type II, one block of type III, three blocks of type IV and one 
block of type V. 

The associated quiver $Q^*$ is obtained from $Q$ by removing the vertices $x_1,x_2,c_1,c_2,c_3$ and the arrows 
$\rho,\eta,\gamma,\sigma,\omega, \alpha_1,\beta_1,\alpha_2,\beta_2,\alpha_3,\beta_3$. Moreover, we have the following $f$-orbits in $Q^*_1$: 
$$(\phi\mbox{ }\epsilon\mbox{ }\psi), \ (\lambda\mbox{ }\pi\mbox{ }\chi), \ (\nu_1\mbox{ }\delta_1\mbox{ }\tau_1), \ 
(\nu_2\mbox{ }\delta_2\mbox{ }\tau_2), \ (\nu_3\mbox{ }\delta_3\mbox{ }\tau_3), $$ 
$$(\alpha_4\mbox{ }\xi_4\mbox{ }\delta_4), \ (\beta_4\mbox{ }\nu_4\mbox{ }\mu_4),\ (\alpha_5\mbox{ }\xi_5\mbox{ }\delta_5), \ 
(\beta_5\mbox{ }\nu_5\mbox{ }\mu_5),\ (u_1\mbox{ }v_1\mbox{ }w_1),\ (u_2\mbox{ }v_2\mbox{ }w_2),\ 
(\theta\mbox{ }\kappa\mbox{ }\iota),\ (\zeta).$$ 
Then the set $\mathcal{O}(g)$ of $g$-orbits in $Q^*_1$ consists of the following five orbits: 
$$\mathcal{O}(\phi)=(\phi\mbox{ }\epsilon\mbox{ }\psi\mbox{ }\pi\mbox{ }\alpha_5\mbox{ }\beta_5\mbox{ }v_2\mbox{ }\tau_3\mbox{ }\tau_2\mbox{ }
\nu_1\mbox{ }\delta_1\mbox{ }v_1\mbox{ }\alpha_4\mbox{ }\beta_4\mbox{ }\lambda),$$ 
$$\mathcal{O}(\chi)=(\chi\mbox{ }\nu_4\mbox{ }\delta_4\mbox{ }w_1\mbox{ }\zeta\mbox{ }u_1\mbox{ }\tau_1\mbox{ }\nu_2\mbox{ }
\delta_2\mbox{ }\nu_3\mbox{ }\delta_3\mbox{ }w_2\mbox{ }\iota\mbox{ }\kappa\mbox{ }u_2\mbox{ }\nu_5\mbox{ }\delta_5),$$ 
$$\mathcal{O}(\xi_5)=(\xi_5\mbox{ }\mu_5),\,\mathcal{O}(\xi_4)=(\xi_4\mbox{ }\mu_4) ,\mbox{ and }\mathcal{O}(\theta)=(\theta).$$   
\end{exmp} 

\bigskip 

\section{Weighted generalized triangulation algebras}\label{sec:3} 
Let $(Q,*)$ be a generalized triangulation quiver, with $Q=glue(B_1,\dots,B_r;\Theta)$. Moreover, let $Q^*$ be the 
associated quiver, and $f,g:Q^*_1\to Q^*_1$ the permutations, for which $f^3=id_{Q_1^*}$ and $g=\bar{f}$, defined 
in the previous section. Recall also that for each arrow $\alpha\in Q^*_1$, we denote by $\mathcal{O}(\alpha)$ the 
$g$-orbit of $\alpha$ in $Q^*_1$, $n_\alpha=|\mathcal{O}(\alpha)|$, and by $\mathcal{O}(g)$ the set of all $g$-orbits 
in $Q^*_1$. A function 
$$m_\bullet:\mathcal{O}(g)\to\mathbb{N}^*:=\mathbb{N}\setminus\{0\}$$ 
is called a \emph{weight function} of $(Q,*)$, and a function 
$$c_\bullet:\mathcal{O}(g)\to K^*:=K\setminus\{0\}$$ 
is called a \emph{parameter function} of $(Q,*)$. In other words, weight and parameter functions are functions on $Q_1^*$ 
constant on $g$-orbits. We write briefly $m_\alpha=m_{\mathcal{O}(\alpha)}$ and $c_\alpha=c_{\mathcal{O}(\alpha)}$ for 
$\alpha\in Q^*_1$. We will say that an arrow $\alpha\in Q^*_1$ is \emph{virtual} if $m_\alpha n_\alpha =2$. Following \cite{ES5}, 
we assume that given weight function $m_\bullet$ of $(Q,*)$ satisfies the following restrictions: 
\begin{enumerate}[(1)] 
\item $m_\alpha n_\alpha\geqslant 2$ for all arrows $\alpha\in Q^*_1$,
\item $m_\alpha n_\alpha\geqslant 3$ for all arrows $\alpha\in Q^*_1$ such that $\bar{\alpha}$ is virtual but not a loop, 
\item $m_\alpha n_\alpha\geqslant 4$ for all arrows $\alpha\in Q^*_1$ such that $\bar{\alpha}$ is a virtual loop.
\end{enumerate} 
Condition (1) is a general assumption, whilst (2) and (3) are included to eliminate particular small algebras 
\cite[see Sections 3 and 4]{ES5}.  

For each arrow $\alpha\in Q^*_1$, we define the paths 
$$A_\alpha=\alpha g(\alpha)\dots g^{m_\alpha n_\alpha -2}(\alpha)\quad\mbox{and}\quad B_\alpha=
\alpha g(\alpha)\dots g^{m_\alpha n_\alpha -1}(\alpha)$$ 
along the $g$-cycle of $\alpha$, of lengths $m_\alpha n_\alpha - 1$ and $m_\alpha n_\alpha$, respectively. If 
$m_\alpha n_\alpha\geqslant 3$, we consider also the following path $A_\alpha '=\alpha g(\alpha)\dots g^{m_\alpha n_\alpha -3}(\alpha)$ 
in $Q^*$ of length $m_\alpha n_\alpha-2$. Obviously, we have $A_\alpha g^{-1}(\alpha)=B_\alpha$ and 
$A'_\alpha g^{-2}(\alpha)=A_\alpha$ if $A'_\alpha$ is defined. 

A loop $\alpha\in Q^*_1$ with $f(\alpha)=\alpha$ (i.e. coming from a block of type I) is said to be a \emph{border loop} of $(Q,*)$, 
and then the vertex $s(\alpha)$ is called a \emph{border vertex} of $(Q,*)$. We denote by $\partial(Q,*)$ the set of all border 
vertices of $(Q,*)$. If $\partial(Q,*)$ is not empty, then any function 
$$b_\bullet:\partial(Q,*)\to K$$ 
is said to be a \emph{border function} of $(Q,*)$. 

We shall define a weighted generalized triangulation algebra 
$$\Lambda(Q,*,m_\bullet,c_\bullet,b_\bullet),$$ 
which in case $Q$ is a glueing of blocks of types I-III coincides with the socle deformed weighted triangulation algebra 
$\Lambda(Q,f,m_\bullet,c_\bullet,b_\bullet)$ (see Definition \ref{df:4.1}). The algebra $\Lambda(Q,*,m_\bullet,c_\bullet,b_\bullet)$ 
will be given by the quiver $Q$ and a set of relations reflecting the blocks defining the generalized triangulation quiver $(Q,*)$ 
as well as a quasi-triangulated structure on $Q^*$. 

We assume that if blocks of type IV occur in $(Q,*)$, then these are blocks $B_1,\dots,B_s$ with the arrows in $B_i$ denoted as follows 
$$
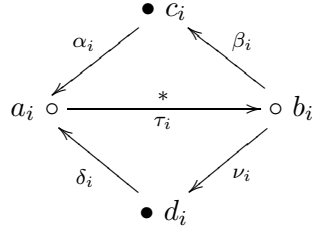
\xymatrix{ & \bullet \mbox{ }c_i  \ar[ld]_{\alpha_i} & \\ 
a_i\mbox{ }\circ\ar[rr]_{\tau_i}^{*} & & \circ\mbox{ } b_i \ar[lu]_{\beta_i}\ar[ld]^{\nu_i} \\ 
 & \bullet\mbox{ }d_i\ar[lu]^{\delta_i} & }$$ 
for $i\in\{1,\dots,s\}$. We set $s=0$ if no block of type IV occurs in $(Q,*)$. 

If blocks of type V occur in $(Q,*)$, then these are blocks $B_{s+1},\dots,B_{s+t}$ with the blocks $B_{s+j}$ of the forms 
$$
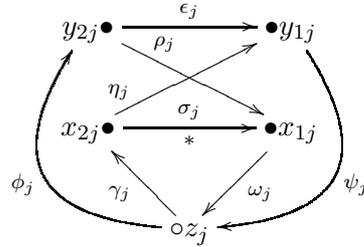
\xymatrix@C=0.5cm{&y_{2j}\bullet\ar[rrd]^(0.33){\rho_j}\ar[rr]^{\epsilon_j}&&\bullet y_{1j}\ar@/^{40pt}/[ldd]^{\psi_j}& \\ 
&x_{2j}\bullet\ar[rr]^{\sigma_j}_{*}\ar[rru]^(0.2){\eta_j}&&\bullet x_{1j}\ar[ld]^{\omega_j}& \\ 
&&\circ z_j\ar[lu]^{\gamma_j}\ar@/^{40pt}/[luu]^{\phi_j}&& }$$ 
for $j\in\{1,\dots,t\}$. We set $t=0$ if no block of type V occurs in $(Q,*)$. 

The definition of a weighted generalized triangulation algebra is now as follows. 

\begin{df}\label{df:3.1} The \emph{weighted generalized triangulation algebra} 
$$\Lambda(Q,*,m_\bullet,c_\bullet,b_\bullet)$$ 
is the quotient $KQ/I$ of the path algebra $KQ$ by the ideal $I=I(Q,*,m_\bullet,c_\bullet,b_\bullet)$ of $KQ$ generated by the following relations: 
\begin{enumerate}[1)] 
\item $\alpha^2-c_{\bar{\alpha}}A_{\bar{\alpha}}-b_{s(\alpha)}B_\alpha$, for all loops $\alpha\in Q_1$ from blocks of type I, 

\item $\alpha f(\alpha)-c_{\bar{\alpha}}A_{\bar{\alpha}}$, for all arrows $\alpha\in Q_1$ from blocks of types II and III,

\item for each $i\in\{1,\dots,s\}$ the relations 
\begin{itemize}
\item $\nu_i\delta_i-\beta_i\alpha_i-c_{\bar{\nu_i}}A_{\bar{\nu_i}}$, $\delta_i\tau_i-c_{\bar{\delta_i}}A_{\bar{\delta_i}}$, 
$\tau_i\nu_i-c_{\bar{\tau_i}} A_{\bar{\tau_i}}$, $\alpha_i\tau_i$, $\tau_i\beta_i$, $\delta_i\tau_i g(\tau_i)$, 
\item $\delta_i g(\delta_i)f(g(\delta_i))$, except $\tau_i$ is virtual or $g(\delta_i)\in\{\nu_1,\dots,\nu_s,\phi_1,\dots,\phi_t\}$, 
and: \item $\tau_i g(\tau_i) f(g(\tau_i))$ except $m_{\nu_i}=1$ and $n_{\nu_i}=3$ or 
$g(\tau_i)\in\{\nu_1,\dots,\nu_s,\phi_1,\dots,\phi_t\}$, 
\end{itemize}

\item  for each $j\in\{1,\dots,t\}$ the relations: 
\begin{itemize}
\item $\phi_j\epsilon_j-\gamma_j\eta_j-c_{\bar{\phi_j}}A_{\bar{\phi_j}}$, 
$\epsilon_j\psi_j-\rho_j\omega_j-c_{\bar{\epsilon_j}}A_{\bar{\epsilon_j}}$, $\psi_j\phi_j-c_{\bar{\psi_j}}A_{\bar{\psi_j}}$, 
$\gamma_j\sigma_j-\phi_j\rho_j$, $\sigma_j\omega_j-\eta_j\psi_j$, $\omega_j\gamma_j$, $\omega_j\phi_j$, $\psi_j\gamma_j$, 
$\phi_j\epsilon_j\psi_j\phi_j$, $\epsilon_j\psi_j\phi_j\epsilon_j$, $\psi_j\phi_j\epsilon_j\psi_j$, $B_{\phi_j}\bar{\phi}_j$, 
and: \item $\psi_j g(\psi_j) f(g(\psi_j))$, except $g(\psi_j)\in\{\nu_1,\dots,\nu_s,\phi_1,\dots,\phi_t\}$, 
\end{itemize}  

\item $\alpha f(\alpha) g(f(\alpha))$, for all arrows $\alpha\in Q_1$ from blocks of types I-III, unless $f^2(\alpha)$ is virtual 
or unless $f(\bar{\alpha})$ is virtual with $m_{\bar{\alpha}}=1$ and $n_{\bar{\alpha}}=3$, 

\item $\alpha g(\alpha) f(g(\alpha))$, for all arrows $\alpha\in Q_1$ from blocks of types I-III such that $g(\alpha)$ is not in 
$\{\nu_1,\dots,\nu_s,\phi_1,\dots,\phi_t\}$, and unless $f(\alpha)$ is virtual or unless $f^2(\alpha)$ is virtual with $m_{f(\alpha)}=1$ 
and $n_{f(\alpha)}=3$. 
\end{enumerate}
\end{df} 

We abbreviate $\Lambda(Q,*,m_\bullet,c_\bullet)=\Lambda(Q,*,m_\bullet,c_\bullet,b_\bullet)$ if $\partial(Q,*)$ is empty or 
$b_\bullet\equiv 0$. We note that for any weighted generalized triangulation algebra $\Lambda$ its opposite algebra 
$\Lambda^{\op}$ is also a weighted generalized triangulation algebra. 

\begin{exmp} \label{ex:3.2} Let $(Q,*)$ be the generalized triangulation quiver of the form 
$$
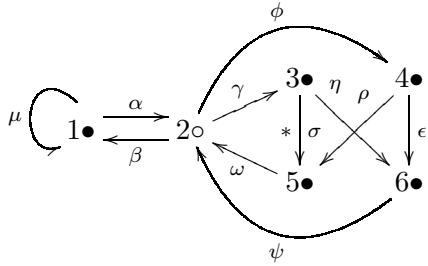
\xymatrix@R=0.01cm{
&&&&\\
&&&3\bullet\ar[dddd]^{\sigma}_{*}\ar[rdddd]^(0.2){\eta}&4\bullet\ar[ldddd]_(0.3){\rho}\ar[dddd]^{\epsilon}\\
&\ar@(lu,ld)[d]_{\mu}&&&\\
& 1\bullet\ar@<0.2cm>[r]^{\alpha}  & 2\circ \ar@<0.1cm>[l]^{\beta}\ar[ruu]^{\gamma}\ar@/^{30pt}/[rruu]^{\phi}&&\\
&&&& \\
&&&5\bullet\ar[luu]^{\omega}&6\bullet\ar@/^{30pt}/[lluu]^{\psi} \\ }$$ 
which is glueing of one block of type III with one block of type V. Then the associated quasi-triangulation quiver $Q^*$ has 
the following shape 
$$
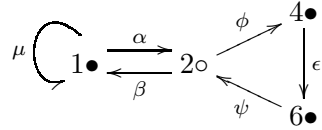
\xymatrix@R=0.01cm{
&&&\\
&&&4\bullet\ar[dddd]^{\epsilon}\\
&\ar@(lu,ld)[d]_{\mu}&&\\
& 1\bullet\ar@<0.2cm>[r]^{\alpha}  & 2\circ \ar@<0.1cm>[l]^{\beta}\ar[ruu]^{\phi}&\\
&&&& \\
&&&6\bullet\ar[luu]^{\psi} \\ }$$ 
and the permutation $f:Q^*_1\to Q^*_1$ has two cycles: $(\alpha\mbox{ }\beta\mbox{ }\mu)$ and $(\phi\mbox{ }\epsilon\mbox{ }\psi)$. 
So we have two $g$-orbits: $\mathcal{O}(\alpha)=(\alpha\mbox{ }\phi\mbox{ }\epsilon\mbox{ }\psi\mbox{ }\beta)$ and 
$\mathcal{O}(\mu)=(\mu)$, and clearly, the border $\partial(Q,*)$ is empty. 

Let $m_\bullet:\mathcal{O}(g)\to\mathbb{N}^*$ and $c_\bullet:\mathcal{O}(g)\to K^*$ be weight and parameter functions of 
$(Q,*)$. We set $m=m_{\mathcal{O}(\alpha)}$, $n=m_{\mathcal{O}(\mu)}$, $c=c_{\mathcal{O}(\alpha)}$ and $d=c_{\mathcal{O}(\mu)}$. 
According to general assumption, we have $n\geqslant 2$. Then the weighted generalized triangulation algebra 
$\Lambda(Q,*,m_\bullet,c_\bullet)$ is given by the quiver $Q$ and the following relations: 
$$\alpha\beta=d\mu^{n-1},\beta\mu=c(\phi\epsilon\psi\beta\alpha)^{m-1}\phi\epsilon\psi\beta, 
\mu\alpha=c(\alpha\phi\epsilon\psi\beta)^{m-1}\alpha\phi\epsilon\psi,$$
$$\phi\epsilon=\gamma\eta+c(\beta\alpha\phi\epsilon\psi)^{m-1}\beta\alpha\phi\epsilon,
\epsilon\psi=\rho\omega+c(\epsilon\psi\beta\alpha\phi)^{m-1}\epsilon\psi\beta\alpha,$$
$$\psi\phi=c(\psi\beta\alpha\phi\epsilon)^{m-1}\psi\beta\alpha\phi,\gamma\sigma=\phi\rho,\sigma\omega=\eta\psi,
\omega\gamma=0,\omega\phi=0,\psi\gamma=0,$$ 
$$\phi\epsilon\psi\phi=0, \epsilon\psi\phi\epsilon=0, \psi\phi\epsilon\psi=0, 
(\phi\epsilon\psi\beta\alpha)^m\beta=0, \psi\beta\mu=0,$$ 
$$\beta\mu^2=0,\mu\alpha\phi=0,\mu^2\alpha=0,\mbox{ and}$$ 
$$\alpha\beta\alpha=0\mbox{ and }\beta\alpha\beta=0\mbox{ if }n\geqslant 3\mbox{ ($\mu$ is not virtual).}$$
\end{exmp} 

\begin{exmp}\label{ex:3.3} Let $(Q,*)$ be the generalized triangulation quiver of the form 
$$
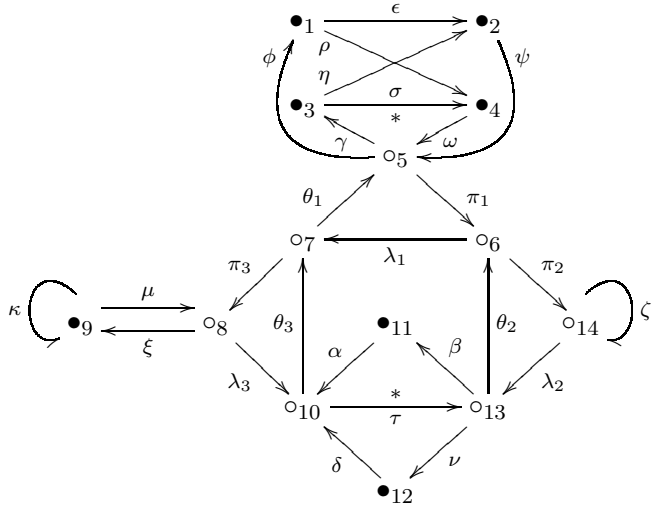
\xymatrix@R=0.01cm@C=0.5cm{
&&&&&&&&\\ 
&&&\bullet_{1}\ar[rr]^{\epsilon}\ar[rrdddd]_(0.15){\rho} & & \bullet_{2}\ar@/^29pt/[ldddddd]^(0.15){\psi} &&&\\
&&&&&&&&\\
&&&&&&&&\\
&&&&&&&&\\
&&&\bullet_{3}\ar[rr]^{\sigma}_{*}\ar[rruuuu]^(0.15){\eta} && \bullet_{4}\ar[ldd]^{\omega} &&&\\
&&&&&&&&\\
&&&&\circ_{5}\ar[luu]^{\gamma}\ar@/^29pt/[luuuuuu]^(0.85){\phi}\ar[rdddd]^(0.7){\pi_1} &&&&\\
&&&&&&&&\\
&&&&&&&&\\
&&&&&&&&\\
&&&\circ_7 \ar[ruuuu]^(0.3){\theta_1}\ar[ldddd]_{\pi_3} & &\circ_6\ar[ll]^{\lambda_1}\ar[rdddd]^{\pi_2} & & &\\
&&&&&&&&\\
&&&&&&&&\\
\ar@(lu,ld)[d]_{\kappa} &&&&&&\ar@(ru,rd)[d]^{\zeta} & &\\ 
\bullet_9\ar@<0.2cm>[rr]^{\mu} & &\circ_8\ar@<0.1cm>[ll]^{\xi}\ar[rdddd]_{\lambda_3} && \bullet_{11}\ar[ldddd]_{\alpha} 
 &   & \circ_{14}\ar[ldddd]^{\lambda_2} & &\\
&&&&&&&&\\
&&&&&&&&\\
&&&&&&&&\\
&& &\circ_{10}\ar[rr]_{\tau}^{*}\ar[uuuuuuuu]^{\theta_3} &  &\circ_{13}\ar[uuuuuuuu]_{\theta_2}\ar[luuuu]_{\beta}\ar[ldddd]^{\nu} & & &\\
&&&&&&&&\\
& & & & & & & &\\ 
&&&&&&&&\\
&& && \bullet_{12}\ar[luuuu]^{\delta}& & & & }$$ 
We observe that $Q$ is glueing of one block of type I, three blocks of type II, one block of type III, one block 
of type IV and one block of type V. 

The associated quiver $Q^*$ is obtained from $Q$ by removing the vertices $3,4,11$ and the arrows 
$\rho,\eta, \gamma,\sigma,\omega,\alpha,\beta$. The permutation $f:Q^*_1\to Q^*_1$ has the following orbits: $(\zeta)$, 
$(\mu\mbox{ }\xi\mbox{ }\kappa)$, $(\nu\mbox{ }\delta\mbox{ }\tau)$, $(\phi\mbox{ }\epsilon\mbox{ }\psi)$ and 
$(\pi_i\mbox{ }\lambda_i\mbox{ }\theta_i)$ for $i\in\{1,2,3\}$. Hence the permutation $g:Q^*_1\to Q^*_1$ has the following orbits: 
$$\mathcal{O}(\phi)=(\phi\mbox{ }\epsilon\mbox{ }\psi\mbox{ }\pi_1\mbox{ }\pi_2\mbox{ }\zeta\mbox{ }\lambda_2\mbox{ }\nu\mbox{ }\delta
\mbox{ }\theta_3\mbox{ }\theta_1),$$ 
$$\mathcal{O}(\lambda_1)=(\lambda_1\mbox{ }\pi_3\mbox{ }\xi\mbox{ }\mu\mbox{ }\lambda_3\mbox{ }\tau\mbox{ }\theta_2),$$
$$\mathcal{O}(\kappa)=(\kappa).$$ 
We also note that $\partial(Q,*)=\{14\}$. 

Let $m_\bullet:\mathcal{O}(g)\to\mathbb{N}^*$ and $c_\bullet:\mathcal{O}(g)\to K^*$ be weight and parameter functions of $(Q,*)$. We set 
$m=m_{\mathcal{O}(\phi)}$, $n=m_{\mathcal{O}(\lambda_1)}$, $p=m_{\mathcal{O}(\kappa)}$, $a=c_{\mathcal{O}(\phi)}$, $c=c_{\mathcal{O}(\lambda_1)}$, 
$d=c_{\mathcal{O}(\kappa)}$. We note that $m\geqslant 1$, $n\geqslant 1$, $p\geqslant 2$ and for $p=2$ the loop $\kappa$ is virtual. Moreover, 
let $b_\bullet:\partial(Q,*)\to K$ be a border function with $b=b_{14}$. 

Then the weighted generalized triangulation algebra $\Lambda=\Lambda(Q,*,m_\bullet,c_\bullet,b_\bullet)$ is given by the quiver $Q$ and the relations: 
\begin{enumerate}[(1)] 
\item $\zeta^2=a(\lambda_2\nu\delta\theta_3\theta_1\phi\epsilon\psi\pi_1\pi_2\zeta)^{m-1}\lambda_2\nu\delta\theta_3\theta_1\phi\epsilon\psi\pi_1\pi_2 + 
b(\zeta\lambda_2\nu\delta\theta_3\theta_1\phi\epsilon\psi\pi_1\pi_2)^m$, 

\item $\pi_1\lambda_1=a(\phi\epsilon\psi\pi_1\pi_2\zeta\lambda_2\nu\delta\theta_3\theta_1)^{m-1}\phi\epsilon\psi\pi_1\pi_2\zeta\lambda_2\nu\delta\theta_3$, 
$\newline$ 
$\lambda_1\theta_1=a(\pi_2\zeta\lambda_2\nu\delta\theta_3\theta_1\phi\epsilon\psi\pi_1)^{m-1}\pi_2\zeta\lambda_2\nu\delta\theta_3\theta_1\phi\epsilon\psi$, 
$\newline$
$\theta_1\pi_1=c(\pi_3\xi\mu\lambda_3\tau\theta_2\lambda_1)^{n-1}\pi_3\xi\mu\lambda_3\tau\theta_2$, 
$\pi_2\lambda_2=c(\lambda_1\pi_3\xi\mu\lambda_3\tau\theta_2)^{n-1}\lambda_1\pi_3\xi\mu\lambda_3\tau$, 
$\newline$
$\lambda_2\theta_2=
a(\zeta\lambda_2\nu\delta\theta_3\theta_1\phi\epsilon\psi\pi_1\pi_2)^{m-1}\zeta\lambda_2\nu\delta\theta_3\theta_1\phi\epsilon\psi\pi_1$, 
$\newline$ 
$\theta_2\pi_2=a(\nu\delta\theta_3\theta_1\phi\epsilon\psi\pi_1\pi_2\zeta\lambda_2)^{m-1}\nu\delta\theta_3\theta_1\phi\epsilon\psi\pi_1\pi_2\zeta$, 
$\newline$ $\pi_3\lambda_3=a(\theta_1\phi\epsilon\psi\pi_1\pi_2\zeta\lambda_2\nu\delta\theta_3)^{m-1}\theta_1\phi\epsilon\psi\pi_1\pi_2\zeta\lambda_2\nu\delta$, 
$\newline$ $\lambda_3\theta_3=c(\xi\mu\lambda_3\tau\theta_2\lambda_1\pi_3)^{n-1}\xi\mu\lambda_3\tau\theta_2\lambda_1$, 
$\theta_3\pi_3=c(\tau\theta_2\lambda_1\pi_3\xi\mu\lambda_3)^{n-1}\tau\theta_2\lambda_1\pi_3\xi\mu$, 
$\newline$
$\mu\xi=d\kappa^{p-1}$, $\xi\kappa=c(\lambda_3\tau\theta_2\lambda_1\pi_3\xi\mu)^{n-1}\lambda_3\tau\theta_2\lambda_1\pi_3\xi$, 
$\kappa\mu=c(\mu\lambda_3\tau\theta_2\lambda_1\pi_3\xi)^{n-1}\mu\lambda_3\tau\theta_2\lambda_1\pi_3$, 

\item $\nu\delta=\beta\alpha+c(\theta_2\lambda_1\pi_3\xi\mu\lambda_3\tau)^{n-1}\theta_2\lambda_1\pi_3\xi\mu\lambda_3,$ $\newline$ 
$\delta\tau=a(\delta\theta_3\theta_1\phi\epsilon\psi\pi_1\pi_2\zeta\lambda_2\nu)^{m-1}\delta\theta_3\theta_1\phi\epsilon\psi\pi_1\pi_2\zeta\lambda_2$, $\newline$ 
$\tau\nu=a(\theta_3\theta_1\phi\epsilon\psi\pi_1\pi_2\zeta\lambda_2\nu\delta)^{m-1}\theta_3\theta_1\phi\epsilon\psi\pi_1\pi_2\zeta\lambda_2\nu$, $\newline$ 
$\alpha\tau=0$, $\tau\beta=0$, $\delta\tau\theta_2=0$, $\tau\theta_2\pi_2=0$, $\delta\theta_3\pi_3=0$,

\item $\phi\epsilon=\gamma\eta+a(\pi_1\pi_2\zeta\lambda_2\nu\delta\theta_3\theta_1\phi\epsilon\psi)^{m-1}\pi_1\pi_2\zeta\lambda_2\nu\delta\theta_3\theta_1\phi\epsilon$, $(\phi\epsilon\psi\pi_1\pi_2\zeta\lambda_2\nu\delta\theta_3\theta_1)^{m}\pi_1=0$, 

$\epsilon\psi=\rho\omega+a(\epsilon\psi\pi_1\pi_2\zeta\lambda_2\nu\delta\theta_3\theta_1\phi)^{m-1}
\epsilon\psi\pi_1\pi_2\zeta\lambda_2\nu\delta\theta_3\theta_1$, 

$\psi\phi=a(\psi\pi_1\pi_2\zeta\lambda_2\nu\delta\theta_3\theta_1\phi\epsilon)^{m-1}\psi\pi_1\pi_2\zeta\lambda_2\nu\delta\theta_3\theta_1\phi$, $\gamma\sigma=\phi\rho$, $\sigma\omega=\eta\psi$, 

$\omega\gamma=0$, $\omega\phi=0$, $\psi\gamma=0$, $\psi\pi_1\lambda_1=0$, 
$\phi\epsilon\psi\phi=0$, $\epsilon\psi\phi\epsilon=0$, $\psi\phi\epsilon\psi=0$,   

\item $\zeta^2\lambda_2=0$, $\pi_1\lambda_1\pi_3=0$, $\lambda_1\theta_1\phi=0$, $\theta_1\pi_1\pi_2=0$, 
$\pi_2\lambda_2\nu=0$, $\lambda_2\theta_2\lambda_1=0$, $\theta_2\pi_2\zeta=0$, 
$\pi_3\lambda_3\tau=0$, $\lambda_3\theta_3\theta_1=0$, $\theta_3\pi_3\xi=0$, 
$\xi\kappa^2=0$, $\kappa\mu\lambda_3=0$, and $\mu\xi\mu=0$ if $p\geqslant 3$, 

\item $\zeta\lambda_2\theta_2=0$, $\pi_1\pi_2\lambda_2=0$, $\lambda_1\pi_3\lambda_3=0$, 
$\pi_2\zeta^2=0$, $\theta_2\lambda_1\theta_1=0$, $\pi_3\xi\kappa=0$, $\lambda_3\tau\nu=0$, $\theta_3\theta_1\pi_1=0$, 
$\mu\lambda_3\theta_3=0$, $\kappa^2\mu=0$, and $\xi\mu\xi=0$ if $p\geqslant 3$.  
\end{enumerate} 
\end{exmp} 

\section{Basis and dimension}\label{sec:3+} 

This separate section is devoted to present precise formula for the dimension of a weighted generalized 
triangulation algebra. In particular, we will describe a canonical basis. We fix a weighted generalized 
triangulation algebra $\Lambda=\Lambda(Q,*,m_\bullet,c_\bullet,b_\bullet)$, where $Q=glue(B_1,\dots,B_r;\Theta)$ 
is a glueing of blocks $B_1,\dots,B_r$ with $s$ blocks of type IV and $t$ blocks of type V denoted as in 
the previous section. Let also $f$ and $g=\bar{f}$ be the permutations of arrows of the associated quiver $Q^*$. 
We assume from now on that $Q$ has at least two vertices. 

Our first aim is to describe the socles of the indecomposable projective modules in $\mod\Lambda$. Recall that 
for any $\eta\in Q_1^*$ we have the associated cycles $B_\eta$ and paths $A_\alpha$ along $g$-orbits of $\alpha$. 
We may extend this definition to some particular arrows outside of $Q^*$. Namely, for any $i\in\{1,\dots,s\}$ and 
$j\in\{1,\dots,t\}$, let $C^*_{\delta_i}$, $C^*_{\epsilon_j}$ and $C^*_{\psi_j}$ denote paths in $Q^*$ (along the 
$g$-cycles) satisfying the following conditions: 
$$\delta_i C^*_{\delta_i}=A_{\delta_i},\qquad\epsilon_j C^*_{\epsilon_j}=A_{\epsilon_j}\qquad\mbox{and}\qquad\psi_j C^*_{\psi_j}=A_{\psi_j}.$$ 
Then we may define the paths $A_{\alpha_i}=\alpha_i C^*_{\delta_i}$, $A_{\eta_j}=\eta_j C^*_{\epsilon_j}$ and 
$A_{\omega_j}=\omega_j C^*_{\psi_j}$, which naturally give the following ones 
$$B_{\alpha_i}=A_{\alpha_i}\beta_i,\quad B_{\eta_j}:=A_{\eta_j}\gamma_j\quad\mbox{and}\quad B_{\omega_j}=A_{\omega_j}\rho_j$$ 
in $Q$. Observe that $B_{\alpha_i}$ is a cycle in $Q$ of length $m_{\delta_i}n_{\delta_i}$ and both cycles $B_{\eta_j}$ and $B_{\omega_j}$ (in $Q$) are of length $m_{\psi_j}n_{\psi_j}$. Note that the definition of $B_{\alpha_i}$ is 
the same as for virtual mutations, where we consider only blocks of type IV. 

For an arrow $\eta\in Q_1^*$, we write 
$$\Delta_\eta=\eta f(\eta)f^2(\eta),\quad\zeta_\eta=\eta f(\eta)g(f(\eta))\mbox{ and }\xi_\eta=\eta g(\eta)f(g(\eta)).$$
Clearly $\Delta_\eta\bar{\eta}=\eta\zeta_{f(\eta)}$. We will also use the notation $\rho\equiv\rho'$ if the 
$\rho=\lambda\rho'$, for a nonzero scalar $\lambda\in K$. \smallskip 

Let us start with some preparatory lemmas. 

\begin{lem}\label{prepL1} Let $\eta\in Q_1^*$ be an arrow not lying in a block of type I and $\eta\neq\phi_j$, for 
$j\in\{1,\dots,t\}$. Then $c_\eta B_\eta=\Delta_\eta=c_{\bar{\eta}}B_{\bar{\eta}}$. 

Moreover, for any $j\in\{1,\dots,t\}$, we have 
$$c_{\phi_j}B_{\phi_j}=\phi_k\epsilon_j\psi_j-\gamma_j\sigma_j\omega_j=c_{\overline{\phi_j}}B_{\overline{\phi_j}},$$ 
and for a border loop $\eta$ the following holds $\Delta_\eta=c_{\bar{\eta}}B_{\bar{\eta}}+b_{s(\eta)}B_\eta\eta$. 
\end{lem} 

\begin{proof} Let first $\eta$ be an arrow in a block of types II or III. Then we deduce from Definition \ref{df:3.1}.2) 
that $\Delta_\eta=(\eta f(\eta))f^2(\eta)=c_{\bar{\eta}}A_{\bar{\eta}}g^{-1}(\bar{\eta})=c_{\bar{\eta}}B_{\bar{\eta}}$. 
Similarily, using the relations 2) for the arrow $f(\eta)$, we obtain 
$\Delta_\eta=\eta(f(\eta)f(f(\eta)))=\eta(c_{\overline{f(\eta)}}A_{\overline{f(\eta)}})=c_{g(\eta)}\eta A_{g(\eta)}=
c_\eta B_{\eta}$. \smallskip 

Next, applying Definition \ref{df:3.1}.3), we conclude that for any $i\in\{1,\dots,s\}$ arrows $\delta:=\delta_i$, 
$\nu:=\nu_i$ and $\tau:=\tau_i$ satisfy the following equalities: 
\begin{itemize}
\item $\Delta_\delta=(\delta\tau)\nu=(c_{\bar{\delta}}A_{\bar{\delta}})\nu=c_{\delta}B_{\delta}=c_{\bar{\delta}}B_{\bar{\delta}}$; 
\item $\Delta_\tau=(\tau\nu)\delta=c_{\bar{\tau}}A_{\bar{\tau}}g^{-1}(\bar{\tau})=c_{\bar{\tau}}B_{\bar{\tau}}= 
\tau(\nu\delta)=\tau(\beta\alpha+c_{\bar{\nu}}A_{\bar{\nu}})=c_{g(\tau)}\tau A_{g(\tau)}=c_{\tau}B_{\tau}$; 
\item $\Delta_\nu=(\nu\delta)\tau=(\beta\alpha+c_{\bar{\nu}}A_{\bar{\nu}})\tau=c_{\bar{\nu}}A_{\bar{\nu}}g^{-1}(\bar{\nu})=
c_{\bar{\nu}}B_{\bar{\nu}}=\nu(\delta\tau)=\nu(c_{\delta}A_{\delta})=c_{\nu}B_{\nu}$. 
\end{itemize} 
Therefore, the required equality holds also for $\eta$ lying in a block of type IV. 

Now, using the relations 4) in Definition \ref{df:3.1}, we obtain that for any $j\in\{1,\dots,t\}$, the 
following is satisfied ($\phi:=\phi_j$, $\epsilon:=\epsilon_j$ and $\psi:=\psi_j$): 
\begin{itemize}
\item $\Delta_\epsilon=\epsilon(\psi\phi)=\epsilon(c_{\bar{\psi}}A_{\bar{\psi}})=c_{\epsilon}B_{\epsilon}=
c_{\bar{\epsilon}}B_{\bar{\epsilon}}$; 
\item $\Delta_\psi=(\psi\phi)\epsilon=c_{\bar{\psi}}A_{\bar{\psi}}\epsilon=c_{\psi}B_{\psi}=c_{\bar{\psi}}B_{\bar{\psi}}$. 
\end{itemize} 
This shows the first part of the claim. \smallskip 

Eventually, the second part of the claim follows immediately from the relations 1) and 4) in Definition 
\ref{df:3.1}. \end{proof} 

The next technical result describes exceptional behaviour of paths $\xi_\eta$. 

\begin{lem}\label{prepL2} Let $\alpha\in Q_1^*$ be an arrow such that $g^2(\alpha)$ is different from $\delta_i$, 
$i\in\{1,\dots,s\}$, and from $\epsilon_j$ and $\psi_j$, $j\in\{1,\dots,t\}$. Then  
\begin{enumerate} 
\item[a)] $\xi_\alpha=\alpha f(\alpha)$, if $f(\alpha)$ is virtual. Moreover, then $g^{-1}(\alpha)\xi_\alpha=\xi_{g^{-1}(\alpha)}=0$.  
\item[b)] $\xi_\alpha=\bar{\alpha}f(\bar{\alpha})$, if $f^2(\alpha)$ is virtual with $(m_{f(\alpha)},n_{f(\alpha)})=(1,3)$. 
In this case, we have $g^{-2}(\alpha)g^{-1}(\alpha)\xi_\alpha=\xi_{g^{-2}(\alpha)}f(\bar{\alpha})=0$.  
\item[c)] $\xi_\alpha=c_{\alpha} B_\alpha$, if $\alpha=\tau_k$ with $(m_{\nu_k},n_{\nu_k})=(1,3)$ or $\alpha=\nu_k$. 
\item[d)] $\xi_\alpha= c_{\alpha} A_\alpha$, if $\alpha=\delta_k$ and $\tau_k$ is virtual.  
\item[e)] $\xi_\alpha= c_\alpha B_{\alpha}$, for $\alpha=\epsilon_j$. In this case $g^{-1}(\alpha)\xi_\alpha=0$. 
\item[f)] $\xi_\alpha= 0$, otherwise. 
\end{enumerate} \end{lem} 

\begin{proof} Assume first that $\alpha$ belongs to a block of types I-III. By the assumption $g(\alpha)\neq \nu_i$ 
and $g(\alpha)\neq \phi_j$ (and $\epsilon_j$), hence it follows from Definition \ref{df:3.1}.6) that $\xi_\alpha=0$ 
unless $f(\alpha)$ is virtual or $f^2(\alpha)$ is virtual with $(m_{f(\alpha)},n_{f(\alpha)})=(1,3)$. 

In the first case, we get 
$\xi_\alpha=\alpha (g(\alpha)f(g(\alpha)))=\alpha(c_{\overline{g(\alpha)}}A_{\overline{g(\alpha)}})=\alpha A_{f(\alpha)}=
\alpha f(\alpha)$ and $g^{-1}(\alpha)\xi_\alpha=\xi_{g^{-1}(\alpha)}=0$, by the relations 6), because 
$f(g^{-1}(\alpha))=\bar{\alpha}$ is not virtual and if $f^2(g^{-1}(\alpha))$ is virtual, then $\bar{\alpha}$ belongs to 
a $g$-orbit of length $n_{\bar{\alpha}}\geqslant 4$. This shows a). \smallskip 

In the second case, the $g$-cycle of $f(\alpha)$ has length $3$ and $g(f(\alpha))=f(\bar{\alpha})$, so 
$$\xi_\alpha=\alpha (g(\alpha)f(g(\alpha)))=\alpha(c_{\overline{g(\alpha)}}A_{\overline{g(\alpha)}})=
\alpha A_{f(\alpha)}=\alpha f(\alpha) g(f(\alpha))=\bar{\alpha}f(\bar{\alpha}).$$ 
Clearly, we have $f(g^{-1}(\alpha))=\bar{\alpha}$, and hence $g^{-2}(\alpha)g^{-1}(\alpha)\xi_{\alpha}=
\xi_{g^{-2}(\alpha)}f(\bar{\alpha})$. Moreover, $\xi_{g^{-2}(\alpha)}=0$, because both $f(\alpha)$ and 
$f^2(\alpha)$ are not virtual, and we are done in b). \smallskip 

Items c)-d) deal with the case when $\alpha$ is in a block of type IV. Here we obtain the required 
equalities using the relations 3):  
\begin{itemize}
\item if $\alpha=\tau_k$, then $\xi_\alpha=0$ except $(m_{\nu_k},n_{\nu_k})=(1,3)$, and in the case: 
$\xi_\alpha=\tau_k g(\tau_k)f(g(\tau_k))=\tau_k(\overline{\nu_k}f(\overline{\nu_k}))=\tau_k A_{\nu_k}=\tau_k\nu_k\delta_k= 
\Delta_{\tau_k}=c_{\tau_k} B_{\tau_k}$ (see Lemma \ref{prepL1}); 
\item if $\alpha=\delta_k$, then $\xi_\alpha=0$ except $\tau_k$ is virtual, and in the case: 
$\xi_\alpha=\delta_k \overline{\tau_k} f(\overline{\tau_k})=\delta_k\tau_k=c_{\delta_k}A_{\delta_k}$; 
\item for $\alpha=\nu_k$, we get $g(\alpha)=f(\alpha)=\delta_k$, so $\xi_\alpha=\nu_k\delta_k\tau_k=
\Delta_{\nu_k}=c_{\nu_k}B_{\nu_k}$, by Lemma \ref{prepL1}. 
\end{itemize} 

We proved above that for any arrow in a block of types I-IV, we have either $\xi_\alpha=0$, or $\alpha$ satisfies 
one of the conditions described in a)-d). Finally, it remains to see that for $\alpha$ lying in a block of type 
V, we must have $\alpha=\epsilon_j$ or $\psi_j$ (since by the assumption $g^2(\alpha)\neq \psi_j$, so $\alpha$ cannot 
be equal to $\phi_j$). If $\alpha=\epsilon_j$, then $\xi_{\epsilon_j}=\epsilon_j\psi_j\phi_j=\Delta_{\epsilon_j}=
c_{\epsilon_j}B_{\epsilon_j}$ and $g^{-1}(\alpha)\xi_\alpha=0$, as it is one of the zero relations in Definition 
\ref{df:3.1}.4). This completes the proof. \end{proof} 

The following proposition is the first key observation. 

\begin{prop}\label{prepP1} $B_\eta\bar{\eta}=0$, for any arrow $\eta\in Q_1^*$. \end{prop} 

\begin{proof} If $\eta=\phi_j$, then $B_\eta\bar{\eta}$ is one of the relations in Definition \ref{df:3.1}.4). 
If $\eta=\epsilon_j$ or $\psi_j$, then we conclude from Lemma \ref{prepL1} that $B_\eta\bar{\eta}\equiv\Delta_\eta\bar{\eta}$, 
and hence $B_\eta\bar{\eta}=0$, because $\Delta_\eta\bar{\eta}=\epsilon_j\psi_j\phi_j\epsilon_j$ or 
$\psi_j\phi_j\epsilon_j\psi_j$ is again one of the relations in Definition \ref{df:3.1}.4). Therefore, the claim 
holds for $\eta$ in a block $B_{s+j}$ of type V, $j\in\{1,\dots,t\}$. \smallskip 

Next, suppose that $\eta$ is an arrow in a block $B_i$, $i\in\{1,\dots,s\}$, of type IV. Then $B_\eta\equiv\Delta_\eta$ 
and $\Delta_\eta\bar{\eta}$ is equal to one of the following paths: $\nu\delta\tau\bar{\nu}$, $\delta\tau\nu\delta$ or 
$\tau\nu\delta\bar{\tau}$, where we abbreviate $\nu:=\nu_i$, $\delta:=\delta_i$ and $\tau:=\tau_i$. The first path 
is equal to $\nu\zeta_\delta$, and it vanishes in $\Lambda$, since $\zeta_\delta$ is one of the relations in 3) 
from Definition \ref{df:3.1}. Similarly, we deduce that $\delta\tau\nu\delta=
\delta\tau(\beta\alpha+c_{\bar{\nu}}A_{\bar{\nu}})=\delta\tau A_{g(\tau)}=\zeta_\delta g^2(\tau)\cdots =0$. For the 
third zero relation, it remains to prove that $B_\tau\bar{\tau}=0$. Indeed, $\bar{\tau}=f(g^{-1}(\tau))$, so 
$\bar{B}:=B_\tau\bar{\tau}=\tau g(\tau) \cdots g^{-3}(\tau)\xi_\alpha$, where $\alpha=g^{-2}(\tau)$. According to 
Lemma \ref{prepL2}, we conclude that either $\xi_\alpha=0$ or $\alpha$ satisfies one of the conditions a)-e). 

For a) or b), the path $\bar{B}$ admits a subpath $g^{-1}(\alpha)\xi_\alpha=0$ or $g^{-2}(\alpha)g^{-1}(\alpha)\xi_\alpha$, 
which are zero in $\Lambda$. The same works in case e). If $\alpha=\tau_k$ with $(m_{\nu_k},n_{\nu_k})=(1,3)$, then 
$\bar{B}$ is a path along $g$-orbit of $\alpha$ which contains a subpath of the form $B_{\delta_i}\delta_i\equiv 
\Delta_{\delta_i}\delta_i$, which is zero by previous considerations. If $\alpha=\nu_k$ or $\delta_k$ (with $\tau_k$ virtual), 
then $\bar{B}$ contains a subpath $B_{\delta_k}\delta_k=0$, and we are done in cases c) and d). \smallskip 

Assume now that $\eta$ is an arrow in a block of types II or III. Here we get $B_\eta\bar{\eta}\equiv\Delta_\eta\bar{\eta}=
\eta\zeta_{f(\eta)}$ and by Definition \ref{df:3.1}.5), we have $\zeta_{f(\eta)}=0$ in $\Lambda$ if and only if $f^2(f(\eta))
=\eta$ is not virtual and $f(\overline{f(\eta)})=f(g(\eta))$ is not virtual with $(m_{g(\eta)},n_{g(\eta)})=(1,3)$. 
If $\eta$ is virtual, then $B_\eta\bar{\eta}\equiv\Delta_\eta\bar{\eta}=\eta (f(\eta) f^2(\eta))\bar{\eta}=
\eta A_{g(\eta)} f(g(\eta))=\xi_\eta$, and $\xi_\eta=0$, by Definition \ref{df:3.1}.6), because $f(\eta)$ and $f^2(\eta)$ 
are not virtual. In case $f(g(\eta))$ is virtual with $A_{g(\eta)}=g(\eta)g^2(\eta)$ we obtain the following equalities: 
$$B_\eta\bar{\eta}\equiv \Delta_\eta\bar{\eta}=\eta (f(\eta)f^2(\eta))\bar{\eta}=\eta A_{g(\eta)} f(g^2(\eta))=
\eta g(\eta) A_{\overline{g^2(\eta)}}=\eta g(\eta) f(g(\eta)).$$ 
Therefore, again $B_\eta\bar{\eta}\equiv\xi_\eta$, which must be zero by Definition \ref{df:3.1}.6), since in this 
case both $f(\eta)$ and $f^2(\eta)$ are not virtual. \smallskip 

Finally, let $\eta$ be the loop from a block of type I. In particular, then $g(\eta)=\bar{\eta}$ so $\eta$ and 
$\bar{\eta}$ belong to one $g$-orbit. With notation $c:=c_\eta$ and $b=b_{s(\eta)}$, we deduce from Lemma \ref{prepL1} 
that $\Delta_\eta=c B_\eta + b B_\eta\eta$ and $B_{\bar{\eta}}\equiv \Delta_{\bar{\eta}} \equiv B_{\eta}$, since 
$\bar{\eta}$ does not lie in a block of type I (by the assumption, $Q$ has at least two vertices). Hence it 
follows from previous considerations that $B_\eta\eta\equiv B_{\bar{\eta}}\eta=0$, so in fact $\Delta_\eta\equiv B_\eta$. 
As a result, we conclude that $B_\eta\bar{\eta}\equiv \Delta_\eta\bar{\eta}=\eta^3\bar{\eta}=\eta\zeta_\eta$, thus 
$B_\eta\bar{\eta}=0$, because $\zeta_\eta$ is one of the relations in Definition \ref{df:3.1}.5). \end{proof} 

The following result provides the description of the socles of projective $\Lambda$-modules. 

\begin{coro}\label{lem:3.4}\begin{enumerate}[(1)] 
\item For any $\eta\in Q^*_1$ the cycle $B_{\eta}$ belongs to the right (left) socle of $\Lambda$. 
\item For any $i\in\{1,\dots,s\}$ the cycle $B_{\alpha_i}$ belongs to the right (left) socle of $\Lambda$. 
\item For $j\in\{1,\dots,t\}$, $B_{\eta_j}$ and $B_{\omega_j}$ belong to the right (left) socle of $\Lambda$. 
\end{enumerate}\end{coro} 

\begin{proof} Using Lemma \ref{prepL1} and Proposition \ref{prepP1} we easily deduce that also $B_\eta\eta=0$, 
for any arrow $\eta\in Q_1^*$, because $B_\eta \equiv B_{\bar{\eta}}$. Hence, we have $B_\eta\eta=B_{\eta}\bar{\eta}=0$ 
for any $\eta\in Q_1^*$. Now, it remains to see that 
$$B_{\nu_i}\beta_i=0, \mbox{ } B_{\phi_j}\gamma_j =0 \mbox{ and }B_{\epsilon_j}\rho_j=0,$$ 
for all $i\in\{1,\dots,s\}$ and $j\in\{1,\dots,t\}$. Clearly, $c_{\phi_j}B_{\phi_j}=
\phi_j\epsilon_j\psi_j-\gamma_j\sigma_j\omega_j$, so $B_{\phi_j}\gamma_j=0$, by Definition \ref{df:3.1}.4). Similarly, 
since $c_{\epsilon_j}B_{\epsilon_j}=\Delta_{\epsilon_j}$, we deduce from the relations \ref{df:3.1}.4) that 
$c_{\epsilon_j}B_{\epsilon_j}\rho_j=\epsilon_j\psi_j\phi_j\rho_j=\epsilon_j\psi_j\gamma_j\sigma_j=0$. Finally, 
applying the relations of type 3) in $\Lambda$, we obtain 
$c_{\delta_i}A_{\delta_i}\beta_i=c_{\bar{\delta_i}}A_{\bar{\delta_i}}\beta_i=\delta_i\tau_i\beta_i=0$, hence 
$A_{\delta_i}\beta_i=0$, for any $i\in\{1,\dots,s\}$. Therefore, $B_{\nu_i}\beta_i=\nu_i A_{\delta_i}\beta_i$ 
is zero in $\Lambda$, and we are done. Consequently, for a given $\eta\in Q_1^*$ we have $B_\eta\alpha=0$, for 
all arrows $\alpha\in Q_1$ starting at $s(B_\eta)=t(B_\eta)$. This shows that $B_\eta$ is an element of the right 
socle of $\Lambda$. \smallskip 

For (2), it is sufficient to see that $B_{\alpha_i}\alpha_i=0$, since $\alpha_i$ is the unique arrow starting at 
$c_i=s(B_{\alpha_i})$. Indeed, using the relations 3) in $\Lambda$, we get $c_{\delta_i}A_{\alpha_i}\nu_i=
\alpha_i(c_{\delta_i}A_{g(\delta_i)})=\alpha_i(c_{\delta_i}A_{\bar{\tau_i}})=\alpha_i\tau_i\nu_i=0$. Hence 
$A_{\alpha_i}\nu_i=0$, and consequently, we conclude from the relations 3) that 
$$B_{\alpha_i}\alpha_i=A_{\alpha_i}\beta_i\alpha_i=A_{\alpha_i}\nu_i\delta_i - c_{\bar{\nu_i}}A_{\alpha_i}A_{\bar{\nu_i}}=
- c_{\bar{\nu_i}}A_{\alpha_i}A_{\bar{\nu_i}}.$$  
But the path $A_{\alpha_i}A_{\bar{\nu_i}}$ is also zero in $\Lambda$, because it admits a subpath of the form 
$\rho=\alpha g(\alpha)f(g(\alpha))$, where $\alpha=g^{-2}(\nu_i)$, and it follows from Lemma \ref{prepL2} that 
$\rho=0$ or $A_{\alpha_i}A_{\bar{\nu_i}}$ contains a subpath of the form $\xi_\beta=0$ or $B_{\delta_k}\delta_k=0$. \smallskip 

Eventually, we will prove (3). First, observe that:  
$$(*)\qquad A_{\epsilon_j}\gamma_j=0,\, A_{\eta_j}\phi_j=0,\, A_{\omega_j}\epsilon_j=0,\mbox{ and } A_{\psi_j}\rho_j=0,$$ 
for $j\in\{1,\dots,t\}$. Indeed, we have 
$c_{\epsilon_j} A_{\epsilon_j}\gamma_j=c_{\bar{\epsilon_j}} A_{\bar{\epsilon_j}}\gamma_j=(\epsilon_j\psi_j-\rho_j\omega_j\gamma_j=0$, 
because $\psi_j\gamma_j=0$ and $\omega_j\gamma_j=0$ in $\Lambda$, by the relations 4). Similarly, we get 
$c_{\psi_j}A_{\psi_j}\rho_j=c_{\bar{\psi_j}}A_{\bar{\psi_j}}\rho_j=\psi_j\phi_j\rho_j=\psi_j\gamma_j\sigma_j=0$. 
Moreover, it is easy to see that   
$$c_{\epsilon_j} A_{\eta_j}\phi_j= \eta_j(c_{\epsilon_j}C^*_{\epsilon_j}\phi_j)=\eta_j(c_{\psi_j}A_{\psi_j} )= 
\eta_j\psi_j\phi_j=\sigma_j\omega_j\phi_j=0 $$ 
and 
$$c_{\epsilon_j} A_{\omega_j}\epsilon_j=\omega_j(c_{\epsilon_j}C^*_{\psi_j}\epsilon_j)=\omega_j(c_{g(\psi_j)}A_{g(\psi_j)})= 
\omega_j(c_{\bar{\phi_j}}A_{\bar{\phi_j}})=\omega_j(\phi_j\epsilon_j-\gamma_j\eta_j)=0,$$ 
so (*) is proved. Now, note that the relations 4) imply $B_{\eta_j}\eta_j=\eta_jC^*_{\epsilon_j}\gamma_j\eta_j=
\eta_jC^*_{\epsilon_j}\phi_j\epsilon_j-c_{\bar{\phi_j}}\eta_j C^*_{\epsilon_j}A_{\bar{\phi_j}}=
\eta_j B_{\psi_j}-c_{\bar{\phi_j}}\eta_j C^*_{\epsilon_j}A_{\bar{\phi_j}}$, so $B_{\eta_j}\eta_j=0$, because 
$\eta_jB_{\psi_j}=\eta_j\psi_j\phi_j\epsilon_j=\sigma_j\omega_j\phi_j\epsilon_j=0$ and $C^*_{\epsilon_j}A_{\bar{\phi_j}}$ admits 
a subpath of the form $\nabla_{\alpha}$, where $\alpha=g^{-2}(\phi_j)$; see Lemma \ref{prepL2}. On the other hand, 
$B_{\eta_j}\sigma_j=\eta_jC^*_{\epsilon_j}\gamma_j\sigma_j=A_{\eta_j}\phi_j\rho_j=0$, by $(*)$. Therefore, $B_{\eta_j}$ 
belongs to the right socle. Finally, we have $B_{\omega_j}\omega_j= A_{\omega_j}\rho_j\omega_j=
\omega_j B_{\bar{\phi_j}}-c_{\epsilon_j}A_{\omega_j}A_{\epsilon_j}$, which is zero in $\Lambda$, since $B_{\bar{\phi_j}}$ 
is in the socle and $A_{\omega_j}A_{\epsilon_j}$ admits an initial subpath of the form $\omega_j B_{\bar{\phi_j}}=0$. 
As a result, we conclude that $B_{\omega_j}$ belongs to the right socle of $\Lambda$. \smallskip 

Using the dual relations in the opposite algebra $\Lambda^{\op}$, one can directly deduce that $B_\eta$, for 
$\eta\in Q_1^*$ or $\eta=\alpha_i,\eta_j,\omega_j$, is also an element of the left socle of $\Lambda$. The proof 
is now complete. \end{proof}

Now, let us present some notation needed to describe bases of indecomposable projective modules over $\Lambda$. Let 
$\eta$ be an arrow in $Q_1$ which is different from the arrows $\beta_i$, for $i\in\{1,\dots,s\}$ and the arrows 
$\gamma_j,\rho_j,\sigma_j$, for $j\in\{1,\dots,t\}$. Then we write $\mathcal{B}_\eta$ for the set of all proper 
initial submonomials of $B_\eta$, and $\mathcal{B}_\eta^\nu$ for the set of all paths of the form $u\beta_i$ if 
$u\nu_i$ is a proper initial submonomial of $B_\eta$, for $i\in\{1,\dots,i\}$. Finally, we denote by $\mathcal{B}_\eta^\phi$ 
the set of all paths $u\gamma_j$, if $u\phi_j$ is a proper initial submonomial of $B_\eta$, and all paths $u\gamma_j\sigma_j$, 
if $u\phi_j$ is a proper initial submonomial of $B_\eta$, except $u\phi_j=A_{\psi_j}$ or $A_{\omega_j}$, for 
$j\in\{1,\dots,t\}$. We abbreviate $\tilde{\mathcal{B}}_\eta:=\mathcal{B}_\eta\cup\mathcal{B}_\eta^\nu\cup\mathcal{B}_\eta^\phi$. 

\begin{rem} We suggest to compare the exceptional conditions in definitions of $\mathcal{B}^\nu_\eta$ and 
$\mathcal{B}^\phi_\eta$ (i.e. cases when $u\beta_i$, $u\gamma_j$ or $u\gamma_j\sigma_j$ is not counted to the set) with 
equalities used in the above proved Corollary \ref{lem:3.4}, especially equalities presented in $(*)$. In each case the path 
is not counted, it is either zero or a socle element $B_\eta$. \end{rem} 

\begin{lem}\label{prepL3} Let $\zeta^k_\alpha=\alpha \beta g(\beta) \cdots g^k(\beta)$, where $\alpha\in Q_1^*$, 
$\beta=f(\alpha)$ and $k\geqslant 1$. Then $\zeta^k_\alpha=0$ or it is a combination of paths along $g$-orbits 
in $Q^*$. \end{lem} 

\begin{proof} For $k=1$ we have $\zeta^k_\alpha=\zeta_\alpha$. Assume first that $\alpha$ is in a block of type V. 
If $\alpha=\phi_j$ or $\epsilon_j$, then $\zeta_\alpha$ is itself a path along $g$-orbit of $\alpha$. For $\alpha=\psi_j$, 
we obtain $\zeta_\alpha=\zeta_{\psi_j}=(\psi_j\phi_j)\epsilon_j=c_{\psi_j}A_{\psi_j}\epsilon_j=c_{\psi_j}B_{\psi_j}$, by 
the relations 4) in Definition \ref{df:3.1}, so the claim holds also in this case. \smallskip 

Let now $\alpha$ be an arrow in a block of types I-IV. Then it follows from the relations 3) and 5) that either 
$\zeta_\alpha=0$ or $\alpha=\tau_i$ or $\nu_i$ or $\alpha$ belongs to a block of types I-III and then one of the 
following holds: 
\begin{enumerate}
\item[a)] $f^2(\alpha)$ is virtual; 
\item[b)] $f(\bar{\alpha})$ is virtual with $(m_{\bar{\alpha}},n_{\bar{\alpha}})=(1,3)$. 
\end{enumerate} 
For $\eta=\tau_i$, one gets $\zeta_\alpha=\tau_i\nu_i\delta_i=c_{\tau_i}B_{\tau_i}$, and for $\alpha=\nu_i$, 
we have $\zeta_\alpha=\nu_i\delta_i g(\delta_i)$, so in both cases $\zeta_\alpha$ is a path along a $g$-orbit of 
$\alpha$. Therefore, we may assume that $\alpha$ is in a block of types I-III. \smallskip 

In case a) we have $\zeta_\alpha=(\alpha f(\alpha))g(f(\alpha))=\bar{\alpha}f(\bar{\alpha})=c_\alpha A_\alpha$, 
whereas in b), one deduce that $\zeta_\alpha=(\alpha f(\alpha))g(f(\alpha))=\bar{\alpha}g(\bar{\alpha})f(g(\bar{\alpha}))=
\bar{\alpha} f(\bar{\alpha})=c_\alpha A_\alpha$. In both cases $\zeta_\alpha$ is equal to a path along $g$-orbit 
of $\alpha$, as required. \smallskip 

Now, suppose that $k\geqslant 1$ and consider $\zeta_\alpha^{k+1}=\zeta_\alpha g^2(\beta)\cdots g^{k+1}(\beta)$. 
If $\alpha$ belongs to a block of types IV or V, then $\zeta^{k+1}_\alpha=0$ or it is a path along $g$-orbit. 
For $\alpha$ in a block of types I-III, we conclude from the relations 5) that $\zeta_\alpha$ is either zero, 
or $\zeta^{k+1}_\alpha=\bar{\alpha}f(\bar{\alpha})g^2(\beta)\cdots g^{k+1}(\beta)$. If $f^2(\alpha)$ is virtual, 
then $g(f(\bar{\alpha}))=g^2(\beta)$, and in this case $\zeta^{k+1}_\alpha=\zeta^k_{\bar{\alpha}}$. In the 
case $f(\bar{\alpha})$ is virtual with $(m_{\bar{\alpha}},n_{\bar{\alpha}})=(1,3)$, we obtain that 
$\zeta^{k+1}_\alpha=c_\alpha A_\alpha g^2(\beta)\cdots g^{k+1}(\beta)=c_\alpha B_\alpha g^3(\beta)\cdots $, 
which is $c_\alpha B_\alpha$, for $k=2$, and zero, for $k\geqslant 3$ (use that $B_\alpha$ is in the right socle). 
Therefore, the claim follows from the induction on $k\geqslant 1$. \end{proof} 

\begin{coro}\label{prepCo} Every path in $Q^*$ is a combination of paths along $g$-orbits. \end{coro} 

\begin{proof} If $\rho$ is a combination of paths along $g$-orbits with common source and target, and $\eta\in Q_1^*$ is an 
arrow with $t(\alpha)=s(\rho)$, then $\alpha\rho$ is a combination of paths along $g$-orbits and paths of the form 
$\zeta^k_\alpha$, so by Lemma \ref{prepL3}, this is also a combination of paths along $g$-orbits. Hence the claim follows 
immediately from the induction on the length of a path. \end{proof}

\begin{prop}\label{prop:3.5} Let $\Lambda=\Lambda(Q,*,m_\bullet,c_\bullet,b_\bullet)$ be a weighted generalized 
triangulation algebra and $x$ a vertex of $Q$. Then the following statements hold. 
\begin{enumerate}[(1)] 
\item If $x$ is a starting vertex of two arrows $\eta,\bar{\eta}\in Q_1^*$ which are not virtual, then the module 
$e_x\Lambda$ has a basis of the form $\mathcal{B}_x=\tilde{\mathcal{B}}_\eta\cup\tilde{\mathcal{B}}_{\bar{\eta}}\cup\{e_x,B_\eta\}$.

\item If $x$ is a starting vertex of two arrows $\eta,\bar{\eta}\in Q_1^*$ with $\bar{\eta}$ virtual, then the 
module $e_x\Lambda$ has a basis of the form $\mathcal{B}_x=\tilde{\mathcal{B}}_\eta\cup\{e_x,B_\eta,\eta f(\eta)\}$. 

\item If $x=c_i$ or $d_i$, for $i\in\{1,\dots,s\}$, and $\eta$ is the unique arrow in $Q$ starting at $x$, then the 
module $e_{x}\Lambda$ has a basis $\mathcal{B}_x=\tilde{\mathcal{B}}_{\eta}\cup\{e_x,B_{\eta}\}$. 

\item  If $x=x_{1j}$ or $y_{1j}$, for $j\in\{1,\dots,t\}$, and $\eta$ is the unique arrow in $Q$ starting at $x$, 
then the module $e_{x}\Lambda$ has a basis $\mathcal{B}_x=\tilde{\mathcal{B}}_{\eta}\cup\{e_x,B_{\eta}\}$. 

\item If $x=x_{2j}$, for $j\in\{1,\dots,t\}$, then the module $e_{x}\Lambda$ has a basis of the form 
$\mathcal{B}_x=\tilde{\mathcal{B}}_{\eta_j}\cup\{e_x,B_{\eta_j},\sigma_j\}$.  

\item If $x=y_{2j}$, for $j\in\{1,\dots,t\}$, then the module $e_{x}\Lambda$ has a basis of the form 
$\mathcal{B}_x=\tilde{\mathcal{B}}_{\epsilon_j}\cup\{e_x,B_{\epsilon_j},\rho_j\}$. 
\end{enumerate}
\end{prop} 

\begin{proof} We denote by $\mathcal{B}_0$ the set of idempotents $\mathcal{B}_0=\{e_i\}_{i\in Q_0}$ of $\Lambda$, and put 
$$\mathcal{B}_1:=\{\sigma_1,\dots,\sigma_t,\rho_1,\dots,\rho_t\}\cup\{\eta f(\eta);\mbox{ }\bar{\eta}\mbox{: virtual}\}.$$ 
Moreover, we write $\mathcal{B}$ for the union of sets $\mathcal{B}_\eta$, together with all socle elements $B_\eta$, for 
$\eta\neq\beta_i,\gamma_j,\rho_j,\sigma_j$. Similarly, let $\mathcal{B}^\nu$ (respectively, $\mathcal{B}^\phi$) denote the union 
of sets $\mathcal{B}_\eta^\nu$ (respectively, $\mathcal{B}^\phi_\eta$), for all $\eta\in Q_1$ such that $\mathcal{B}_\eta$ 
is defined. We set $\widetilde{\mathcal{B}}=\mathcal{B}_0\cup\mathcal{B}_1\cup\mathcal{B}\cup\mathcal{B}^\nu\cup\mathcal{B}^\phi$. 
It is clear from the definition, that $\tilde{\mathcal{B}}$ is a disjoint union of sets $\mathcal{B}_x$, $x\in Q_0$, 
defined in (1)-(6). We will prove that $\tilde{\mathcal{B}}$ is a basis of $\Lambda$. \smallskip 

First, it follows from Corollary \ref{prepCo} that each path $\rho$ in $Q^*$ can be written as a combination of 
paths in $\mathcal{B}_0\cup\mathcal{B}$. If $\rho$ is a path in $Q$, which does not pass through vertices $x_{1j},x_{2j}$, 
$j\in\{1,\dots,t\}$, then it is either a path in $Q^*$, or it passes through vertices $c_i$, $i\in\{1,\dots,s\}$. 
Hence, by relations 3), one gets that $\rho$ is either a combination of paths in $Q^*$ or a combination of paths 
in $\mathcal{B}^\nu$. Eventually, if $\rho$ is a path in $Q$ passing through $x_{1j}$ or $x_{2j}$, then it follows 
from relations 4) that $\rho$ is a combination of paths in $\tilde{\mathcal{B}}$. Therefore, we conclude that 
$\tilde{\mathcal{B}}$ generates $\Lambda$. \smallskip 

It remains to show that each $\mathcal{B}_x$, for $x\in Q_0$, is linearly independent. It is clear that for any 
$\eta\in Q_1$, $\eta\neq\beta_i,\gamma_j,\sigma_j,\rho_j$, the set $\mathcal{B}_\eta$ is linearly independent, 
since $B_\eta$ belongs to the (right) socle of $\Lambda$. In a similar way, one can prove that the set 
$\mathcal{B}_\eta\cup\mathcal{B}_{\bar{\eta}}$ is linearly independent, if $\eta,\bar{\eta}\in Q^*_1$ are both 
not virtual. Indeed, suppose that 
$$\sum\lambda_k\rho_k+\sum\mu_k\bar{\rho}_k=0,\leqno{(\dagger)}$$ 
where $\rho_k$ and $\bar{\rho}_k$ are initial submonomials of $B_\eta$ 
and $B_{\bar{\eta}}$, respectively, and $\lambda_k,\mu_k\in K$ with at least one of $\lambda$'s and one of $\mu$'s 
non-zero. Moreover, we may take the paths in increasing order of length, and without loss of generality, we assume 
that the length of $\bar{\rho}_1$ is greater or equal to the length of $\rho_1$. Now, let $\rho$ be the unique 
initial submonomial of $B_\gamma$, $s(\gamma)=t(\rho_1)$, such that $\rho\rho_1=B_\gamma$. Then each $\rho\rho_i$, 
for $i\geqslant 2$, is a path along $g$-orbit of $\eta$ of length $>m_\eta n_\eta=m_\gamma n_\gamma$, so 
$\rho\rho_i=B_\gamma\gamma\cdots=0$, because $B_\gamma$ belongs to the (right) socle of $\Lambda$. Premultiplying 
the euality $(\dagger)$ by $\rho$, we obtain the following: 
$$\lambda_1B_\gamma=-\sum\mu_i\rho\bar{\rho}_k.$$ 
It is now sufficient to see that each $\tilde{\rho}_k=\rho\bar{\rho}_k$ is either zero it is congruent 
with an initial subpath of $B_\gamma$. Indeed, observe that $\tilde{\rho}_k= 
\gamma\cdots\xi_{g^{-1}(\alpha)}g(\beta)\cdots g^l(\beta)$, where $l\geqslant 1$, $\beta=f(\alpha)=\bar{\eta}$ 
is the first arrow on $\bar{\rho}_k$ and $\alpha=f^{-1}(\bar{\eta})=g^{-1}(\eta)$ is the last arrow on $\rho$. 
If $\gamma=g^{-p}(\alpha)$ with $p\geqslant 2$, then applying Lemma \ref{prepL2}, we conclude that either 
$\tilde{\rho}_k=0$ or $\tilde{\rho}_k$ contains an initial subpath of the form $B_\gamma\gamma$, which is also 
zero, because $B_\gamma$ is in the socle. If $p=1$, then $\gamma=g^{-1}(\alpha)$, hence using Lemma \ref{prepL2} 
again, we infer that either $\xi_\gamma=0$ or $\xi_\gamma\equiv B_\gamma$ or $A_{\gamma=\delta_i}$, or 
$\xi_{\gamma}g(\beta)=\Delta_\gamma\equiv B_{\gamma}$. In each case, we get $\tilde{\rho}_k=0$ or it contains 
an initial subpath of the form $B_\gamma$ or of the form $A_{\delta_k}\bar{\nu}_k=0$ (see also the proof 
of Corollary \ref{lem:3.4}). In a similar way, for $p=0$, one can use Definition \ref{df:3.1}.5), and deduce 
that $\tilde{\rho}_i=\zeta_\gamma^k=0$ or it is congruent with an initial subpath of $B_\gamma$. As a result, 
one concludes that $B_\gamma=0$ or it is a combination of initial subpaths of $B_\gamma$, a contradiction 
(since $\mathcal{B}_\gamma\cup\{B_\gamma\}$ linearly independet). \smallskip

Hence it is sufficient to prove that for any $x,y\in Q_0$, the sets $e_x\mathcal{B}^\nu e_y$ and 
$e_x\mathcal{B}^\phi e_y$ are linearly independent. 

Let $x\in Q_0$ and denote by $\eta$ and $\bar{\eta}$ (possibly $\eta=\bar{\eta}$) the arrows in $Q$ such that 
$\mathcal{B}_x=\widetilde{\mathcal{B}_\eta}\cup\widetilde{\mathcal{B}_{\bar{\eta}}}\cup\{e_x,B_\eta\}\cup P$, 
where $P=\{\pi\}$ with $\pi\in\mathcal{B}_1$ or $P=\emptyset$. Observe that, if $e_x\mathcal{B}^\nu e_y$ is non-empty, 
then $y=c_i$, for some $i\in\{1,\dots,s\}$, and moreover, at least one of $B_\eta$ or $B_{\bar{\eta}}$ contains 
$\nu_i$. Without loss of generality, we may assume that both. We write $n,\bar{n},m$ and $\bar{m}$ for 
$n_\eta,n_{\bar{\eta}},m_\eta$ and $m_{\bar{\eta}}$, respectively. Then there are unique $p\in\{0,\dots,n-1\}$ 
and $\bar{p}\in\{0,\dots,\bar{n}-1\}$ such that $g^{p+1}(\eta)=\nu_i$ and $g^{\bar{p}+1}(\bar{\eta})=\nu_i$, 
and consequently, we obtain that $e_x\mathcal{B}^\nu e_y$ consists of paths of the form 
$\theta_k=\eta g(\eta)\cdots g^{kn+p}(\eta)\beta_i$, for $k\in\{0,\dots,m-1\}$, and paths of the form 
$\bar{\theta}_k=\bar{\eta} g(\bar{\eta})\cdots g^{k\bar{n}+\bar{p}}(\bar{\eta})\beta_i$, for 
$k\in\{0,\dots,\bar{m}-1\}$. Now, it follows that $e_x\mathcal{B}^\nu e_y$ is linearly independent. Indeed, 
suppose that $\sum_{k=0}^{m-1} \lambda_k\theta_k+\sum_{k=0}^{\bar{m}-1}\mu_k\bar{\theta}_k=0$, for some 
$\lambda_k,\mu_k\in K$, at least one non-zero. Then 
$$(*)\qquad\sum_{k=0}^{m-1} \lambda_k\theta_k\alpha_i+\sum_{k=0}^{\bar{m}-1}\mu_k\bar{\theta}_k\alpha_i=0,$$ 
and for any $k$, we get $\theta_k\alpha_i=\eta\cdots g^{kn+p}(\eta)\beta_i\alpha_i=
\eta g(\eta)\cdots g^{kn+p}(\eta)\nu_i\delta_i$, because 
$\beta_i\alpha_i=\nu_i\delta_i-c_{\bar{\nu_i}}A_{\bar{\nu_i}}$, by relations 3), and $\theta_kA_{\bar{\nu_i}}$ 
contains a subpath of the form $\xi_\beta$ with $\beta=g^{kn+p-1}(\eta)=g^{-2}(\nu_i)$ or it is of the form 
$\zeta_\eta$. Similarly as above, it follows from Lemma \ref{prepL2} (and its dual), that always 
$\theta_k A_{\bar{\nu_i}}=0$. The same works for $\bar{\theta}_k\alpha_i=
\bar{\eta}\cdots g^{k\bar{n}+\bar{p}}(\bar{\eta})\nu_i\delta_i$. As a result, $(*)$ contradicts linear 
independence of $\mathcal{B}_\eta\cup\mathcal{B}_{\bar{\eta}}$, because then 
$$\theta_k\alpha_i=\eta g(\eta) \cdots g^{kn+p+2}(\eta)\in\mathcal{B}_\eta\mbox{ and }
\bar{\theta}_k\alpha_i=\bar{\eta}\cdots g^{k\bar{n}+\bar{p}+2}(\bar{\eta}) \in\mathcal{B}_{\bar{\eta}}.$$ 
This proves linear independence of $e_x\mathcal{B}^\nu e_y$, for all $y\in Q_0$. 

Similarly, if $e_x\mathcal{B}^\phi e_y\neq\emptyset$, then $y=x_{1j}$ or $x_{2j}$, and the corresponding 
sets are independent in each case. For $y=x_{1j}$, as above basis of $e_x\mathcal{B}^\phi e_y$ consists of 
paths $\theta_k=\eta \cdots g^{kn+p}(\eta)\gamma_j\sigma_j$, $k\in\{0,\dots,m-1\}$ (and possibly analogous 
paths $\bar{\theta}_k$, $k\in\{0,\dots,\bar{m}-1\}$), and 
$$\theta_k\omega_j=\eta\cdots g^{kn+p}(\eta)\phi_j\epsilon_j\psi_j,$$ 
because $\gamma_j\sigma_j\omega_j=\phi_j\rho_j\omega_j=\phi_j\epsilon_j\psi_j-c(\phi_jA_{\epsilon_j})$, by 
the relations 3), and $\eta\cdots g^{kn+p}(\eta)\phi_jA_{\epsilon_j}=\eta\cdots g^{kn+p}(\eta)B_{\phi_j}=0$, 
since it is a path along $g$-cycle of $\eta$ of length $> mn$ ($m=m_{\eta}=m_{\phi_j}$) and $B_{\phi_j}$ is 
in the socle of $\Lambda$. Consequently, we obtain that $\theta_k\omega_j=\eta g(\eta)\cdots g^{kn+p+3}(\eta)$ 
is a path in $\mathcal{B}_\eta$ or $0$ (the same for $\bar{\theta}_k\omega_j$), and hence, any linear combination 
of paths in $e_x\mathcal{B}^\phi e_{x_{1j}}$ multiplied by $\omega_j$ (from the right side) gives rise to a 
combination of (pairwise distinct) paths in $\mathcal{B}_\eta\cup\mathcal{B}_{\bar{\eta}}$. Thus we are done 
in case $y=x_{1j}$. Case $y=x_{2j}$ follows immediately from the case $y=x_{1j}$, if we multiply aprropriate 
combinations by $\sigma_j$ from the right. 

Summing up, we proved that for any $x\in Q_0$, the sets $\mathcal{B}_\eta$, $\mathcal{B}^\nu_\eta$ and $\mathcal{B}^\phi_\eta$ are 
linearly independent, so their union $\widetilde{\mathcal{B}}_\eta$ is also linearly independent, because paths in these sets are 
ending at different sets of vertices. As a result, we conclude that $\mathcal{B}_x$ is linearly independent for all $x\in Q_0$, 
and therefore, $\widetilde{\mathcal{B}}$ is indeed a basis of $\Lambda$. This completes the proof. \end{proof} 

For an arrow $\eta\in Q_1^*$, let 
$$n^\nu_\eta:=|\{i\in\{1,\dots,s\}:\mbox{ }\mathcal{O}(\eta)=\mathcal{O}(\nu_i)\}|\mbox{ and }
n^\phi_\eta:=|\{j\in\{1,\dots,t\}:\mbox{ } \mathcal{O}(\eta)=\mathcal{O}(\phi_j)\}|.$$  

\begin{coro}\label{co:3.6} Let $\Lambda=\Lambda(Q,*,m_\bullet,c_\bullet,b_\bullet)$ be a weighted generalized 
triangulation algebra and $x$ a vertex of $Q$. 
\begin{enumerate}[(1)] 
\item If $x$ is a starting vertex of two arrows $\eta,\bar{\eta}\in Q^*_1$, then 
$$|\mathcal{B}_x|=m_\eta(n_\eta+n_\eta^\nu+2n_\eta^\phi)+ m_{\bar{\eta}}(n_{\bar{\eta}}+n_{\bar{\eta}}^\nu+2n_{\bar{\eta}}^\phi).$$

\item If $x=c_i$ or $d_i$, for $i\in\{1,\dots,s\}$, then 
$|\mathcal{B}_x|=m_{\delta_i}(n_{\delta_i}+n_{\delta_i}^\nu+2n_{\delta_i}^\phi)$. 

\item If $x=y_{1j}$ or $y_{2j}$, for $j\in\{1,\dots,t\}$, then 
$|\mathcal{B}_{x}|=m_{\psi_j}(n_{\psi_j}+n_{\psi_j}^\nu+2n_{\psi_j}^\phi)$. 

\item If $x=x_{1j}$ or $x_{2j}$, for $j\in\{1,\dots,t\}$, then 
$|\mathcal{B}_{x}|=m_{\psi_j}(n_{\psi_j}+n_{\psi_j}^\nu+2n_{\psi_j}^\phi)$. \end{enumerate} \end{coro} 

\begin{proof} For (1) we will consider two cases. First assume that $x$ is a starting vertex of two arrows 
$\eta,\bar{\eta}\in Q^*_1$, both not virtual. Then it follows from Proposition \ref{prop:3.5}(1) that 
$|\mathcal{B}_x|=|\tilde{\mathcal{B}}_\eta|+|\tilde{\mathcal{B}}_{\bar{\eta}}|+2$. Moreover, we have 
$|\mathcal{B}_\eta|=m_\eta n_\eta - 1$, and neither $x=d_i$ nor $x= y_{1j},y_{2j}$ or $x_{1j},x_{2j}$, so 
$|\mathcal{B}_\eta^\nu|=m_\eta n_\eta^\nu$ and $|\mathcal{B}_\eta^\phi|=2m_\eta n_\eta^\phi$. The sets 
$\mathcal{B}_\eta$, $\mathcal{B}_\eta^\nu$ and $\mathcal{B}_\eta^\phi$ are disjoint, and consequently, 
$|\tilde{\mathcal{B}}_\eta|=m_\eta(n_\eta+n_\eta^\nu+2n_\eta^\phi)-1$. We have similar formula for 
$|\tilde{\mathcal{B}}_{\bar{\eta}}|$, and hence the required equality holds. Now, let $\bar{\eta}$ be a 
virtual arrow of $Q^*$. In the case, using Proposition \ref{prop:3.5}(2) we may conclude that 
$|\mathcal{B}_x|=|\tilde{\mathcal{B}}_\eta|+3=m_\eta(n_\eta+n_\eta^\nu + 2n_\eta^\phi)+2$, so the required \
equality also holds, because for a virtual arrow $\bar{\eta}$ we have $m_{\bar{\eta}}n_{\bar{\eta}}=2$, while 
$n_{\bar{\eta}}^\nu= n_{\bar{\eta}}^\phi=0$. 

Next, it follows from (3) in Proposition \ref{prop:3.5} that, for any $i\in\{1,\dots,s\}$, we have the equalities 
$$|\mathcal{B}_{c_i}|=|\tilde{\mathcal{B}}_{\alpha_i}|+2=|\tilde{\mathcal{B}}_{\delta_i}|+2=|\mathcal{B}_{d_i}|,$$ 
due to the definition of the cycles $B_{\alpha_i}$, $i\in\{1,\dots,s\}$. Clearly, 
$|\mathcal{B}_{\delta_i}|=m_{\delta_i}n_{\delta_i}-1$, whereas $|\mathcal{B}_{\delta_i}^\nu|=
m_{\delta_i}n_{\delta_i}^\nu-1$, because $A_{\delta_i}\nu_i=B_{\delta_i}$ is not a proper submonomial of 
$B_{\delta_i}$ (in fact, $A_{\delta_i}\beta_i=0$). Finally, we have $|\mathcal{B}_{\delta_i}^\phi|=
2m_{\delta_i}n_{\delta_i}^\phi$, hence indeed, we obtain that $|\mathcal{B}_{c_i}|=|\mathcal{B}_{d_i}|=
m_{\delta_i}(n_{\delta_i}+n_{\delta_i}^\nu+2n_{\delta_i}^\phi)$, and (2) follows. 

Now, let $j\in\{1,\dots,t\}$. By Proposition \ref{prop:3.5}(4), we get that $|\mathcal{B}_{y_{1j}}|=
|\tilde{\mathcal{B}}_{\psi_j}|+2$. Simple analysis shows that $|\mathcal{B}_{\psi_j}|=m_{\psi_j}n_{\psi_j}-1$, 
$|\mathcal{B}_{\psi_j}^\nu|=m_{\psi_j}n_{\psi_j}^\nu$ and $|\mathcal{B}_{\psi_j}^\phi|=2m_{\psi_j}n_{\psi_j}^\phi-1$, 
because $A_{\psi_j}'\phi_j=A_{\psi_j}$. Hence, we obtain the required formula for $|\mathcal{B}_{y_{1j}}|$. 
By Proposition \ref{prop:3.5}(6), $|\mathcal{B}_{y_{2j}}|=|\tilde{\mathcal{B}}_{\epsilon_j}|+3$ and in this 
case, we have $|\mathcal{B}_{\epsilon_j}|=m_{\epsilon_j}n_{\epsilon_j}-1$, $|\mathcal{B}_{\epsilon_j}^\nu|=
m_{\epsilon_j}n_{\epsilon_j}^\nu$ and $|\mathcal{B}_{\epsilon_j}^\phi|=2(m_{\epsilon_j}n_{\epsilon_j}^\phi-1)$, 
since we count all paths of the form $u\gamma_k$ and all paths of the form $u\gamma_k\sigma_k$, except 
$u=A_{\epsilon_j}$ (then $u\phi_k=B_{\epsilon_j}$ is not proper). This proves (3). 

Finally, consider $j\in\{1,\dots,t\}$. Applying Proposition \ref{prop:3.5}(4) and (5), we obtain that 
$|\mathcal{B}_{x_{1j}}|=|\tilde{\mathcal{B}}_{\omega_j}|+2$ and $|\mathcal{B}_{x_{2j}}|=
|\tilde{\mathcal{B}}_{\eta_j}|+3$. Since $B_{\omega_j}$ is a cycle of length $m_{\psi_j}n_{\psi_j}$, we 
have $|\mathcal{B}_{\omega_j}|=m_{\psi_j}n_{\psi_j}-1$ and $|\mathcal{B}_{\omega_j}^\nu|=m_{\psi_j}n_{\psi_j}^\nu$. 
Moreover, $|\mathcal{B}_{\omega_j}^\phi|=2m_{\psi_j}n_{\psi_j}^\phi - 1$, because we do not count the path 
$u\gamma_j\sigma_j$, for which $u\phi_j=A_{\omega_j}$. This proves the formula for $|\mathcal{B}_{x_{1j}}|$. 
Similarly, we have $|\mathcal{B}_{\eta_j}|=m_{\psi_j}n_{\psi_j}-1$ and $|\mathcal{B}_{\eta_j}^\nu|=
m_{\psi_j}n_{\psi_j}^\nu$, while $|\mathcal{B}_{\eta_j}^\phi|=2(m_{\psi_j}n_{\psi_j}^\phi - 1)$, because 
$A_{\eta_j}\phi_j$ is even not a submonomial of $B_{\eta_j}$, which also gives the required equality, 
and the proof is now finished. \end{proof} 

As a result we obtain the following theorem. 

\begin{theorem}\label{thm:3.7} Let $\Lambda=\Lambda(Q,*,m_\bullet,c_\bullet,b_\bullet)$ be a weighted 
generalized triangulation algebra. 
Then 
$$\dim_K\Lambda=\sum_{\eta\in Q_1^*}m_\eta(n_\eta+n_\eta^\nu+2n_\eta^\phi) + 
\sum_{i=1}^s m_{\delta_i}(n_{\delta_i}+n_{\delta_i}^\nu+2n_{\delta_i}^\phi) + $$
$$+\sum_{j=1}^t 2m_{\psi_j}(n_{\psi_j}+n_{\psi_j}^\nu+2n_{\psi_j}^\phi).$$ 
\end{theorem} 

\begin{proof} First, applying (1)-(3) from Corollary \ref{co:3.6}, we obtain that 
$$\sum_{x\in Q_0^*}|\mathcal{B}_x|=\sum_{\eta\in Q^*_1}m_\eta(n_\eta+n_\eta^\nu+2n_\eta^\phi).$$ 
On the other hand, vertices which do not belong to $Q_0^*$, are of the form $c_i$, for some $i\in\{1,\dots,s\}$, 
or $x_{1j}$ or $x_{2j}$, for some $j\in\{1,\dots,t\}$, so the remaining summands in the above expression are 
provided by conditions (2) and (4) from Corollary \ref{co:3.6}. \end{proof}

\section{Weighted triangulation algebras}\label{sec:4} 
We review some basic facts on weighted triangulation algebras needed for the proof of our main theorem. 

Recall that a triangulation quiver is a pair $(Q,f)$, where $Q=(Q_0,Q_1,s,t)$ is a finite connected $2$-regular 
quiver and $f:Q_1\to Q_1$ is a permutation such that $s(f(\alpha))=t(\alpha)$ for any arrow $\alpha\in Q_1$, and 
$f^3$ is the identity on $Q_1$. Equivalently, a triangulation quiver is given by a block decomposable quiver 
$Q=glue(B_1,\dots,B_r;\Theta)$ with all blocks of types I-III, and $f$ is defined by rotations of arrows in these 
blocks. In particular, $Q$ is a generalized triangulation quiver $(Q,*)$ with empty marking $*$, and $Q=Q^*$ 
(see also Proposition \ref{prop:2.2}). We keep notation introduced in the previous section. 

Let $(Q,f)$ be a triangulation quiver, $g:Q_1\to Q_1$ the permutation such that $g(\alpha)=\overline{f(\alpha)}$ 
for $\alpha\in Q_1$, and $\mathcal{O}(g)$ the set of all $g$-orbits in $Q_1$. Moreover, let $\partial(Q,f)$ be 
the set of all border vertices of $Q_0$, in the sense we discussed for generalized triangulation quivers. 
Consider a weight function $m_\bullet:\mathcal{O}(g)\to\mathbb{N}^*$, a parameter function 
$c_\bullet:\mathcal{O}(g)\to K^*$, and a border function $b_\bullet:\partial(Q,f)\to K$, if $\partial(Q,f)$ is 
not empty. We also assume that $m_\bullet$ satisfies the same restrictions as in Section \ref{sec:3}. 

In this paper, a weighted triangulation algebra is defined as follows (originally, this was called the socle 
deformed weighted triangulation algebra \cite[see Section 8]{ES2}). 

\begin{df}\label{df:4.1} The \emph{weighted triangulation algebra} $\Lambda(Q,f,m_\bullet,c_\bullet,b_\bullet)$ is the quotient 
$KQ/I$ of the path algebra $KQ$ by the ideal $I=I(Q,f,m_\bullet,c_\bullet,b_\bullet)$ generated by the following relations:   
\begin{enumerate}[(1)] 
\item $\alpha^2-c_{\bar{\alpha}}A_{\bar{\alpha}}-b_{s(\alpha)}B_\alpha$, for all border loops $\alpha\in Q_1$,
\item $\alpha f(\alpha)-c_{\bar{\alpha}}A_{\bar{\alpha}}$, for all arrows $\alpha\in Q_1$ which are not border loops, 
\item $\alpha f(\alpha) g( f(\alpha) )$, for all arrows $\alpha\in Q_1$ unless $f^2(\alpha)$ is virtual or unless $f(\bar{\alpha})$ 
is virtual with $m_{\bar{\alpha}}=1$ and $n_{\bar{\alpha}}=3$, 
\item $\alpha g(\alpha) f( g(\alpha) )$, for all arrows $\alpha\in Q_1$ unless $f(\alpha)$ is virtual or unless $f^2(\alpha)$ is 
virtual with $m_{f(\alpha)}=1$ and $n_{f(\alpha)}=3$.
\end{enumerate} \end{df} 

Clearly, the paths $A_\alpha$ (respectively, $B_\alpha$) are defined in the same way as in the previous section. We also 
note that for a genralized triangulation quiver $(Q,*)$ with $Q=Q^*$, definitions \ref{def:2.1} and \ref{df:4.1} coincide, 
i.e. algebras $\Lambda(Q,*,m,c,b)$ and $\Lambda(Q,f,m,c,b)$ are isomorphic. The following is a consequence of results 
proved in \cite{ES5}.  

\begin{theorem}\label{thm:4.2} Let $\Lambda=\Lambda(Q,f,m_\bullet,c_\bullet,b_\bullet)$ be a weighted triangulation algebra other 
than singular disc, triangle, tetrahedral or spherical algebra. Then the following statements hold. 
\begin{enumerate}[(1)] 
\item $\Lambda$ is a finite-dimensional algebra with $\dim_K=\sum_{\mathcal{O}\in\mathcal{O}(g)}m_\mathcal{O}n^2_\mathcal{O}$. 
\item $\Lambda$ is a tame symmetric algebra of infinite representation type.
\item $\Lambda$ is a periodic algebra of period $4$. 
\end{enumerate} \end{theorem} 

We recall also the following description of indecomposable projective modules over a weighted triangulation algebra, established in 
\cite[Lemma 4.7]{ES5}. 

\begin{prop}\label{prop:4.3} Let $\Lambda=\Lambda(Q,f,m_\bullet,c_\bullet,b_\bullet)$ be a weighted triangulation algebra, $i$ a 
vertex of $Q$ and $\alpha,\bar{\alpha}$ the two arrows in $Q_1$ starting at $i$. Then the following statements hold. 
\begin{enumerate}[(1)] 
\item Assume $\alpha$ is virtual. Then the module $e_i\Lambda$ has a basis $\mathcal{B}_i$ formed by all initial submonomials of 
$B_{\bar{\alpha}}$ together with $e_i$ and $\bar{\alpha}f(\bar{\alpha})$.  

\item Assume that $\alpha$ and $\bar{\alpha}$ are not virtual. Then the module $e_i\Lambda$ has a basis $\mathcal{B}_i$ formed by 
all proper initial submonomials of $B_\alpha$ and $B_{\bar{\alpha}}$ together with $e_i$ and $B_{\alpha}$. 

\item We have the equalities 
$$\alpha f(\alpha) f^2(\alpha)=c_\alpha B_\alpha = c_{\bar{\alpha}} B_{\bar{\alpha}}=\bar{\alpha} f(\bar{\alpha}) f^2(\bar{\alpha})$$ 
and this element generates the socle of $e_i\Lambda$. 
\end{enumerate} \end{prop} 

We saw in Proposition \ref{prop:3.5} how the above result extends to the weighted generalized triangulation algebras. 

Now, let $\Lambda=\Lambda(Q,*,m_\bullet,c_\bullet,b_\bullet)$ be a weighted generalized triangulation algebra, and 
$Q=glue(B_1,\dots,B_r;\Theta)$. We shall associate to $\Lambda$ a weighted triangulation algebra 
$$\Lambda^\Delta=\Lambda(Q^\Delta,f^\Delta,m^\Delta_\bullet,c^\Delta_\bullet,b^\Delta_\bullet),$$ 
defined in a canonical way. We set $\Lambda^\Delta=\Lambda$ if $\Lambda$ is a weighted triangulation algebra, or 
equivalently, all blocks $B_1,\dots,B_r$ are of types I-III. So assume throughout that blocks of type IV or V occur 
in the family $B_1,\dots,B_r$. 

Refering to the notation introduced before, if $(Q,*)$ contains blocks of type IV, then these are the blocks $B_i$ 
of the form 
$$
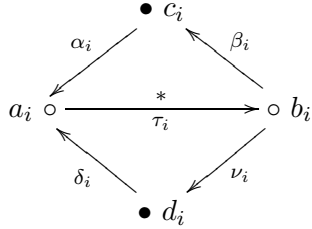
\xymatrix{ & \bullet \mbox{ }c_i  \ar[ld]_{\alpha_i} & \\ 
a_i\mbox{ }\circ\ar[rr]_{\tau_i}^{*} & & \circ\mbox{ } b_i \ar[lu]_{\beta_i}\ar[ld]^{\nu_i} \\ 
 & \bullet\mbox{ }d_i\ar[lu]^{\delta_i} & }$$ 
for $i\in\{1,\dots,s\}$, and $s=0$ if no block of type IV occurs. Similarily, if $(Q,*)$ contains blocks of 
type V, then these are the blocks $B_{s+j}$ of the form 
$$
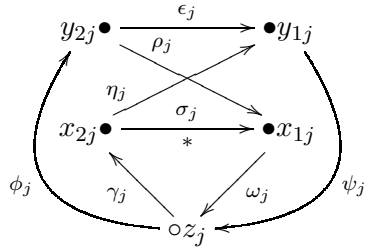
\xymatrix@C=0.5cm{&y_{2j}\bullet\ar[rrd]^(0.33){\rho_j}\ar[rr]^{\epsilon_j}&&\bullet y_{1j}\ar@/^{40pt}/[ldd]^{\psi_j}& \\ 
&x_{2j}\bullet\ar[rr]^{\sigma_j}_{*}\ar[rru]^(0.2){\eta_j}&&\bullet x_{1j}\ar[ld]^{\omega_j}& \\ 
&&\circ z_j\ar[lu]^{\gamma_j}\ar@/^{40pt}/[luu]^{\phi_j}&& }$$ 
for $j\in\{1,\dots,t\}$, where $t=0$ if no block of type V occurs in $(Q,*)$. Now, we define the triangulation 
quiver $(Q^\Delta,f^\Delta)$, which is obtained from $Q$ in the following way. 
\begin{enumerate}[(1)] 
\item We replace any block $B_i$, $i\in\{1,\dots,s\}$, of type IV by the following quiver $B_i^\Delta$ 
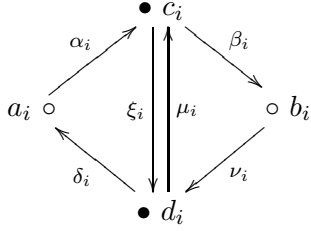
$$\xymatrix{ & \bullet \mbox{ }c_i \ar@<-0.1cm>[dd]_{\xi_i}  \ar[rd]^{\beta_i} & \\ 
a_i\mbox{ }\circ\ar[ru]^{\alpha_i}  & & \circ\mbox{ } b_i \ar[ld]^{\nu_i} \\ 
 & \bullet\mbox{ }d_i\ar[lu]^{\delta_i}\ar@<-0.1cm>[uu]_{\mu_i} & }$$ 
being gluing of two blocks of type II, and hence we have the $f^\Delta$-orbits $(\alpha_i\mbox{ }\xi_i\mbox{ }\delta_i)$ and 
$(\beta_i\mbox{ }\nu_i\mbox{ }\mu_i)$. 

\item We replace any block $B_{s+j}$, $j\in\{1,\dots,t\}$, of type V by the following quiver $B^\Delta_{s+j}$ 
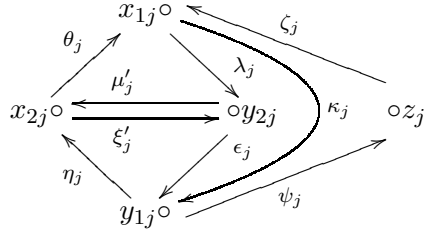
$$\xymatrix@C=0.5cm{&x_{1j}\circ\ar[rd]^(0.75){\lambda_j}\ar@/^{65pt}/[dd]^{\kappa_j} &&&\\ 
x_{2j}\circ\ar[ru]^{\theta_j}\ar@<-0.1cm>[rr]_(0.4){\xi_j'} & & 
\circ y_{2j}\ar[ld]^(0.25){\epsilon_j}\ar@<-0.1cm>[ll]_(0.6){\mu_j'} & &\circ z_j\ar@<-0.25cm>[lllu]_{\zeta_j}\\
&y_{1j}\circ\ar[lu]^{\eta_j}\ar@<-0.25cm>[rrru]_{\psi_j} &&& }$$ 
being gluing of three blocks of type II, and hence we have the $f^\Delta$-orbits $(\eta_j\mbox{ }\xi_j'\mbox{ }\epsilon_j)$, 
$(\theta_j\mbox{ }\lambda_j\mbox{ }\mu_j')$, and $(\psi_j\mbox{ }\zeta_j\mbox{ }\kappa_j)$. 

\item We set $f^\Delta(\alpha)=f(\alpha)$ for any arrow $\alpha\in Q_1$ which belongs to blocks of types I-III.
\end{enumerate} 

Observe that $Q^\Delta=glue(B_1^\Delta,\dots,B_r^\Delta)$ and $(Q^\Delta,f^\Delta)$ is a triangulation quiver with 
permutation $f^\Delta$ given as above (clearly, we put $B^\Delta_k:=B_k$, for $k>s+t$). Clearly, we have 
$\partial(Q^\Delta,f^\Delta)=\partial(Q,*)$. We denote by $g^\Delta:Q^\Delta_1\to Q^\Delta_1$ the permutation 
associated to $f^\Delta$, that is, $g^\Delta(\alpha)=\overline{f^\Delta(\alpha)}$ for any arrow $\alpha\in Q^\Delta_1$. 
Further, we denote by $\mathcal{O}^\Delta(\alpha)$ the $g^\Delta$-orbit of an arrow $\alpha\in Q^\Delta_1$, any by 
$\mathcal{O}(g^\Delta)$ the set of all $g^\Delta$-orbits in $Q^\Delta_1$. 

The set $\mathcal{O}(g^\Delta)$ consists of the following three types of $g^\Delta$-orbits: 
\begin{enumerate}[(1)] 
\item $\mathcal{O}^\Delta(\xi_i)=(\xi_i\mbox{ }\mu_i)$ in the quivers $B^\Delta_i$, for $i\in\{1,\dots,s\}$, 

\item $\mathcal{O}^\Delta(\xi_j')=(\xi_j'\mbox{ }\mu_j')$ and $\mathcal{O}^\Delta(\theta_j)=
(\theta_j\mbox{ }\kappa_j\mbox{ }\eta_j)$ in the quivers $B^\Delta_{s+j}$, for $j\in\{1,\dots,t\}$, and 

\item the $g^\Delta$-orbits obtained from the $g$-orbits in $Q^*_1$ by replacing the arrows $\tau_j$ by the 
paths $\alpha_i\beta_i$, for $i\in\{1,\dots,s\}$, and the arrows $\phi_j$ by the paths $\zeta_j\lambda_j$, for $j\in\{1,\dots,t\}$.
\end{enumerate} 

We may also define the weight function $m^\Delta_\bullet:\mathcal{O}(g^\Delta)\to\mathbb{N}^*$, the parameter 
function $c^\Delta_\bullet:\mathcal{O}(g^\Delta)\to K^*$, and the border function 
$b^\Delta_\bullet:\partial(Q^\Delta,f^\Delta)\to K$ as described below. 

We set 
$$m^\Delta_{\mathcal{O}^\Delta(\xi_i)}=1, m^\Delta_{\mathcal{O}^\Delta(\xi_j')}=1, m^\Delta_{\mathcal{O}^\Delta(\theta_j)}=1, 
m^\Delta_{\mathcal{O}^\Delta(\alpha_i)}=m_{\mathcal{O}(\tau_i)},$$ 
$$c^\Delta_{\mathcal{O}^\Delta(\xi_i)}=1, c^\Delta_{\mathcal{O}^\Delta(\xi_j')}=1, c^\Delta_{\mathcal{O}^\Delta(\theta_j)}=1, 
c^\Delta_{\mathcal{O}^\Delta(\alpha_i)}=c_{\mathcal{O}(\tau_i)},$$
for $i\in\{1,\dots,s\}$ and $j\in\{1,\dots,t\}$. The remaining $g^\Delta$-orbits in $Q^\Delta_1$ are of the form 
$\mathcal{O}^\Delta(\eta)$ for some arrow $\eta\in Q^*_1$, and we put 
$$m^\Delta_{\mathcal{O}^\Delta(\eta)}=m_{\mathcal{O}(\eta)}\quad\mbox{and}\quad c^\Delta_{\mathcal{O}^\Delta(\eta)}=
c_{\mathcal{O}(\eta)}.$$ 

Moreover, for any vertex $i\in\partial(Q^\Delta,f^\Delta)=\partial(Q,*)$ we set $b^\Delta_i=b_i$. 

\begin{exmp}\label{ex:4.4} Let $(Q,*)$ be the generalized triangulation quiver considered in Example 
\ref{ex:3.2}. Then the associated triangulation 
quiver $(Q^\Delta,f^\Delta)$ is of the form 
$$
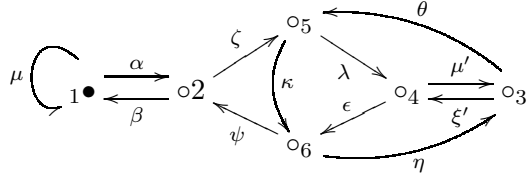
\xymatrix@R=0.01cm{
&&&\circ_5\ar[rddd]_{\lambda}\ar@/_10pt/[ddddd]^{\kappa}&&\\
&&&&&\\
&\ar@(lu,ld)[d]_{\mu}&&&&\\
& _{1}\bullet\ar@<0.2cm>[r]^{\alpha}  &\circ 2 \ar@<0.1cm>[l]^{\beta}\ar[ruuu]^{\zeta}&& 
\circ_4\ar[ldd]_{\epsilon}\ar@<0.1cm>[r]^{\mu '}& 
\circ_3\ar@/_10pt/@<-0.1cm>[lluuu]_{\theta}\ar@<0.1cm>[l]^{\xi '} \\  
&&&&&\\
&&&\circ_6\ar@/_10pt/@<-0.1cm>[rruu]_{\eta}\ar[luu]^{\psi}&&\\} $$ 
with $f^\Delta$-orbits: $(\eta\mbox{ }\xi '\mbox{ }\epsilon)$, $(\theta\mbox{ }\lambda\mbox{ }\mu ')$, 
$(\psi\mbox{ }\zeta\mbox{ }\kappa)$, $(\alpha\mbox{ }\beta\mbox{ }\mu)$. Then $\mathcal{O}(g^\Delta)$ consists 
of the following $g^\Delta$-orbits: 
$$\mathcal{O}^\Delta(\xi ')=(\xi '\mbox{ }\mu '),\mathcal{O}^\Delta(\theta)=(\theta\mbox{ }\kappa\mbox{ }\eta),$$ 
$$\mathcal{O}^\Delta(\alpha)=(\alpha\mbox{ }\zeta\mbox{ }\lambda\mbox{ }\epsilon\mbox{ }\psi\mbox{ }
\beta), \mathcal{O}^\Delta(\mu)=(\mu).$$  
Clearly, the border $\partial(Q^\Delta,f^\Delta)$ is empty. 

Further, if $m_\bullet:\mathcal{O}(g)\to \mathbb{N}^*$ and $c_\bullet:\mathcal{O}(g)\to K^*$ are weight and 
parameter functions of $(Q,*)$, then the induced weight and parameter functions 
$m^\Delta_\bullet:\mathcal{O}(g^\Delta)\to \mathbb{N}^*$ and $c^\Delta_\bullet:\mathcal{O}(g^\Delta)\to K^*$ 
of $(Q^\Delta,f^\Delta)$ are given by the following formulas: 
$$m^\Delta_{\mathcal{O}^\Delta(\xi ')}=1, m^\Delta_{\mathcal{O}^\Delta(\theta)}=
1,m^\Delta_{\mathcal{O}^\Delta(\alpha)}=m_{\mathcal{O}(\alpha)}, 
m^\Delta_{\mathcal{O}^\Delta(\mu)}=m_{\mathcal{O}(\mu)},$$
$$c^\Delta_{\mathcal{O}^\Delta(\xi ')}=1, c^\Delta_{\mathcal{O}^\Delta(\theta)}=1,
c^\Delta_{\mathcal{O}^\Delta(\alpha)}=c_{\mathcal{O}(\alpha)}, 
c^\Delta_{\mathcal{O}^\Delta(\mu)}=c_{\mathcal{O}(\mu)}.$$ 
We set, as in Example \ref{ex:3.2}, $m=m_{\mathcal{O}(\alpha)}$, $n=m_{\mathcal{O}(\mu)}$, 
$c=c_{\mathcal{O}(\alpha)}$, $d=c_{\mathcal{O}(\mu)}$. 

Then the associated weighted triangulation algebra $\Lambda^\Delta=
\Lambda(Q^\Delta,f^\Delta,m_\bullet^\Delta,c_\bullet^\Delta)$ is defined by the quiver $Q^\Delta$ and the relations: 
$$\alpha\beta=d\mu^{n-1},\beta\mu=c(\zeta\lambda\epsilon\psi\beta\alpha)^{m-1}\zeta\lambda\epsilon\psi\beta,
\mu\alpha=c(\alpha\zeta\lambda\epsilon\psi\beta)^{m-1}\alpha\zeta\lambda\epsilon\psi,$$
$$\zeta\kappa=c(\beta\alpha\zeta\lambda\epsilon\psi)^{m-1}\beta\alpha\zeta\lambda\epsilon,
\kappa\psi=c(\lambda\epsilon\psi\beta\alpha\zeta)^{m-1}\lambda\epsilon\psi\beta\alpha,\psi\zeta=\eta\theta,$$ 
$$\eta\xi '=c(\psi\beta\alpha\zeta\lambda\epsilon)^{m-1}\psi\beta\alpha\zeta\lambda,
\xi '\epsilon=\theta\kappa,\epsilon\eta=\mu ',$$
$$\theta\lambda=\xi ',\lambda\mu '=\kappa\eta,\mu '\theta=c(\epsilon\psi\beta\alpha\zeta\lambda)^{m-1}\epsilon\psi\beta\alpha\zeta,$$ 
$$\zeta\kappa\eta=0,\kappa\psi\beta=0,\eta\xi '\mu '=0,\lambda\mu '\xi '=0,\xi '\epsilon\psi=0,$$
$$\mu '\theta\kappa=0,\beta\mu^2=0,\mu\alpha\zeta=0,\mbox{ and }\alpha\beta\alpha=0\mbox{ if }n\geqslant 3$$
$$\zeta\lambda\mu '=0,\kappa\eta\xi '=0,\psi\beta\mu=0,\xi '\mu '\theta=0,\mu '\xi '\epsilon=0,$$ 
$$\theta\kappa\psi=0,\mu^2\alpha=0,\alpha\zeta\kappa=0,\mbox{ and }\beta\alpha\beta=0\mbox{ if }n\geqslant 3,$$
Observe that for $n=2$ the loop $\mu$ is virtual. For example, there are no zero relations of the forms $\psi \zeta \lambda=0$ and 
$\epsilon\psi\zeta=0$, since $f(\bar{\psi})=\xi '$ is virtual with $m^\Delta_{\bar{\psi}}=1$ and $n^\Delta_{\bar{\psi}}=3$, whereas 
$f^2(\epsilon)=\xi '$ is virtual with $m^\Delta_{f(\epsilon)}=1$ and $n^\Delta_{f(\epsilon)}=3$ (note that obviously, these are not 
all relations excluded in items 3)-4) of Definition \ref{df:4.1}). Due to Theorem \ref{thm:4.2}, $\Lambda^\Delta$ is of dimension: 
$$\dim_K\Lambda^\Delta=\sum_{\mathcal{O}^\Delta\in\mathcal{O}(g^\Delta)} 
m^\Delta_{\mathcal{O}^\Delta}(n^\Delta_{\mathcal{O}^\Delta})^2=2^2+3^2+m\cdot 6^2 + n\cdot 1^2 = 36m+n+13.$$ 

We also observe that  $\Lambda^\Delta$ is given by its Gabriel quiver $Q_{\Lambda^\Delta}$, obtained from the quiver $Q^\Delta$ by 
removing the virtual arrows $\xi ',\mu '$ and $\mu$ if $n=2$, and the relations obtained from the above presented by substitutions 
$\xi '=\theta\lambda$ and $\mu '=\epsilon\eta$ (and $\mu=\alpha\beta$ if $n=2$). \end{exmp} 

\section{Proofs of Theorems 1 and 2}\label{sec:proof} 

In this section, we are going to prove Theorems 1 and 2. Since Theorem 1 is a direct consequence of Theorem 2 and Theorems 
\ref{thm:1.2}, \ref{thm:1.3}, \ref{thm:1.4} and \ref{thm:4.2}, we will only focus on the proof of Theorem 2. \smallskip 

Let 
$$\Lambda=\Lambda(Q,*,m_\bullet,c_\bullet,b_\bullet)\mbox{ and }
\Lambda^\Delta=\Lambda(Q^\Delta,f^\Delta,m_\bullet^\Delta,c_\bullet^\Delta,b^\Delta_\bullet)$$ 
be a weighted generalized triangulation algebra and the associated weighted triangulation algebra. We keep notation introduced 
in Section \ref{sec:4} and assume that $Q=glue(B_1,\dots,B_r;\Theta)$ with at least one of blocks $B_1,\dots,B_r$ of type IV or 
V (otherwise, we have $\Lambda=\Lambda^\Delta$ and there is nothing to prove). We note that this assumption implies 
$|Q^\Delta_0|=|Q_0|\geqslant 4$. In particular, it follows that $\Lambda^\Delta$ is not a singular disc, triangle, or tetrahedral 
algebra. In general, $\Lambda^\Delta$ is so called generalized spherical algebra \cite[Section 5]{HSS} if and only if $Q$ is 
glueing of two blocks of type IV with middle arrows $\tau_1,\tau_2$ in opposite directions. If $\Lambda^\Delta$ is a spherical 
algebra, then we always assume that $\Lambda^\Delta$ is not singular (equivalently, $m_{\tau_1}\geqslant 2$ or $m_{\tau_1}=1$ 
and $c_{\delta_1}c_{\tau_1}\neq -1$; see \cite[Section 5(2)]{HSS}). We have to exclude all mentioned singular algebras to ensure 
that $\Lambda^\Delta$ will be always tame, symmetric and periodic algebra of period $4$; see also 
\cite[Propositions 4.13(ii) and 3.9(iii)]{ES5}.  

It follows from definition of $\Lambda^\Delta$ that triangulation quiver $(Q^\Delta,f^\Delta)$ has the following 
$g^\Delta$-orbits consisting of virtual arrows: 
$$\mathcal{O}^\Delta(\xi_i)=(\xi_i\mbox{ }\mu_i),\quad\mbox{for }i\in\{1,\dots,s\},\mbox{ if }s\geqslant 1,$$
$$\mathcal{O}^\Delta(\xi_j')=(\xi_j'\mbox{ }\mu_j'),\quad\mbox{for }j\in\{1,\dots,t\},\mbox{ if }t\geqslant 1.$$ 
We consider the sequence of virtual arrows 
$$(\xi,\xi ')=(\xi_1,\dots,\xi_s,\xi_1',\dots,\xi_t').$$ 

For each vertex $x\in Q^\Delta_0$, we denote by $P^\Delta_x$ the associated indecomposable projective module $e_x\Lambda^\Delta$ 
in $\mod\Lambda^\Delta$. Moreover, for any arrow $\eta\in Q^\Delta_1$, we identify $\eta$ with the homomorphism 
$\eta:P^\Delta_{t(\eta)}\to P^\Delta_{s(\eta)}$ in $\mod\Lambda^\Delta$ given by left multiplication by $\eta$. We use the same 
convention for $\eta\in KQ^\Delta$ being a combination of paths with given source and target. 

Consider the following complexes in the homotopy category $K^b(P_{\Lambda^\Delta})$ of projective modules in $\mod\Lambda^\Delta$: 
$$T_x^{(\xi,\xi ')}:\qquad\xymatrix{0\ar[r] & P^\Delta_x\ar[r] & 0}$$ 
concentrated in degree $0$, for all vertices $x\in Q^\Delta_0$ different from $c_i$ and $x_{2j}$, 
$i\in\{1,\dots,s\}$, $j\in\{1,\dots,t\}$, and 
$$T_{c_i}^{(\xi,\xi ')}:\quad \xymatrix{0\ar[r] & P^\Delta_{c_i}\ar[r]^{\alpha_i} & P^\Delta_{a_i}\ar[r] & 0},$$
$$T_{x_{2j}}^{(\xi,\xi ')}:\quad \xymatrix{0\ar[r] & P^\Delta_{x_{2j}}\ar[r]^{\eta_j} & P^\Delta_{y_{1j}}\ar[r] & 0}$$ 
concentrated in degrees $1$ and $0$, for any $i\in\{1,\dots,s\}$ and $j\in\{1,\dots,t\}$. We put 
$T^{(\xi,\xi ')}:=\bigoplus_{x\in Q_0} T_x^{(\xi,\xi ')}$, and define the following algebra 
$$\Lambda^\Delta(\xi,\xi '):=\End_{K^b(P_{\Lambda^\Delta})}(T^{(\xi,\xi ')}).$$ 

Observe that $T^{(\xi,\xi ')}$ is an Okuyama-Rickard complex, so it is a tilting complex in $K^b(P_{\Lambda^\Delta})$, by 
Proposition \ref{pro:1.5}. Hence $\Lambda^\Delta(\xi,\xi ')$ is derived equivalent to $\Lambda^\Delta$. Moreover, we note that if the 
border function $b_\bullet$ is zero (or the set of border vertices is empty), then $\Lambda^\Delta(\xi,\xi ')$ is exactly the virtual 
mutation of the weighted triangulation algebra $\Lambda^\Delta$ with respect to a sequence $(\xi,\xi ')$ of virtual arrows in the 
sense of \cite{HSS}. 

We have the following consequence of Theorem \ref{thm:4.2} and the results on derived equivalences of 
self-injective algebras presented 
in Section \ref{sec:1}. 

\begin{theorem}\label{thm:4.5} The following statements hold. 
\begin{enumerate}[(1)] 
\item $\Lambda^\Delta(\xi,\xi ')$ is a finite-dimensional algebra. 
\item $\Lambda^\Delta(\xi,\xi ')$ is a tame symmetric algebra of infinite representation type. 
\item $\Lambda^\Delta(\xi,\xi ')$ is a periodic algebra o period $4$. 
\end{enumerate} \end{theorem} 

Following \cite{HSS}, we describe now $\Lambda^\Delta(\xi,\xi ')$ by a quiver $Q^\Delta(\xi,\xi ')$ and relations. In fact, 
$\Lambda^\Delta(\xi,\xi')$ is a virtual mutation except border relations. Indeed, the quiver $Q^\Delta(\xi,\xi ')$ has the same 
vertices as $Q^\Delta$, and it is obtained from $Q^\Delta$ by applying the following operations: 
\begin{itemize}
\item removing the (virtual) arrows $\xi_i,\mu_i,\xi_j',\mu_j'$, for all $i\in\{1,\dots,s\}$ and $j\in\{1,\dots,t\}$, 
\item reversing the direction of the arrows $\alpha_i,\beta_i$ and $\eta_j,\theta_j$, for all $i\in\{1,\dots,s\}$ and 
$j\in\{1,\dots,t\}$, 
\item adding the arrows $\xymatrix{a_i\ar[r]^{\tau_i} & b_i}$ and $\xymatrix{y_{1j}\ar[r]^{\pi_j} & x_{1j}}$, for all 
$i\in\{1,\dots,s\}$ and $j\in\{1,\dots,t\}$.
\end{itemize} 
In this way, the quiver $Q^\Delta(\xi,\xi ')$ is obtained from $Q^\Delta$ by: 
\begin{enumerate}[(1)] 
\item replacing every subquiver $B^\Delta_i$ in $Q^\Delta$ by the block $B_i$ of type IV, for $i\in\{1,\dots,s\}$,
\item replacing every subquiver $B^\Delta_{s+j}$ in $Q^\Delta$ by the subquiver $B'_{s+j}$ of the form 
$$
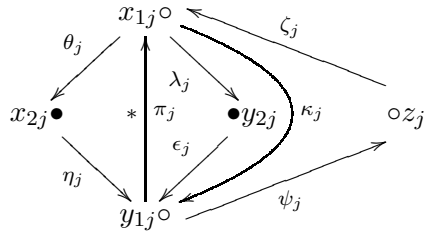
\xymatrix@C=0.5cm{ &x_{1j}\circ\ar[rd]_{\lambda_j}\ar[ld]_{\theta_j}\ar@/^{55pt}/[dd]^{\kappa_j} &&&\\ 
x_{2j}\bullet\ar[rd]_{\eta_j} & & \bullet y_{2j}\ar[ld]_{\epsilon_j} & &\circ z_j\ar@<-0.25cm>[lllu]_{\zeta_j}\\ 
&y_{1j}\circ \ar@<-0.25cm>[rrru]_{\psi_j}\ar[uu]_{\pi_j}^{*} &&& }$$ 
$j\in\{1,\dots,t\}$, being a gluing of block of type II with block of type IV. 
\end{enumerate} $\newline$ 
Therefore, we obtain a generalized triangulation quiver $(Q^\Delta(\xi,\xi '),*)$ by the above indicated marking of created blocks 
of type IV in the subquivers $B'_{s+j}$ -- in blocks $B_i$ we always mark triangle $(\beta_i \ \alpha_i \ \tau_i$). We abbreviate 
$Q'=(Q'_0,Q'_1,s,t):=Q^\Delta(\xi,\xi ')$. Note also that the associated quiver $(Q')^*$ is obtained from $Q'$ by removing vertices 
$c_i$, $x_{2j}$, the arrows $\alpha_i,\beta_i$ and the arrows $\eta_j,\theta_j$, for all $i\in\{1,\dots,s\}$ and $j\in\{1,\dots,t\}$. 

Let $f':(Q')^*_1\to (Q')^*_1$ and $g':(Q')^*_1\to (Q')^*_1$ be the permutations with $g'=\overline{f'}$ given by the generalized 
triangulation quiver $(Q',*)$. Moreover, denote by $\mathcal{O}(g')$ the set of all $g'$-orbits in $(Q')^*_1$. We define the weight 
function $m'_\bullet:\mathcal{O}(g')\to\mathbb{N}^*$, the parameter function $c'_\bullet:\mathcal{O}(g')\to K^*$, and the border 
function $b'_\bullet:\partial(Q',*)\to K$ as follows. For any $i\in\{1,\dots,s\}$ we set 
$$m'_{\mathcal{O}(\tau_i)}=m^\Delta_{\mathcal{O}^\Delta(\alpha_i)}\mbox{ and } c'_{\mathcal{O}(\tau_i)}=
c^\Delta_{\mathcal{O}^\Delta(\alpha_i)},$$ 
for any $j\in\{1,\dots,t\}$ 
$$m'_{\mathcal{O}(\pi_j)}=m^\Delta_{\mathcal{O}^\Delta(\eta_j)}\mbox{ and }c'_{\mathcal{O}(\pi_j)}=
c^\Delta_{\mathcal{O}^\Delta(\eta_j)},$$ 
and $m'_{\mathcal{O}}=m^\Delta_{\mathcal{O}}$ and $c'_{\mathcal{O}}=c^\Delta_{\mathcal{O}}$, for all remaining 
$g'$-orbits $\mathcal{O}$ in $(Q')^*_1$ (in case $\mathcal{O}=\mathcal{O}(\alpha)$ does not contain arrows of 
type $\pi_j$ or $\tau_i$, we have $\mathcal{O}(\alpha)=\mathcal{O}^\Delta(\alpha)$). Finally, since border loops 
are blocks of type I we have $\partial(Q',*)=\partial(Q^\Delta,f^\Delta)$, and the border function is induced from 
$b^\Delta_\bullet$, that is $b'_{x}=b^\Delta_x$, for any border vertex $x\in\partial(Q',*)$. 

Let now $\Lambda'=\Lambda(Q',*,m'_\bullet,c'_\bullet,b'_\bullet)$ be the associated weighted generalized triangulation 
algebra. By the construction, we have $\Lambda '=\Lambda$ if $Q=glue(B_1,\dots,B_r;\Theta)$ with all the blocks 
$B_1,\dots,B_r$ of types I-IV. 

\begin{theorem}\label{thm:4.6} The algebras $\Lambda '$ and $\Lambda^\Delta(\xi,\xi ')$ are isomorphic. \end{theorem} 

\begin{proof} We assume that the border $\partial(Q,*)$ of $(Q,*)$ is not empty, since otherwise $\Lambda '$ is isomorphic to 
$\Lambda^\Delta(\xi,\xi ')$, by definition (see \cite{HSS}). Denote by $\Lambda_0'$ the algebra $\Lambda '$ with border function 
$b'_\bullet=0$. It was proved \cite{HSS} that $\Lambda_0'$ is isomorphic to the virtual mutation $\Lambda^\Delta_0(\xi,\xi ')$ with 
respect to sequence $(\xi,\xi ')$ of virtual arrows, where $\Lambda^\Delta_0=
\Lambda(Q^\Delta,f^\Delta,m^\Delta_\bullet,c^\Delta_\bullet,0)$, that is, the algebra $\Lambda^\Delta_0$ is the weighted 
triangulaion algebra $\Lambda^\Delta_0=\Lambda(Q^\Delta,f^\Delta,m^\Delta_\bullet,c^\Delta_\bullet)$ (i.e. without border relations). 
In particular, it follows that the Gabriel quivers of $\Lambda '$ and $\Lambda^\Delta(\xi,\xi ')$ coincide. 

Further, all relations 2)-6) defining the weighted generalized triangulation algebra $\Lambda '$ hold also in $\Lambda_0'\cong 
\Lambda^\Delta_0(\xi,\xi ')$, and hence, all these relations are satisfied in $\Lambda^\Delta(\xi,\xi ')$. Note that there are no 
relations of type 4) in $\Lambda '$ since $Q'$ is a glueing of blocks of types I-IV. Eventually, we have to prove that all relations 
of type 1) defining $\Lambda '$ are satisfied in $\Lambda^\Delta(\xi,\xi ')$. So let 
$\rho=\alpha^2-c'_{\bar{\alpha}}A_{\bar{\alpha}}-b_{s(\alpha)}B_{\alpha}$, for a border loop $\alpha\in Q '$. 
Then $\alpha$ and $\bar{\alpha}$ belong to the same $g'$-orbit which is not equal to $(\pi_j\mbox{ }\kappa_j)$. Hence $\rho$, vieved 
as an element in $\Lambda^\Delta(\xi,\xi ')$, is a homomorphism $T_{s(\alpha)}^{(\xi,\xi ')}\to T_{s(\alpha)}^{(\xi,\xi ')}$ given by 
a homomorphism $P_{s(\alpha)}^{\Delta}\to P_{s(\alpha)}^\Delta$ in $\mod\Lambda^\Delta$ identified with the relation 
$\rho^\Delta=\alpha^2-c'_{\bar{\alpha}}A^\Delta_{\bar{\alpha}}-b'_{s(\alpha)}B^\Delta_{\alpha}$ in $\Lambda^\Delta$, where 
$A^\Delta_{\bar{\alpha}}$ and $B^\Delta_{\alpha}$ are the paths in $Q^\Delta$ obtained from the paths $A_{\bar{\alpha}}$ and 
$B_\alpha$ (in $Q'$) by replacing each arrow $\tau_i$ by the path $\alpha_i\beta_i$, for all $i\in\{1,\dots,s\}$. According to 
the above definition of weight and parameter functions $m'_\bullet$ and $c'_\bullet$ we obtain that $\rho^\Delta =0$ in 
$\Lambda^\Delta$, since $\rho^\Delta$ is then a relation of type 1) in the weighted triangulation algebra $\Lambda^\Delta$ 
(see Definition \ref{df:4.1}), and consequently, also $\rho=0$ in $\Lambda^\Delta(\xi,\xi ')$. 

Summing up, we have proved that all relations defining $\Lambda '$ hold also in $\Lambda^\Delta(\xi,\xi ')$, so 
$\Lambda^\Delta(\xi,\xi ')$ is a quotient of $\Lambda '$. Finally, because $\Lambda '$ is a socle deformation of $\Lambda_0'$ and 
$\Lambda^\Delta$ is a socle deformation of $\Lambda^\Delta_0$, we infer that algebras $\Lambda '$ and $\Lambda^\Delta(\xi,\xi ')$ 
have the same dimension, and the proof is now finished. \end{proof} 

Assume now that $Q=glue(B_1,\dots,B_r;\Theta)$ with one of the blocks $B_k$ of type V, equivalently, $t\geqslant 1$. We consider 
the quiver $Q''$ obtained from the quiver $Q'$ by removing the arrows $\pi_j,\kappa_j$, for all $j\in\{1,\dots,t\}$. Let also 
$\Lambda ''$ denote the algebra given by the quiver $Q''$ and relations: 
$$\lambda_j\epsilon_j\psi_j=\theta_j\eta_j\psi_j+c'_{\lambda_j}A_{\lambda_j},\zeta_j\lambda_j\epsilon_j=
\zeta_j\theta_j\eta_j+c'_{\bar{\zeta_j}}A_{\bar{\zeta_j}},\epsilon_j\psi_j\zeta_j=c'_{\epsilon_j}A_{\epsilon_j},$$  
$$\psi_j\zeta_j\lambda_j=c'_{\psi_j}A_{\psi_j},\psi_j\zeta_j\theta_j=0,\eta_j\psi_j\zeta_j=0,
\lambda_j\epsilon_j\psi_j\zeta_j\lambda_j=0,\leqno{(R)}$$
$$\epsilon_j\psi_j\zeta_j\lambda_j\epsilon_j=0,\psi_j\zeta_j\lambda_j\epsilon_j\psi_j=0,
\zeta_j\lambda_j\epsilon_j\psi_j\zeta_j=0,A_{\lambda_j}\bar{\zeta_j}=0$$
$$\mbox{ and }\psi_j g'(\psi_j)f'(g'(\psi_j))=0,\mbox{if }g'(\psi_j)\notin\{\nu_1,\dots,\nu_s\},$$ 
for $j\in\{1,\dots,t\}$, and all relations of types 1)-3),5) and 6) from definition of a weighted generalized triangulation 
algebra $Q'=\Lambda(Q',*,m'_\bullet,c'_\bullet,b'_\bullet)$, with respect only to blocks in $Q'$ corresponding to original blocks 
$B_1,\dots,B_s$ of type IV from $Q$ (note that all blocks $B_1,\dots,B_s$ of type IV in $Q$ remain unchanged in $Q''$). Now, we have 
the following result. 

\begin{prop}\label{prop:4.7} The algebras $\Lambda '$ and $\Lambda ''$ are isomorphic. \end{prop} 

\begin{proof} It is sufficient to observe that $\pi_j=\psi_j\zeta_j$ and $\kappa_j=\lambda_j\epsilon_j-\theta_j\eta_j$ 
in $\Lambda '$, so the arrows of $Q'$ corresponding to $\pi_j$ and $\kappa_j$ do not occur in the Gabriel quiver 
$Q_{\Lambda '}$ of $\Lambda '$, and hence, $Q_{\Lambda '}=Q''=Q_{\Lambda ''}$. The relations in $\Lambda ''$ are 
obtained from the relations in the weighted generalized triangulation algebra $\Lambda '$ by substituting the 
above formulas for $\pi_j$ and $\kappa_j$ and performing some standard operations. Therefore, the required 
isomorphism $\Lambda '\cong\Lambda ''$ holds. \end{proof} 

Note only that the last six relations defining $\Lambda ''$ above arise from relations of types 5) and 6) in 
the weighted generalized triangulation algebra $\Lambda'$. For example, the relation of type 5) for 
$\alpha=\kappa_j$ gives $\kappa_j\psi_jg'(\psi_j)=0$ in $\Lambda'$, so we obtain $A_{\lambda_j}\bar{\zeta_j}=0$ 
in $\Lambda '' $, since $\kappa_j\psi_j=c'_{\bar{\kappa_j}}A_{\bar{\kappa_j}}=c'_{\lambda_j}A_{\lambda_j}$ and 
$g'(\psi_j)=\overline{f'(\psi_j)}=\bar{\zeta_j}$. 

Finally, we will show that $\Lambda$ is derived equivalent to $\Lambda ''$. For each vertex $x\in Q''$ we denote by 
$P''_x=e_x\Lambda ''$ the associated indecomposable projective module in $\mod\Lambda ''$. Moreover, for any arrow 
$\eta\in Q''_1$ (or a combination of paths) we identify $\eta$ with the homomorphism $\eta:P''_{t(\eta)}\to P''_{s(\eta)}$, 
defined as the left multiplication by $\eta$. 

We consider the following complex $T=\bigoplus_{x\in Q''_0} T_x$ in $K^b(P_{\Lambda ''})$, where for any $x\in Q_0''$ 
different from $x_{1j}$, for $j\in\{1,\dots,t\}$, $T_x$ is the complex  
$$T_x:\;\xymatrix{0\ar[r] & P''_{x}\ar[r] & 0}$$ 
concentrated in degree $0$, whereas for any $j\in\{1,\dots,t\}$, $T_{x_{1j}}$ is of the form 
$$T_{x_{1j}}:\;\xymatrix{0\ar[r] & P''_{x_{1j}}\ar[r]^{\zeta_j} & P''_{z_j}\ar[r] & 0}$$ 
(concentrated in degrees $1$ and $0$). Then we have the following lemma. 

\begin{lem}\label{lem:4.8} $T$ is a tilting complex in $K^b(P_{\Lambda ''})$. \end{lem} 

\begin{proof} It is enough to see that $\Lambda ''$ has a decomposition $\Lambda ''= P\oplus Q$, $P=\bigoplus_{j=1}^t P''_{x_{1j}}$ 
and $Q=\bigoplus_{x\in Q_0''\setminus\{x_{11},\dots,x_{1t}\}}P''_x$, and $T$ is of the form $T=T^1\oplus T^2$ with $T^1=0\to Q \to 0$ 
and $T^2=\xymatrix{0\ar[r] & P \ar[r]^{\zeta} & Q'\ar[r] & 0}$, where $Q'=\bigoplus_{j=1}^t P''_{z_j}$ and 
$\zeta=diag(\zeta_1,\dots,\zeta_t)$ is a minimal left $\add(Q)$-approximation of $P$ (this is a consequence of the fact that all 
nontrivial paths in $Q''$ ending at $x_{1j}$ admit factorization through $\zeta_j$). Then $T$ is a tilting complex, due to 
Proposition \ref{pro:1.5}. \end{proof} 

Set $\Gamma(\Lambda):=\End_{K^b(P_{\Lambda ''})}(T)$. In particular, it follows from Theorem \ref{thm:1.1} that $\Gamma(\Lambda)$ is 
derived equivalent to $\Lambda''$. 

\begin{theorem}\label{thm:4.9} We have an isomorphism $\Lambda\cong\Gamma(\Lambda)$. \end{theorem} 

\begin{proof} We recall that the modules $\hat{P}_x=\Hom_{K^b(P_{\Lambda ''})}(T,T_x)$, for $x\in Q_0''$, form a complete set of all 
indecomposable projective modules in $\mod\Gamma(\Lambda)$. For $j\in\{1,\dots,t\}$, we consider the following morphisms between 
indecomposable direct summands of $T$: 
$$\hat{\gamma_j}:T_{x_{2j}}\to T_{z_j},\mbox{ given by homomorphism }\zeta_j\theta_j:P''_{x_{2j}}\to P''_{z_j},$$ 
$$\hat{\phi_j}:T_{y_{2j}}\to T_{z_j},\mbox{ given by homomorphism }\zeta_j\lambda_j:P''_{y_{2j}}\to P''_{z_j},$$ 
$$\hat{\omega_j}:T_{z_{j}}\to T_{x_{1j}},\mbox{ given by the identity }id:P''_{z_{j}}\to P''_{z_j},$$ 
$$\hat{\sigma_j}:T_{x_{1j}}\to T_{x_{2j}},\mbox{ given by homomorphism: }\eta_j\psi_j:P''_{z_{j}}\to P''_{x_{2j}},$$ 
$$\hat{\rho_j}:T_{x_{1j}}\to T_{y_{2j}},\mbox{ given by homomorphism: }\epsilon_j\psi_j-c'_{\epsilon_j} A_{\epsilon_j}':
P''_{z_{j}}\to P''_{y_{2j}},$$ 
$$\hat{\eta_j}:T_{y_{1j}}\to T_{x_{2j}},\mbox{ given by homomorphism: }\eta_j:P''_{y_{1j}}\to P''_{x_{2j}},$$ 
$$\hat{\epsilon_j}:T_{y_{1j}}\to T_{y_{2j}},\mbox{ given by homomorphism: }\epsilon_j:P''_{y_{1j}}\to P''_{y_{2j}},$$
$$\hat{\psi_j}:T_{z_{j}}\to T_{y_{1j}},\mbox{ given by homomorphism: }\psi_j:P''_{z_{j}}\to P''_{y_{1j}},$$ 
and $\hat{\eta}:T_{t(\eta)}\to T_{s(\eta)}$, given by $\eta:P''_{t(\eta)}\to P''_{s(\eta)}$, for every arrow 
$\eta\in Q''_1$ different from $\zeta_j,\theta_j,\eta_j,\psi_j,\lambda_j,\epsilon_j$, $j\in\{1,\dots,t\}$. Any 
morphism of the above form $\hat{\alpha}:T_x\to T_y$ induces an irreducible homomorphism 
$\alpha=\Hom_{K^b(P_{\Lambda ''})}(T,\hat{\alpha}):\hat{P}_x\to\hat{P}_y$ between projective modules in 
$\mod\Gamma(\Lambda)$, which is corresponding to an arrow $\alpha: y\to x$ in the Gabriel quiver 
$Q_{\Gamma(\Lambda)}$. It is easy to see that these arrows exhaust all arrows in $Q_{\Gamma(\Lambda)}$. In 
particular, for any $j\in\{1,\dots,t\}$, $Q_{\Gamma(\Lambda)}$ has the following subquiver $B_{s+j}''$ 
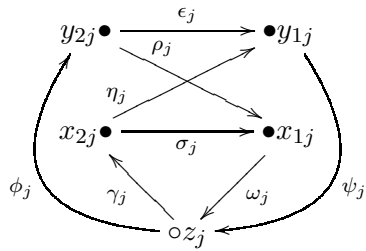
$$\xymatrix@C=0.5cm{&y_{2j}\bullet\ar[rrd]^(0.33){\rho_j}\ar[rr]^{\epsilon_j}&&\bullet y_{1j}\ar@/^{40pt}/[ldd]^{\psi_j}& \\ 
&x_{2j}\bullet\ar[rr]_{\sigma_j}\ar[rru]^(0.2){\eta_j}&&\bullet x_{1j}\ar[ld]^{\omega_j}& \\ 
&&\circ z_j\ar[lu]^{\gamma_j}\ar@/^{40pt}/[luu]^{\phi_j}&& }$$ 
which coincides with the block $B''_{s+j}=B_{s+j}$ of type V in $Q$, and all remaining arrows in $Q_{\Gamma(\Lambda)}$ 
are induced by the arrows $\eta$ in blocks of types I-IV in $Q''$ (except blocks of type IV coming from $B'_{s+j}$). 
Consequently, the Gabriel quivers of $\Gamma(\Lambda)$ and $\Lambda$ coincide. In particular, it follows also that 
all relations defining the weighted generalized triangulation algebra $\Lambda$ except relations of type 4) are 
satisfied in $\Gamma(\Lambda)$. 

We will prove that all relations of type 4) from definition of $\Lambda$ are also satisfied in $\Gamma(\Lambda)$. Fix 
$j\in\{1,\dots,t\}$. Observe first that the element $\rho=\phi_j\epsilon_j-\gamma_j\eta_j-c_{\bar{\phi_j}}A_{\bar{\phi_j}}$ 
considered as the element of $\Gamma(\Lambda)$ corresponds to the map $\hat{\rho}:T_{y_{1j}}\to T_{z_j}$ given by homomorphism of 
the form 
$$\zeta_j\lambda_j\epsilon_j-\zeta_j\theta_j\eta_j-c'_{\bar{\zeta_j}}A_{\bar{\zeta_j}}:P''_{y_{1j}}\to P''_{z_j},$$ 
because the path $A_{\bar{\phi_j}}$ in $Q$ corresponds to the path $A_{\bar{\zeta_j}}$ in $Q''$ (obtained by replacing each subpath 
$\zeta_k\lambda_k$ by $\phi_k$). Therefore $\hat{\rho}=0$, by the relations (R) defining $\Lambda ''$ (see page 29), and hence 
$\rho=0$ in $\Gamma(\Lambda)$. Similarly, the element $\rho=\epsilon_j\psi_j-\rho_j\omega_j-c_{\bar{\epsilon_j}}A_{\bar{\epsilon_j}}$ 
is zero in $\Gamma(\Lambda)$, since it is given by 
$\hat{\rho}=\epsilon_j\psi_j-(\epsilon_j\psi_j-c'_{\epsilon_j}A'_{\epsilon_j})-c_{\bar{\epsilon_j}}A_{\bar{\epsilon_j}}'=0$ 
in $\Lambda ''$. Further, if $\rho=\psi_j\phi_j-c_{\bar{\psi_j}}A_{\bar{\psi_j}}$, then $\rho=0$ in $\Gamma(\Lambda)$, 
because the corresponding element $\hat{\rho}=\psi_j\zeta_j\lambda_j-c'_{\psi_j}A_{\psi_j}$ is zero in $\Lambda ''$, 
due to the relations (R). Next, we observe that the element $\rho=\gamma_j\sigma_j-\phi_j\rho_j$ in $\Gamma(\Lambda)$ 
is given by 
$$\hat{\rho}=\zeta_j\theta_j\eta_j\psi_j-\zeta_j\lambda_j(\epsilon_j\psi_j-c'_{\epsilon_j} A_{\epsilon_j}')= 
-\zeta_j(\lambda_j\epsilon_j\psi_j-\theta_j\eta_j\psi_j)+c'_{\epsilon_j}\zeta_j\lambda_j A_{\epsilon_j}'= $$ 
$$-c'_{\lambda_j}\zeta_j A_{\lambda_j}+c'_{\epsilon_j}B_{\zeta_j}=-c'_{\lambda_j}B_{\zeta_j}+c'_{\epsilon_j}B_{\zeta_j}$$ 
in $\Lambda ''$, which is equal to zero, since $\lambda_j$ and $\epsilon_j$ belong to the same $g'$-orbit in $Q'_1$. This shows 
that $\rho=0$ in $\Gamma(\Lambda)$. It is clear from definition of arrows in $Q_{\Gamma(\Lambda)}$ that we have also the 
relation $\sigma_j\omega_j-\eta_j\psi_j$ in $\Gamma(\Lambda)$. Moreover, $\rho=\omega_j\gamma_j$ is corresponding to homomorphism 
$h:T_{x_{2j}}\to T_{x_{1j}}$ which is given by homomorphism $\zeta_j\theta_j:P''_{x_{2j}}\to P''_{z_j}$ factorizing through 
$\zeta_j$, so $h$ is homotopic to zero, and so $\rho =0$ in $\Gamma(\Lambda)$. The same arguments show that $\omega_j\phi_j=0$ 
in $\Gamma(\Lambda)$. Further, we have $\psi_j\zeta_j\theta_j=0$ in $\Lambda ''$, hence $\psi_j\gamma_j=0$ in $\Gamma(\Lambda)$, 
and the relation $\psi_j g(\psi_j) f(g(\psi_j))$ in $\Gamma(\Lambda)$ follows from analogous relation in $\Lambda ''$, because 
we have $g(\psi_j)=g'(\psi_j)=\bar{\zeta_j}$. Finally, observe that the remaining zero relations of length 4 follow from analogous 
relations in $\Lambda''$, whereas $B_{\phi_j}\bar{\phi_j}=0$ in $\Gamma(\Lambda)$, since this path corresponds to a path 
$B_{\zeta_j}\bar{\zeta}_j$ in $e_{z_j}\Lambda '' e_{t(\bar{\zeta}_j)}$, containing a subpath $A_{\lambda_j}\bar{\zeta}_j$, 
which is included in (R).   

As a result, we proved that the algebra $\Gamma(\Lambda)$ is of the form $\Gamma(\Lambda)=KQ/I$, where $I$ is an 
ideal of $KQ$, which contains the ideal $I(Q,*,m_\bullet,c_\bullet,b_\bullet)$ defining algebra $\Lambda$. In 
particular, socles of indecomposable projective $\Gamma(\Lambda)$-modules are given by cycles of the form $B_\eta$ 
(see also Corollary \ref{lem:3.4}), so the bases for indecomposable projective $\Lambda$-modules described in 
Proposition \ref{prop:3.5} give rise to bases of indecomposable projective $\Gamma(\Lambda)$-modules. Therefore, 
the dimensions of $\Lambda$ and $\Gamma(\Lambda)$ coincide, and hence, these algebras are indeed isomorphic. \end{proof}

\section{Generalized triangulation quivers from surfaces}\label{sec:5}

In this article, by a surface we mean a connected compact oriented real two-dimensional manifold, with or without boundary. It is 
well known that every surface $S$ admits an additional structure of a finite two-dimensional triangular cell complex, and hence a 
triangulation, by the deep Triangulation Theorem (see for example \cite[Section 2.3]{Ca}). 

For a positive natural number $n$, we denote by $D^n$ the unit disc in the $n$-dimensional Euclidean space $\mathbb{R}^n$, formed by 
all points of distance $\leqslant 1$ from the origin. The boundary $\partial D^n$ is then the unit sphere $S^{n-1}$ in $\mathbb{R}^n$ 
containing all points of distance exactly $1$ from the origin. Further, by an $n$-cell we mean a topological space homeomorphic to 
the open disc $int(D^n)=D^n\setminus\partial D^n$. In particular, $S^0=\partial D^1$ consists of two points, while $\partial D^0$ 
is empty, so $0$-cell is a singleton. We refer to \cite[Appendix]{Ha2} for some basic topological facts about cell complexes. 

Let $S$ be a surface. A finite family of maps $\phi_i^n:D^n_i\to S$, with $n\in\{0,1,2\}$ and $D^n_i=D^n$, is called a finite 
two-dimensional cell complex on $S$, provided the following conditions are satisfied. 
\begin{enumerate}[(1)] 
\item Each $\phi^n_i$ restricts to a homeomorphism $int(D^n_i)\to\phi^n_i(int(D^n_i))$ and the $n$-cells 
$e^n_i:=\phi^n_i(int(D^n_i))$ of $S$ are pairwise disjoint and their union is $S$. 

\item For any $2$-cell $e^2_i$ of $S$, $\phi^2_i(\partial D^2_i)$ is the union of $k$ $1$-cells and $m$ $0$-cells, where 
$k\in\{2,3\}$, $m\in\{1,2,3\}$, and different from $\xymatrix{\bullet\ar@{-}@<0.1cm>[r]&\bullet\ar@{-}@<0.1cm>[l]}$. 
\end{enumerate} 

Then the closures $\phi^2_i(D^2_i)$ of all $2$-cells $e^2_i$ are called \emph{triangles} of $S$, and the closures $\phi^1_i(D^1_i)$ 
of all $1$-cells $e^1_i$ are called \emph{edges} of $S$. The collection of all triangles is called a \emph{triangulation} of $S$. 
It follows from the assumption that such a triangulation $\mathcal{T}$ of $S$ has at least two edges and is different from the 
unpunctured digon $\xymatrix{\bullet\ar@{-}@<0.1cm>[r]&\bullet\ar@{-}@<0.1cm>[l]}$. Hence $\mathcal{T}$ is a finite collection of 
triangles of the form 
$$
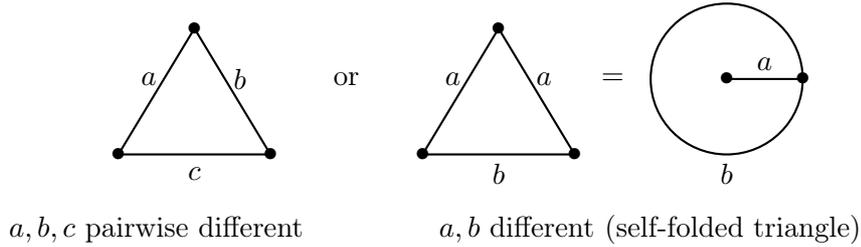
\begin{tikzpicture}
\draw[thick](-1,0)--(1,0);
\draw[thick](-1,0)--(0,1.67);
\draw[thick](1,0)--(0,1.67);
\node() at (-1,0){$\bullet$};
\node() at (1,0){$\bullet$};
\node() at (0,1.67){$\bullet$};
\node() at (-0.6,1){$a$};
\node() at (0.6,1){$a$};
\node() at (0,-0.25){$b$};

\draw[thick](-5,0)--(-3,0);
\draw[thick](-5,0)--(-4,1.67);
\draw[thick](-3,0)--(-4,1.67);
\node() at (-5,0){$\bullet$};
\node() at (-3,0){$\bullet$};
\node() at (-4,1.67){$\bullet$};
\node() at (-4.6,1){$a$};
\node() at (-3.4,1){$b$};
\node() at (-4,-0.25){$c$};

\draw[thick](3,1) circle (1);
\node() at (3,1){$\bullet$};
\node() at (4,1){$\bullet$};
\draw[thick](3,1)--(4,1);
\node() at (3.5,1.2){$a$};
\node() at (3,-0.25){$b$};

\node() at (1.5,1){$=$};
\node() at (-2,1){or};
\node() at (-4.5,-1){$a,b,c$ pairwise different};
\node() at (2,-1){$a,b$ different (self-folded triangle)};

\end{tikzpicture}$$ 
such that every edge is either the edge of exactly two triangles, is the self-folded edge, or lies on the boundary $\partial S$ of 
$S$. We note that a given surface $S$ admits many finite two-dimensional triangular cell complex structures, and so, triangulations. 

By a \emph{triangulated surface} we mean a pair $(S,\mathcal{T})$, where $S$ is a surface and $\mathcal{T}$ a triangulation of $S$. 
It turns out, that the blocks of types IV and V in a triangulation quiver can be reconstructed from appropriate marking of the 
self-folded triangles.  

\begin{df}\label{df:5.1} A \emph{marked triangulated surface} is a triple $(S,\mathcal{T},*)$ consisting of a triangulated surface 
$(S,\mathcal{T})$ together with marking (possibly empty) of a family of self-folded triangles in $\mathcal{T}$ by 
$$
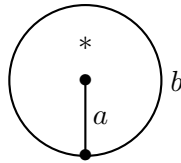
\begin{tikzpicture}
\draw[thick](0,0) circle (1);
\draw[thick](0,0)--(0,-1);
\node() at (0,0){$\bullet$};
\node() at (0,-1){$\bullet$};
\node() at (0.2,-0.5){$a$};
\node() at (1.2,0){$b$};
\node() at (0,0.5){$*$};
\end{tikzpicture}$$ 
such that the unfolded edge $b$ of any marked self-folded triangle is not a boundary edge. \end{df} 

Let $(S,\mathcal{T},*)$ be a marked triangulated surface. We assocaite to it the generalized triangulation quiver 
$(Q,*)=(Q(S,\mathcal{T}),*)$ defined as follows. The set of vertices $Q_0$ consists of edges of $\mathcal{T}$, while the set $Q_1$ 
of arrows is defined by the following rules:  
\begin{enumerate}[(1)] 
\item for any boundary edge $a$ of $\mathcal{T}$ we have the loop 
$$
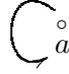
\xymatrix@R=0.01cm{&_{\circ}\ar@(lu,ld)[d]\\& a }$$ 

\item for any triangle $\Delta=(a\mbox{ }b\mbox{ }c)$ in $\mathcal{T}$ with pairwise different edges $a,b,c$, and oriented according 
to the orientation of $S$, we have the cycle of arrows 
$$
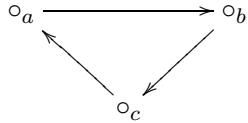
\xymatrix{\circ_a\ar[rr]& &\circ_b\ar[ld]\\ &\circ_c\ar[lu]&}$$  

\item for any unmarked self-folded triangle 
$$
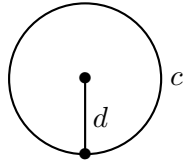
\begin{tikzpicture}
\draw[thick](0,0) circle (1);
\draw[thick](0,0)--(0,-1);
\node() at (0,0){$\bullet$};
\node() at (0,-1){$\bullet$};
\node() at (0.2,-0.5){$d$};
\node() at (1.2,0){$c$};
\end{tikzpicture}$$ 
in $\mathcal{T}$, we have the quiver $
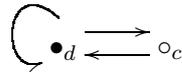
\xymatrix@R=0.01cm{&\ar@(lu,ld)[d]&\\& \bullet_d \ar@<0.2cm>[r]  &\circ_c \ar@<0.1cm>[l] }$ 

\item for any marked self-folded triangle in $\mathcal{T}$ 
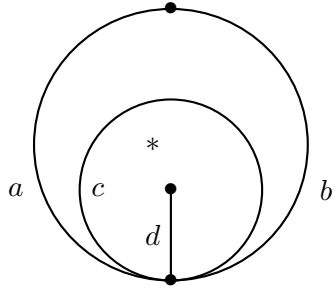
$$\begin{tikzpicture}[scale=1.2]
\draw[thick] (0,-1)--(0,0); 
\draw[thick](0,0) circle (1);
\node() at (0,0){$\bullet$};
\node() at (0,-1){$\bullet$};  
\node() at (-0.8,0){$c$};
\node() at (-0.2,0.5){$*$};
\node() at (-0.2,-0.5){$d$};
\draw[thick](0,0.5) circle (1.5);
\node() at (-1.7,0){$a$};
\node() at (1.7,0){$b$};
\node() at (0,2){$\bullet$};
\end{tikzpicture}$$ 
with $(a\mbox{ }b\mbox{ }c)$ a triangle in $\mathcal{T}$, we have the following marked quiver
$$
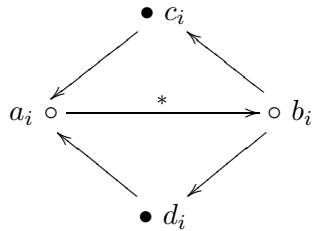
\xymatrix{ & \bullet \mbox{ }c_i  \ar[ld]& \\ 
a_i\mbox{ }\circ\ar[rr]^{*} & & \circ\mbox{ } b_i \ar[lu]\ar[ld] \\ 
 & \bullet\mbox{ }d_i\ar[lu]& }$$
\item for any pair of marked self-folded triangles in $\mathcal{T}$ 
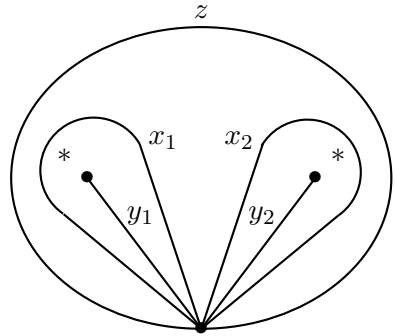
$$\begin{tikzpicture} 
\draw[thick](0,2) ellipse (2.5 and 2); 
\node() at (0,0){$\bullet$}; 

\draw[thick](0,0)--(1.5,2);
\node() at (1.5,2){$\bullet$}; 
\node() at (1.8,2.3){$*$}; 
\draw[thick](0,0)--(-1.5,2);
\node() at (-1.5,2){$\bullet$};
\node() at (-1.8,2.3){$*$}; 
 
\draw[thick](0,0)--(1.8,1.5); 
\draw[thick](1.8,1.5) arc (-55:150:0.7);
\draw[thick](0,0)--(0.81,2.45);

\draw[thick](0,0)--(-1.81,1.52); 
\draw[thick](-0.81,2.45) arc (30:235:0.7);
\draw[thick](0,0)--(-0.81,2.45); 

\node() at (0.5,2.5){$x_2$}; 
\node() at (-0.5,2.5){$x_1$}; 
\node() at (-0.8,1.5){$y_1$}; 
\node() at (0.8,1.5){$y_2$};
\node() at (0,4.2){$z$}; 
\end{tikzpicture}$$ 
with $(x_1\mbox{ }z\mbox{ }x_2)$ a triangle in $\mathcal{T}$, we have the marked quiver 
$$
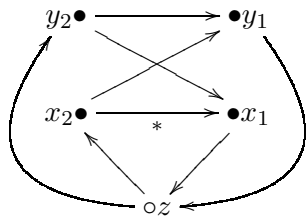
\xymatrix@C=0.5cm{&y_{2}\bullet\ar[rrd]\ar[rr]&&\bullet y_{1}\ar@/^{40pt}/[ldd]& \\ 
&x_{2}\bullet\ar[rr]_{*}\ar[rru]&&\bullet x_{1}\ar[ld]& \\ 
&&\circ z\ar[lu]\ar@/^{40pt}/[luu]&& }$$
\item for the triple of marked self-folded triangles in $\mathcal{T}$ 
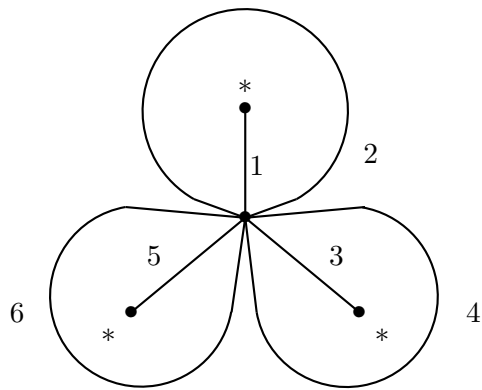
$$\begin{tikzpicture}[scale=1.5]
\draw[thick](-1,0)--(0,0.835)--(1,0);
\draw[thick](0,0.835)--(0,1.8); 

\draw[thick](0.1,0) arc (-170:80:0.8);
\draw[thick](0.1,0)--(0,0.835)--(1.05,0.93);

\draw[thick](-1.05,0.93) arc (100:350:0.8);
\draw[thick](-0.12,0)--(0,0.835)--(-1.05,0.93); 

\draw[thick](0.45,1) arc (-60:240:0.9);
\draw[thick](0.45,1)--(0,0.835)--(-0.45,1);

\node() at (0,0.835){$\bullet$};
\node() at (-1,0){$\bullet$}; 
\node() at (-1.2,-0.2){$*$};
\node() at (1,0){$\bullet$};
\node() at (1.2,-0.2){$*$};
\node() at (0,1.8){$\bullet$};
\node() at (0,2){$*$};

\node() at (0.1,1.3){$1$};
\node() at (0.8,0.5){$3$};
\node() at (-0.8,0.5){$5$};
\node() at (1.1,1.4){$2$};
\node() at (2,0){$4$};
\node() at (-2,0){$6$};

\end{tikzpicture}$$ 
with $(2\mbox{ }4\mbox{ }6)$ a triangle in $\mathcal{T}$, we have the following quiver 
$$
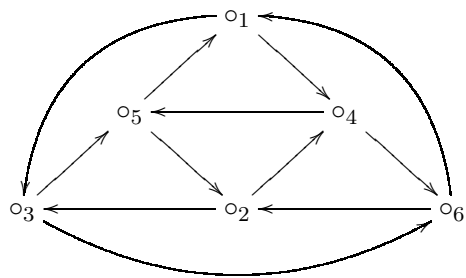
\xymatrix{
&&\circ_1\ar[rd]\ar@/_25pt/[lldd]&&\\
&\circ_5\ar[ru]\ar[rd]&&\circ_4\ar[ll]\ar[rd]&\\
\circ_3\ar[ru]\ar@/_25pt/[rrrr]&&\circ_2\ar[ll]\ar[ru]&&\circ_6\ar[ll]\ar@/_25pt/[lluu]}$$
$\newline$
\end{enumerate} 

We note that the quivers occuring in (4), (5) and (6) are the quivers associated in \cite[Sections 4 and 13]{FST2} to the signed 
adjacency matrices of the presented above collections of triangles of bordered surfaces with marked points. Moreover, the quiver 
occuring in (6) is the triangulation quiver obtained by glueing of four triangles 
$$
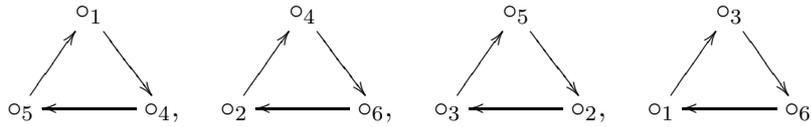
\xymatrix@C=0.33cm{&\circ_1\ar[rd]&&&\circ_4\ar[rd]&&&\circ_5\ar[rd]&&&\circ_3\ar[rd]& \\ 
\circ_5\ar[ru]&&\circ_4,\ar[ll]&\circ_2\ar[ru]&&\circ_6,\ar[ll]&\circ_3\ar[ru]&&\circ_2,\ar[ll]&\circ_1\ar[ru]&&\circ_6\ar[ll]}$$ 
This quiver is the triangulation quiver associated to the tetrahedral triangulation of the sphere $S^2$ 
$$
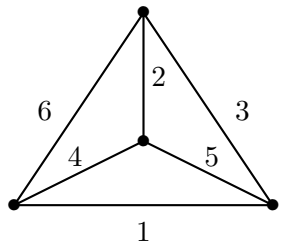
\begin{tikzpicture}
\draw[thick] (-1.7,-0.85)--(1.7,-0.85)--(0,1.7)--(-1.7,-0.85)--(0,0)--(0,1.7); 
\draw[thick] (0,0)--(1.7,-0.85); 
\node() at (-1.7,-0.85){$\bullet$};
\node() at (1.7,-0.85){$\bullet$};
\node() at (0,1.7){$\bullet$};
\node() at (0,0){$\bullet$}; 
\node() at (0,-1.2){$1$};
\node() at (0.2,0.85){$2$};
\node() at (1.3,0.4){$3$};
\node() at (-0.9,-0.2){$4$};
\node() at (0.9,-0.2){$5$};
\node() at (-1.3,0.4){$6$}; 
\end{tikzpicture}$$ 
with coherent orientation $(1\mbox{ }4\mbox{ }5)$, $(2\mbox{ }4\mbox{ }6)$, $(2\mbox{ }3\mbox{ }5)$ and $(1\mbox{ }3\mbox{ }6)$. 

\begin{theorem}\label{thm:5.2} The class of generalized triangulation quivers with at least two vertices coincides with the class 
of quivers associated to marked triangulated surfaces.
\end{theorem} 

\begin{exmp}\label{ex:5.3} Consider the following marked triangulated surface $(S,\mathcal{T},*)$ 
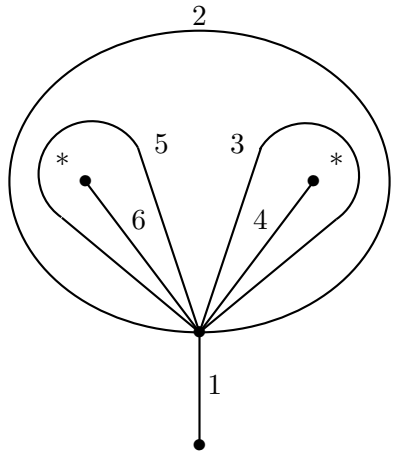
$$\begin{tikzpicture} 
\draw[thick](0,2) ellipse (2.5 and 2); 
\node() at (0,0){$\bullet$}; 

\draw[thick](0,0)--(1.5,2);
\node() at (1.5,2){$\bullet$}; 
\node() at (1.8,2.3){$*$}; 
\draw[thick](0,0)--(-1.5,2);
\node() at (-1.5,2){$\bullet$};
\node() at (-1.8,2.3){$*$}; 
 
\draw[thick](0,0)--(1.8,1.5); 
\draw[thick](1.8,1.5) arc (-55:150:0.7);
\draw[thick](0,0)--(0.81,2.45);

\draw[thick](0,0)--(-1.81,1.52); 
\draw[thick](-0.81,2.45) arc (30:235:0.7);
\draw[thick](0,0)--(-0.81,2.45); 

\node() at (0.5,2.5){$3$}; 
\node() at (-0.5,2.5){$5$}; 
\node() at (-0.8,1.5){$6$}; 
\node() at (0.8,1.5){$4$};
\node() at (0,4.2){$2$}; 
\node() at (0.2,-0.7){$1$};

\draw[thick](0,0)--(0,-1.5); 
\node() at (0,-1.5){$\bullet$};
 
\end{tikzpicture}$$ 
Then the associated generalized triangulation quiver $(Q(S,\mathcal{T}),*)$ is of the form 
$$
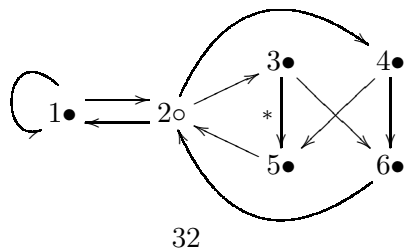
\xymatrix@R=0.01cm{
&&&&\\
&&&3\bullet\ar[dddd]_{*}\ar[rdddd]&4\bullet\ar[ldddd]\ar[dddd]\\
&\ar@(lu,ld)[d]&&&\\
& 1\bullet\ar@<0.2cm>[r] & 2\circ \ar@<0.1cm>[l]\ar[ruu]\ar@/^{30pt}/[rruu]&&\\
&&&& \\
&&&5\bullet\ar[luu]&6\bullet\ar@/^{30pt}/[lluu]}$$
$\newline$
$\newline$ 
and this is the quiver considered in Example \ref{ex:3.2}. 

\end{exmp}


\begin{thebibliography}{00} 
\bibitem{Ar} S. Ariki, Hecke algebras of classical type and their representation type, arXiv: math.QA/0302136v1. 

\bibitem{ASS} I. Assem, D. Simson, A. Skowro\'nski, Elements of the Representation Theory of 
Associative Algebras 1: Techniques of Representation Theory, London Mathematical Society Student 
Texts, vol. 65, Cambridge University Press, Cambridge, 2006. 

\bibitem{BEHSY} J. Bia{\l}kowski, K. Erdmann, A. Hajduk, A. Skowro\'nski, K. Yamagata, 
Socle equivalences of weighted surface algebras, J. Pure and Appl. Algebra 226, Issue 4, 2022, 106886. 

\bibitem{BES} J. Bia{\l}kowski, K. Erdmann, A. Skowro\'nski, Periodicity of self-injective algebras 
of polynomial growth, J. Algebra 443 (2015), 200--269.

\bibitem{BiHS} J. Bia{\l}kowski, T. Holm, A. Skowro\'nski, Derived equivalences for tame weakly 
symmetric algebras having only periodic modules, J. Algebra 269 (2003), 652--668.

\bibitem{BoHS} R. Bocian, T. Holm, A. Skowro\'nski, Derived equivalence classification of weakly 
symmetric algebras of Euclidean type, J. Pure Appl. Algebra 191 (2004), 43--74.

\bibitem{BIKR} I. Burban, O. Iyama, B. Keller, I. Reiten, Cluster tilting for one-dimensional 
hypersurface singularities, Adv. Math. 217 (2008), 2443--2484.

\bibitem{Ca} S. C. Carlson, Topology of Surfaces. Knots and Manifolds, A First Undergraduate 
Course, Wiley, New York, 2001.

\bibitem{CB} W. Crawley-Boevey, Tame algebras and generic modules, Proc. London Math. Soc. 63 
(1991), 241--265.

\bibitem{CR} J. Chuang, J. Rickard, Representations of finite groups and tilting, in: Handbook 
on Tilting Theory, London Math. Soc. Lecture Note Ser., vol. 332, Cambridge University Press, 
Cambridge, 2007, pp. 359--391. 

\bibitem{DWZ1} H. Derksen, J. Weyman, A. Zelevinsky, Quivers with potentials and their 
representations. I. Mutations, Selecta Math. (N. S.) 14 (2008), 59--119.

\bibitem{DWZ2} W. Derksen, J. Weyman, A. Zelevinsky, Quivers with potentials and their 
representations. II. Applications to cluster algebras, J. Amer. Math. Soc. 23 (2010), 749--790.

\bibitem{Du1} A. Dugas, Periodic algebras and self-injective algebras of finite type, J. Pure 
Appl. Algebra 214 (2010), 990--1000.

\bibitem{Du2} A. Dugas, A construction of derived equivalent pairs of symmetric algebras, Proc. 
Amer. Math. Soc. 143 (2015) 2281--2300. 

\bibitem{E} K. Erdmann, Blocks of Tame Representation Type and Related Algebras, Lecture Notes 
Math., vol. 1428, Springer-Verlag, Berlin-Heidelberg, 1990.

\bibitem{ES1} K. Erdmann, A. Skowro\'nski, Periodic algebras, in: Trends in Representation Theory of Algebras and Related Topics, in: Eur. Math. Soc. Congress Reports, Europen Math. Soc., Z\"urich, 2008, pp. 201--251.

\bibitem{ES2} K. Erdmann, A. Skowro\'nski, Weighted surface algebras, J. Algebra 505 (2018), 490--558.

\bibitem{ES3} K. Erdmann, A. Skowro\'nski, Higher tetrahedral algebras, Algebr. Represent. Theory 22 
(2019), 387--406.

\bibitem{ES4} K. Erdmann, A. Skowro\'nski, Algebras of generalized quaternion type, Adv. Math. 349 
(2019), 1036--1116.

\bibitem{ES5} K. Erdmann, A. Skowro\'nski, Weighted surface algebras: General version, J. Algebra 
544 (2020), 170--227.

\bibitem{ES6} K. Erdmann, A. Skowro\'nski, Higher spherical algebras, Arch. Math. 114 (2020), 25--39. 

\bibitem{ES9} K. Erdmann, A. Skowro\'nski, Weighted surface algebras: General version. Corrigendum, 
J. Algebra 569 (2021), 875--889. 

\bibitem{FST1} A. Felikson, M. Shapiro, P. Tumarkin, Skew symmetric cluster algebras of finite 
mutation type, J. Eur. Math. Soc. 14 (2012), 1135--1180.

\bibitem{FG} V. Fock, A. Gontcharov, Moduli spaces of local systems in Teichm\"uller theory, Publ. Math. 
Inst. Hautes \'Etudes Sci. 103 (2006), 1--211. 

\bibitem{FST2} S. Fomin, M. Shapiro, D. Thurston, Cluster algebras and triangulated surfaces. Part I. 
Cluster complexes, Acta Math. 201 (2008), 83--146.

\bibitem{GLFS} C. Geiss, D. Labardini-Fragoso, J. Schr\"oer, The representation type of Jacobian 
algebras, Adv. Math. 290 (2016), 364--452.

\bibitem{GSV} M. Gekhtman, M. Shapiro, A. Vainshtein, Cluster algebras and Poisson geometry, 
Mosc. Math. J. 3 (2003), 899--934.

\bibitem{Ha1} D. Happel, Triangulated Categories in the Representation Theory of 
Finite-Dimensional Algebras, London Math. Soc. Lecture Note Ser., vol. 119, Cambridge 
University Press, Cambridge, 1988.

\bibitem{Ha2} A. Hatcher, Algebraic Topology, Cambridge University Press, Cambridge, 2002. 

\bibitem{HSS} T. Holm, A. Skowro\'nski, A. Skowyrski, Virtual mutations of weighted surface algebras, 
J. Algebra 619 (2023), 822--859. 

\bibitem{Ho} T. Holm, Derived equivalence classification of algebras of dihedral, semidihedral, 
and quaterion type, J. Algebra 211 (1999) 159--205.

\bibitem{Ka} M. Kauer, Derived equivalence of graph algebras, in: Trends in the Representation 
Theory of Finite-Dimensional Algebras, in: Contemp. Math., vol. 229, Amer. Math. Soc., 
Providence, RI, 1998, pp. 201--213.  

\bibitem{KZ} H. Krause, G. Zwara, Stable equivalence and generic modules, Bull. London Math. 
Soc. 32 (2000) 615--618.

\bibitem{MS} R.J. Marsh, S. Schroll, The geometry of Brauer graph algebras and cluster 
mutations, J. Algebra 419 (2014), 141--166.

\bibitem{O} T. Okuyama, Some examples of derived equivalent blocks of finite groups, Preprint, 1998.

\bibitem{Ric1} J. Rickard, Morita theory for derived categories, J. London Math. Soc. 39 (1989), 436--456.

\bibitem{Ric2} J. Rickard, Derived categories and stable equivalence, J. Pure Appl. Algebra 61 (1989), 303--317. 

\bibitem{Ric3} J. Rickard, Derived equivalences as derived functors, J. London Math. Soc. 43 (1991), 37--48. 

\bibitem{S1} A. Skowro\'nski, Selfinjective algebras of polynomial growth, Math. Ann. 285 (1989), 177--199. 

\bibitem{SY} A. Skowro\'nski, K. Yamagata, Frobenius Algebras I. Basic Representation Theory, Eur. Math. 
Soc. Textbooks in Math., European Math. Soc., Z\"urich, 2011.

\bibitem{Sk} A. Skowyrski, Two tilts of higher spherical algebras, Algebr. Represent. Theory 25 (2021), 237--254.
\end{thebibliography}
\end{document}